\documentclass[pdflatex,sn-mathphys-num]{sn-jnl}

\usepackage{graphicx}
\usepackage{multirow}
\usepackage{amsmath,amssymb,amsfonts}
\usepackage{amsthm}
\usepackage{mathrsfs}
\usepackage[title]{appendix}
\usepackage{xcolor}
\usepackage{textcomp}
\usepackage{manyfoot}
\usepackage{booktabs}
\usepackage{listings}
\usepackage{etoolbox}
\usepackage{algorithm}
\usepackage{algorithmicx}
\usepackage{algpseudocode}
\usepackage{url}
\usepackage{geometry}
\usepackage{hypcap}
\usepackage{hyperref}
\usepackage{breakurl}
\usepackage{natbib}

\newcommand{\Input}{\item[\textbf{Input:}]}
\newcommand{\Output}{\item[\textbf{Output:}]}

\theoremstyle{thmstyleone}

\newtheorem{Example}{Example}

\newcommand{\balpha}{\boldsymbol{\alpha}}
\newcommand{\bpi}{\boldsymbol{\pi}}
\newcommand{\btheta}{\boldsymbol{\theta}}

\theoremstyle{thmstyletwo}

\newtheorem{remark}{Remark}

\theoremstyle{remark}

\newtheorem{Lemma}{Lemma}

\theoremstyle{thmstylethree}

\begin{document}

\title[Level-dependent quasi-birth-and-death processes: Application to cost analysis of multi-server systems]{Level-dependent quasi-birth-and-death processes: Application to cost analysis of multi-server systems}

\author*[1]{\fnm{M. Abdullah} \sur{Khokhar}}\email{MuhammadAbdullah.Khokhar@utas.edu.au}

\author[1]{\fnm{Ma{\l}gorzata M.} \sur{O'Reilly}}\email{Malgorzata.OReilly@utas.edu.au}

\author[2]{\fnm{Richard} \sur{Turner}}\email{Richard.Turner@utas.edu.au}

\affil[1]{\orgdiv{School of Natural Sciences, Discipline of Mathematics}, \orgname{University of Tasmania}, \orgaddress{\country{Australia}}}

\affil[2]{\orgdiv{School of Medicine}, \orgname{University of Tasmania}, \orgaddress{\country{Australia}}}

\abstract{Analysing costs is crucial for optimising the operational efficiency and resource allocation in systems evolving under uncertainty. In this paper, we study the distribution of costs associated with the evolution of level-dependent quasi-birth-and-death (LD-QBD) processes, which are useful in modelling many multi-server systems. We derive analytical expressions for the Laplace–Stieltjes transforms (LSTs) of the distribution of total costs accumulated during the times the LD-QBD processes spend in a specified set of levels. We present algorithms for the numerical evaluation of these LSTs. We also give memory efficient versions of the algorithms and discuss their algorithmic complexity. To assess the robustness of the distribution of costs with respect to model parameters, we develop algorithms for the sensitivity analysis of the corresponding LSTs.

To illustrate the application potential of our results, we construct LD-QBD example models for a finite capacity multi-server queueing systems with admissions policies including redirection, preemptive transfer, and guard-channel threshold. The analysis is based on a large dataset obtained from a tertiary referral hospital in Australia. We compute the long-run performance measures, the distribution of time until some number of beds become available following congestion, and the distribution of the associated costs. We present valuable insights into how the system behaves under the various policies. We also perform the sensitivity analysis of the distribution of costs with respect to model parameters.}


\maketitle

\section{Introduction
}\label{sec:introduction}

Level-dependent quasi-birth-and-death processes (LD-QBDs) form an important class of continuous-time Markov chains and are widely used in the analysis of stochastic service systems, see Kharoufeh \cite{kharoufeh2011level} and Latouche and Ramaswami~\cite{latouche1999introduction}. These processes are characterised by a two-dimensional state space and a generator matrix. The state space consists of a level variable that describes the primary state of the system and a phase variable that captures supplementary information about the system. The block tri-diagonal structure of the generator matrix reflects the skip-free evolution of the level process.

A key aspect in the analysis of QBDs is the study of first-passage (hitting) times, which describe the time required for the process to reach a specified system level, as studied in Neuts~\cite{neuts1981matrix} and later also in Kim and Kim~\cite{Kim2013286}, G\'omez-Corral et al.~\cite{GomezCorral2026}, and Aksamit et al.~\cite{aksamit2024sensitivities}. These measures play a crucial role in analysing system performance, particularly in contexts where servers' utilisation or recovery periods are of primary importance. For instance, in finite-capacity service systems, it is of practical importance to determine the time required for the system to transition from highly congested states to lower occupancy levels. Such first-passage times provide insight into the transient behaviour of the system.

In service systems such as hospitals, a range of costs may accumulate during the system’s evolution, particularly when it operates near congestion, as discussed in Foley et al.~\cite{foley2011financial}, Hou et al.~\cite{Hou2022}, and Yadav et al.~\cite{Yadav2022}. These costs are often multifaceted and may include both direct and indirect components. Direct costs arise from the utilisation of clinical resources such as staff time, specialised equipment, and bed occupancy, which tend to increase with the number and complexity of patients in the system, see Drummond et al.~\cite{drummond2015methods}. Indirect costs may be associated with long waiting times, overcrowding, and diversion or transfer of patients to alternative facilities, which can adversely affect service quality and operational efficiency, see Kao et al.~\cite{Kao2015}.

These costs accumulate continuously during the evolution of the system, and may depend not only on how long the system remains in specified set of levels corresponding to congestion, but also on how intense that congestion is during the evolution. Two scenarios may have similar evolution times (e.g., the time until a certain number of beds become available in hospitals), yet one may involve much higher cost, resulting in higher financial burden. Therefore, studying the cost accumulated during the evolution of a system is equally important for the decision makers in addition to the corresponding time. Such information could help decision-makers analyse the performance of a system, assess the impact of congestion periods on costs, and compare alternative policies in terms of resource utilisation.

The objective of this paper is to study the distribution of total cost accumulated during the time a QBD process spends in some specified set of levels. We derive theoretical results to compute the Laplace–Stieltjes transforms (LSTs) of the distribution of cost. The LST provides a convenient way to characterise the distribution of cost, as it allows the total cost accumulation to be handled through tractable matrix-analytic expressions. This approach is particularly well suited to LD-QBD processes, where the structured, level-by-level evolution of the process enables the total cost to be analysed recursively, leading to efficient computational procedures. We construct the memory efficient algorithms to evaluate the LSTs of the distribution of costs, and discuss their algorithmic complexity. The LSTs can then be inverted using the numerical inversion methods in Den Iseger~\cite{DenIseger_2006}, summarised in Grant~\cite{Gusthesis}, to find the distribution of cost.

Moreover, we develop theoretical results and algorithms for the sensitivity of the LSTs of the distribution of cost with respect to model parameters. In our expressions, we assume that the level variable $X(t)$ is bounded from above by some finite $N$, but note that the results can be applied in the unbounded case as well, by applying standard truncation methods, e.g. see Phung-Duc et al.~\cite{Phung_Duc201046}. G\'omez-Corral and Lopez-Garcia~\cite{gomez2018perturbation} performed an initial sensitivity analysis of LD-QBDs. Aksamit et al.~\cite{aksamit2024sensitivities} built on their ideas and performed a sensitivity analysis of a wide range of key metrics.

To illustrate the application potential of our results, we construct examples of multi-server systems with different operational policies, using LD-QBD processes. In our numerical example, we then model a hospital system as a finite-capacity multi-server queueing system, which operates under different admission policies, such as redirection, preemptive transfer, and guard-channel threshold. The model parameters are based on a large dataset from a tertiary referral hospital in Australia, ensuring that the analysis reflects realistic system. As a first step, we evaluate some long-run performance measures to provide an understanding of the system behaviour under each policy. We then focus on the main objective of this paper and evaluate the distribution of the total cost accumulated until some required number of beds become available, together with a sensitivity analysis of these cost measures with respect to key parameters. We use the numerical results to provide insights into how different admission policies influence overall system behaviour and the key performance measures, including the time and cost accumulated during the system's evolution. Finally, we use the sensitivity results to further demonstrate that the cost measures respond in a stable way to changes in system parameters, providing confidence in the robustness of the proposed modelling approach and its potential usefulness for decision-making.

The following are our contributions in this paper. 
\begin{itemize}
    \item We derive theoretical results for the LSTs of the distribution of cost accumulated during the times a QBD process spends within specified set of levels.

    \item We construct memory efficient algorithms for the evaluation of these LSTs and discuss their algorithmic complexity.

    \item We also derive theoretical results and algorithms for the sensitivity analysis of these LSTs.

    \item We construct examples of LD-QBD models with the aim of illustrating their application potential to multi-server systems with various admission policies.  
    
    \item We then apply these models to a healthcare system in a numerical example based on a large dataset obtained from a tertiary referral hospital in Australia. We perform the sensitivity analysis of the distribution of cost and discuss the insights resulting from this analysis. 
\end{itemize}

The rest of the paper is structured as follows. In Section~\ref{sec:QBD-models}, we define a LD-QBD process and construct three examples of its potential application in multi-server systems. In Section~\ref{sec:KeyMetricsQBDs}, we briefly summarise relevant results from the literature for computing stationary distribution and first hitting times. In Section~\ref{sec:sojcosts}, we derive analytical expressions and algorithms to compute the LSTs of the distribution of costs. In Section~\ref{sec:EfficientAlgorithms}, we also give the memory efficient versions of the algorithms and discuss their algorithmic complexity. In Section~\ref{sec:SensAn} we derive the results for the sensitivity analysis of the distribution of costs. In Section~\ref{sec:application}, we illustrate the application of our results through the QBD example models in Section~\ref{sec:QBD-models}. We give our insights based on the numerical results and a comparison of the three models. This is followed by conclusions in Section~\ref{sec:Conclusion}.

\section{Level-dependent QBD (LD-QBD)}\label{sec:QBD-models}

A continuous-time Markov chain $\{(X(t),\varphi(t)):t\geq 0\}$ with level variable $X(t)$, phase variable $\varphi(t)$, and a two-dimensional state space $\mathcal{S}=
\{(n,i):n=0,1,2,\ldots,N;\ i=0,1,\ldots,K_n\}$ is called a Quasi-Birth-and-Death process (QBD) if the level variable $X(t)$ changes by at most one level. A QBD process is associated with an initial distribution vector 
\begin{eqnarray*}
    \balpha=[\balpha_n]_{n=0,1,2,\ldots,N},\  \balpha_n=[\alpha_{n,i}]_{i=1,\ldots K_n},
    \ 
    \alpha_{n,i}=\mathbb{P}(X(0)=n,\varphi(0)=i)
\end{eqnarray*}
and the transitions $(n,i)\to(n',j)$ are specified with the transition rates 
\begin{eqnarray*}
q_{(n,i)\to(n',j)}=
\left.\frac{d}{dt}
\mathbb{P}\big(X(t)=n',
\varphi(t)=j
\ | \ X(0)=n,\varphi(0)=i\big)\right\vert_{t=0}.
\end{eqnarray*}

The transition rates $q_{(n,i)\to(n^{'},j)}$ are 
recorded in a generator matrix ${\bf Q}$, which is a tri-diagonal matrix of block matrices ${\bf Q}={\bf Q}^{[n,n^{'}]}=[q_{(n,i)\to(n',j)}]_{i=0,1,2,...,K_n, j=0,1,2,...,K_{n^{'}}}$, such that

\begin{eqnarray}
	{\bf Q}
	=
	\begin{bmatrix}
		{\bf Q}^{[0,0]} & {\bf Q}^{[0,1]} & {\bf 0} & \cdots & \cdots & \cdots & \cdots & {\bf 0}\\
		{\bf Q}^{[1,0]} & {\bf Q}^{[1,1]} & {\bf Q}^{[1,2]} & {\bf 0} & \cdots & \cdots & \cdots & {\bf 0}\\
		{\bf 0} & {\bf Q}^{[2,1]} & {\bf Q}^{[2,2]} & {\bf Q}^{[2,3]} & \cdots &  \cdots & \cdots & {\bf 0}\\
		\vdots & \vdots & \vdots & \vdots & \cdots & \cdots & \cdots & \vdots\\
		{\bf 0} & {\bf 0} & {\bf 0} & {\bf 0} & \cdots & {\bf Q}^{[N-1,N-2]} & {\bf Q}^{[N-1,N-1]} & {\bf Q}^{[N-1,N]}\\
		{\bf 0} & {\bf 0} & {\bf 0} & {\bf 0} & \cdots & {\bf 0} & {\bf Q}^{[N,N-1]} & {\bf Q}^{[N,N]}
	\end{bmatrix}.
\label{blockmatrix}
\end{eqnarray}
Such a process is referred to as a Level-Dependent QBD (LD-QBD) if ${\bf Q}^{[n,n']}$ dependent on $n$ for $n=1,\ldots,N-1$.

Below, we describe three examples of LD-QBD processes relevant to multi-server systems, which we will later use in the numerical analysis in Section~\ref{sec:application}.

\begin{Example}\label{ex:QBDI_model}

Consider a two-class Erlang loss system  $(M/M/N/N)$ with arrival misclassification, that is, a queueing system with $N$ servers, no waiting area, and two types of customers (Type-A and Type-B) such that {\em upon arrival}, a Type-A customer is correctly classified as Type-A with probability $p_{AA}$, while a Type-B customer is misclassified as Type-A with probability $p_{BA}$. We assume that a true classification is known for all customers already in the system.

Customers of each type arrive according to independent Poisson processes with rates $\lambda_A$ and $\lambda_B$, and their service times are exponentially distributed with rates $\mu_A$ and $\mu_B$, respectively. An arriving customer is immediately assigned to an available server if there is one available. Here we assume that, if all servers are busy, the customer is redirected to an external service facility. This redirection policy models environments where queueing is prohibited due to capacity constraints, however alternative service is offered, e.g. to save lives.

\begin{figure}[h]
	\centering
	{\includegraphics[scale=0.2]{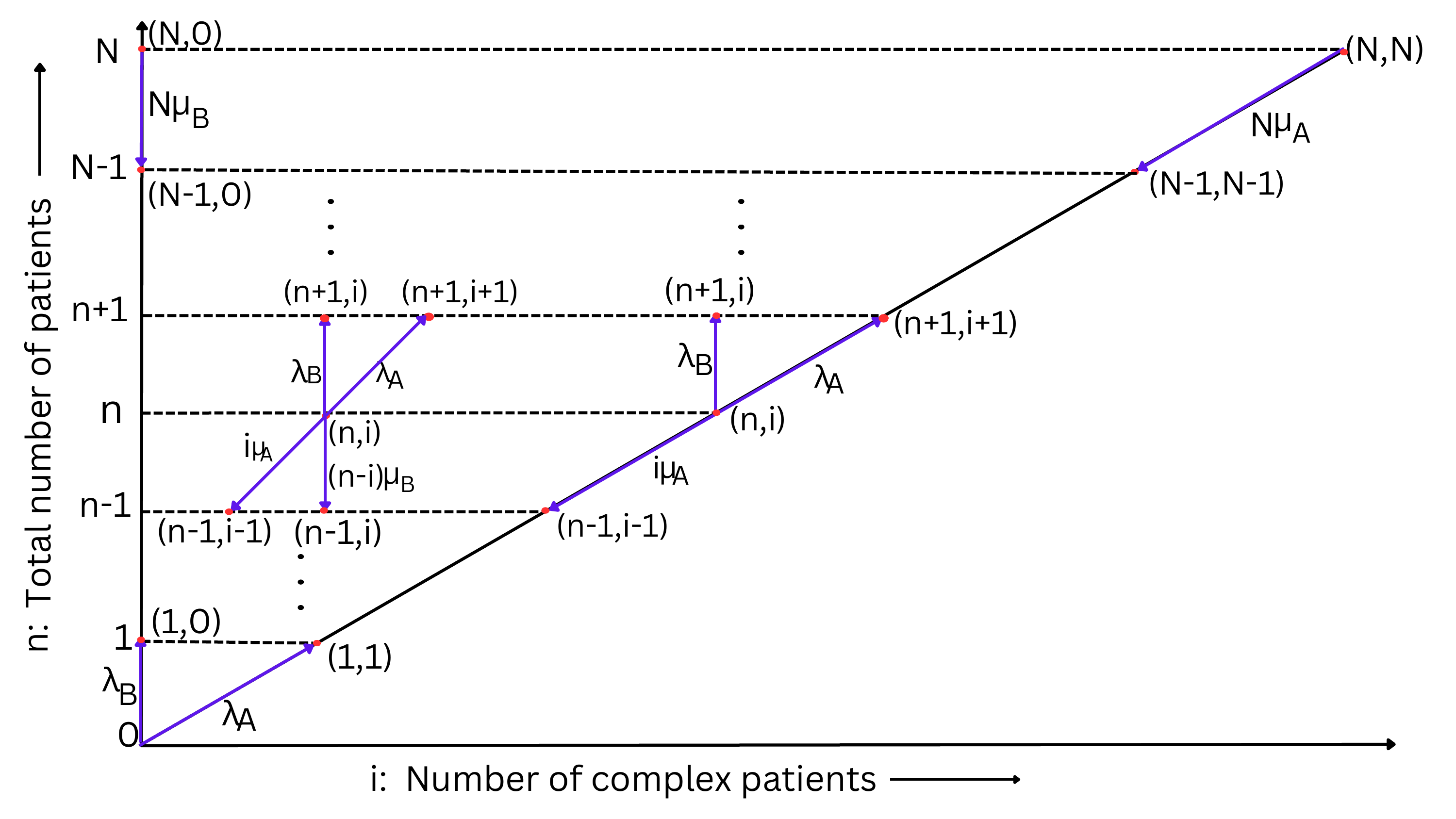}}
	\caption{Transition diagram of the model QBDI.}
	\label{TransitionQBD_I}
\end{figure}

We model the evolution of this system with a continuous-time LD-QBD process $\{(X(t),\varphi(t)):t\geq 0\}$ having a discrete state space given by $\mathcal{S}=
\{(n,i):n=0,1,2,\ldots,N;\ i=0,1,\ldots,n\}$, where the level variable $X(t)=n$ records the total number of customers in the system, and the phase variable $\varphi(t)=i$ records the total number of Type-A customers in the system, at time $t\geq 0$. The number of Type-B customers at time $t$ is therefore $X(t)-\varphi(t)=n-i$. We call this QBD process as QBD-I. It follows that the off-diagonals $q_{(n,i)\to(n',j)}$ within the blocks ${\bf Q}^{[n,n^{'}]}=[q_{(n,i)\to(n',j)}]$ of the generator ${\bf Q}=[{\bf Q}^{[n,n^{'}]}]$ are given by
\begin{eqnarray}\label{QBDI_model_parameters}
		q_{(n,i)\to(n',j)}
		&=&
		\left\{
		\begin{array}{lll}
			\lambda_A &\mbox{if }n^{'}=n+1,j=i+1\ & \&  \ n<N\\
			\lambda_B    &\mbox{if }n^{'}=n+1,j=i\ & \& \ n<N\\
			i\mu_A&\mbox{if }n^{'}=n-1,j=i-1\ & \& \ i>0, n>0\\
			(n-i)\mu_B&\mbox{if }n^{'}=n-1,j=i\ & \& \ i<n, n>0\\
			0 & \mbox{otherwise},
		\end{array}
		\right.\
	\end{eqnarray}
where $N$ is the upper boundary due to the finite size of the system, and the on-diagonals are 
\begin{eqnarray*}
q_{(n,i)\to(n,i)} = 
-\sum_{(n',j)\not=(n,i)  } q_{(n,i)\to(n',j)}.
\end{eqnarray*}
Figure~\ref{TransitionQBD_I}, illustrates the possible transitions between the various states of QBD-I.
\end{Example}

Throughout the examples that follow, we continue to refer to customers as Type-A and Type-B. Further, we distinguish between a customer's \emph{true} type (which corresponds to the service-time distribution) and the \emph{perceived} type assigned at arrival (which corresponds to an admission policy). Upon arrival, a customer's type is perceived according to two simple misclassification probabilities
\begin{eqnarray*}
\mathbb{P}(\text{perceived A}\mid \text{A})
&=&
p_{AA}, 
\qquad\quad\ \ 
\mathbb{P}(\text{perceived A}\mid \text{B}) = p_{BA},\\
\mathbb{P}(\text{perceived B}\mid \text{A})
&=&
1-p_{AA}, 
\qquad 
\mathbb{P}(\text{perceived B}\mid \text{B}) = 1-p_{BA},
\end{eqnarray*}
where $p_{AA}$ is the earlier defined probability of perceiving a Type-A customer as Type-A, $p_{BA}$ is the probability of perceiving a Type-B customer as Type-A, and the complementary probabilities $1-p_{AA}$ and $1-p_{BA}$ correspond to being perceived as Type-B.

\begin{Example}\label{ex:QBDII_model}

Next, consider a queueing system slightly similar to that described in Example~\ref{ex:QBDI_model} together with an admission policy defined as follows. 

First, we assume that if a {\em perceived} Type-A customer (Type-A customer correctly classified as Type-A, or Type-B customer misclassified as Type-A) arrives, and the system is at full capacity, then
\begin{itemize}  
\item a Type-B customer currently in the system is transferred to an external service facility to accommodate the incoming {\em perceived} Type-A customer, provided at least one Type-B customer is present;
\item  alternatively, if no Type-B customer is present, the incoming {\em perceived} Type-A customer is redirected to an external service facility. 
\end{itemize}
When the system is not full, all the arriving {\em perceived} Type-A customers are immediately admitted to available servers. Thus, the arriving {\em perceived} Type-A customers have preemptive priority over Type-B customers already in the system.

Second, we assume that all incoming {\em perceived} Type-B customers (Type-B customers correctly classified as Type-B, or Type-A customers misclassified as Type-B) are redirected when the system is full, regardless of the current composition of customers.

We model the evolution of such system using a continuous-time LD-QBD process $\{(X(t),\varphi(t)):t\geq 0\}$, with the same state space $\mathcal{S} = \{(n, i) : n = 0, 1, \ldots, N;\ i = 0, 1, \ldots, n\}$, where $X(t) = n$ is the total number of customers in the system and $\varphi(t) = i$ is the number of Type-A customers at time $t$. The number of Type-B customers is then $n - i$. We refer to this QBD process as QBD-II. The off-diagonal transition rates $q_{(n,i)\to(n',j)}$ within the generator blocks $\mathbf{Q}^{[n,n']} = [q_{(n,i)\to(n',j)}]$ are given by
\begin{eqnarray}\label{QBDII_model_parameters}
q_{(n,i)\to(n',j)}
&=&
\left\{
\begin{array}{lll}
\lambda_A &\mbox{if }n^{'}=n+1,j=i+1\ & \&  \ n<N\\
\lambda_B   &\mbox{if }n^{'}=n+1,j=i\ & \& \ n<N\\
i\mu_A&\mbox{if }n^{'}=n-1,j=i-1\ & \& \ i>0, n>0\\
(n-i)\mu_B&\mbox{if }n^{'}=n-1,j=i\ & \& \ i<n, n>0\\
p_{AA} \lambda_{A}  & \mbox{if }n^{'}=n,j=i+1\ & \&~  i<N, \ n=N\\
0 & \mbox{otherwise}.
\end{array}
\right.\
\end{eqnarray}
We note that when the system is full in state $(N, i)$ with $i < N$, then an arrival of a Type-A customer correctly classified as Type-A (or an arriving Type-B customer misclassified as Type-A) results in a transfer of a Type-B customer and a transition to state $(N, i+1)$. If the Type-A customer is misclassified as Type-B, or if a Type-B customer arrives and is correctly classified as Type-B, then the system remains in the same state due to redirection.
\end{Example}

\begin{Example}
\label{ex:QBDIII_model}
Finally, consider a generalisation of the previous examples to a two-class Erlang loss system with the admission policy that involves the following factors: (i) the {\em perceived} classification of the incoming customers, (ii) a capacity threshold $M_B$ for Type-B admissions, and (iii) a manager’s probabilistic decision when the system is highly occupied. 

First, we assume that when {\em perceived} Type‑A arrival occurs, then we apply the admission policy as described in Example~\ref{ex:QBDII_model}. 

Second, similar to the threshold‑based admission control policy studied by Yang et al.~\cite{yang2025optimal}, we introduce a guard‑channel threshold $M_B$ for Type‑B customers, where $M_B\in\{0,1,\ldots,N\}$ denotes a guard‑channel threshold that protects a portion $(N-M_B)$ of capacity for Type‑A arrivals. That is, we assume that when {\em perceived} Type‑B arrival occurs, then
\begin{itemize}
\item they are admitted provided that current occupancy is below $M_B$;

\item alternatively, if the system reaches the occupancy at least $M_B$, the {\em perceived} Type‑B arrival is admitted to the available servers with probability $p\in[0,1]$, or redirected to another facility with probability $(1-p)$.
\end{itemize}
When the system is at full capacity, all {\em perceived} Type-B arrivals are redirected.

We model the evolution of this system with a continuous-time LD-QBD process $\{(X(t),\varphi(t)):t\ge 0\}$ on the state space $
\mathcal{S}=\{(n,i): n=0,1,\ldots,N;\ i=0,1,\ldots,n\}$, where $X(t)=n$ records the total number of customers in the system at time $t$, and $\varphi(t)=i$ the number of Type-A customers. The number of Type-B customers is therefore $n-i$. We refer to this LD-QBD process as QBD-III. 
 
The off-diagonal transition rates $q_{(n,i)\to (n',j)}$ within the generator blocks $\mathbf{Q}^{[n,n']}=[q_{(n,i)\to(n',j)}]$ are then given by
\begin{eqnarray}\label{QBDIII_model_parameters}
q_{(n,i)\to(n',j)}
&=&
\left\{
\begin{array}{lll}
\lambda_A &\mbox{if }n^{'}=n+1,j=i+1\ & \&  \ n<N\\

\lambda_B   &\mbox{if }n^{'}=n+1,j=i\ & \& \ n<M_B\\

p\lambda_B   &\mbox{if }n^{'}=n+1,j=i\ & \& \ M_B\leq n<N\\

i\mu_A&\mbox{if }n^{'}=n-1,j=i-1\ & \& \ i>0, n>0\\
(n-i)\mu_B&\mbox{if }n^{'}=n-1,j=i\ & \& \ i<n, n>0\\
p_{AA} \lambda_{A}  & \mbox{if } n^{'}=n,j=i+1\ & \&~  i<N, \ n=N\\
0 & \mbox{otherwise}.
\end{array}
\right.\
\end{eqnarray}
We note that when Type-B arrival misclassified as Type-A occurs and the system is full, then there is no change in the state $(N,i)$ of the system, due to the admission policy. 
\end{Example}

\begin{remark}
The model QBD-III is a generalisation of the models QBD-I and QBD-II. Indeed, if we remove the guard-channel threshold $M_B$ and admission control probability $p$ for Type-B customers by setting $M_{B}=0$ and $p=1$, the QBD-III model reduces to QBD-II. Moreover, if in addition to taking $M_{B}=0$ and $p=1$, we also replace the transfer policy with redirection (which removes the transitions $(n,i)\to(n,i+1)$ for $n=N$ and $i<N$ from~\eqref{QBDII_model_parameters}), the resulting model reduces to QBD-I. Thus, $\text{QBD-I} \subset \text{QBD-II} \subset \text{QBD-III}$. Table~\ref{tab:comparison_QBDI_QBDII_QBDIII} shows a structural comparison of the three models.

\begin{table}[htbp]
\centering

\setlength{\tabcolsep}{3pt}

\begin{tabular}{l|c@{\hspace{0.5cm}}c@{\hspace{0.5cm}}c@{\hspace{0.5cm}}c@{\hspace{0.5cm}}c}
\hline

Model & Guard-channel & Type-B control & Misclassification & Redirection & Transfer \\

 & $(M_B)$ & $(p)$ & $(p_{AA},p_{BA})$ &  &  \\

\hline

QBD-I   & $\times$ & $\times$ & $\checkmark$ & $\checkmark$ & $\times$ \\

QBD-II  & $\times$ & $\times$ & $\checkmark$ & $\checkmark$ & $\checkmark$ \\

QBD-III & $\checkmark$ & $\checkmark$ & $\checkmark$ & $\checkmark$ & $\checkmark$ \\

\hline

\end{tabular}

\caption{Structural comparison of the models QBD-I, QBD-II, QBD-III.}
\label{tab:comparison_QBDI_QBDII_QBDIII}

\end{table}
\end{remark}

\section{Preliminaries}\label{sec:KeyMetricsQBDs}

In this section, we describe the key performance measures of LD-QBDs from the existing literature that may be helpful in the analysis of service systems.

\subsection{Stationary distribution}\label{StationaryDist}

Stationary distribution of a QBD process describes the proportion of time the process spends in each state as $t \to \infty$. We denote the stationary distribution vector by $\boldsymbol{\pi} = [\boldsymbol{\pi}_n]_{n = 0, 1, \ldots, N}$, such that vector $\boldsymbol{\pi}_n = [\pi_{n,i}]_{i = 0, 1, \ldots, K_n}$ records the value $\pi_{n,i}$, where $\pi_{n,i} = \lim_{t \to \infty} \mathbb{P}(X(t) = n, \varphi(t) = i)$ is the limiting probability of observing state $(n,i)$, interpreted as the long-run proportion of time spent in state $(n,i)$. To evaluate $\boldsymbol{\pi}$, we follow the approach of Aksamit et al.~\cite{aksamit2024sensitivities} and Grant~\cite{Gusthesis}, where $\boldsymbol{\pi}_N$ is evaluated first. The procedure is summarised in Algorithm~\ref{stationary_algorithm} in Appendix~\ref{AppendixA}.

\subsection{First hitting times}

First hitting time is the time that a process takes to visit a particular state for the first time during its evolution. This measure is crucial for understanding system responsiveness, as it quantifies how quickly the process reaches specified states.

We consider the Laplace-Stieltjes transform (LST) $ \widetilde{\bf G}^{n,n-k}(s)=[\widetilde G_{ij}^{n,n-k}(s)]_{i=1,\ldots,K_n;j=1,\ldots,K_{n-k}}$ of the distribution of time the process hits a lower level for the first time, such that $\widetilde G_{ij}^{n,n-k}(s)=\int_{0}^{\infty}e^{-st}g^{n,n-k}_{ij}(t)dt$ is the LST of the time to first hit a lower level $(n-k)$ and do so in phase $j=0,1,\ldots,K_{n-k}$, given start from level $n$ in phase $i=0,1,\ldots,K_n$, where for $1\leq k\leq n$, $g^{n,n-k}_{ij}(t)=\frac{\partial}{\partial t}
\mathbb{P}\big(\theta_{n-k}\leq t,\,\varphi(\theta_{n-k})=j \ | \ X(0)=n,\varphi(0)=i\big)$ is the corresponding probability density, and for any $n=0,1,\ldots,N$, the random variable $\theta_n=\inf \{t>0:X(t)=n\}$ is the first hitting time to level $n$. Denote ${\bf g}^{n,n-k}(t)=[g_{ij}^{n,n-k}(t)]_{i=1,\ldots,K_n;j=1,\ldots,K_{n-k}}$, $G_{ij}^{n,n-k}=\widetilde G_{ij}^{n,n-k}(0)$ and ${\bf G}^{n,n-k}=\widetilde{\bf G}^{n,n-k}(0)$. To evaluate $\widetilde {\bf G}^{n,n-k}(s)$, we follow the approach in Aksamit et al.~\cite{aksamit2024sensitivities}, summarised in Algorithm~\ref{Gs_algorithm} in Appendix~\ref{AppendixA}.

Next, we consider the LST matrix $\widetilde{\bf H}^{n,n+k}(s)=[\widetilde H_{ij}^{n,n+k}(s)]_{i=1,\ldots,K_n;j=1,\ldots,K_{n+k}}$ of the distribution of time to hit an upper level for the first time, such that
$$\widetilde H_{ij}^{n,n+k}(s)=\int_{0}^{\infty}e^{-st}h^{n,n+k}_{ij}(t)dt$$ is the LST of the time to first hit an upper level $(n+k)$ and do so in phase $j=0,1,\ldots,K_{n+k}$, given start from level $n$ in phase $i=0,1,\ldots,K_n$, where for $1\leq k\leq N-n$, $$h^{n,n+k}_{ij}(t)=\frac{\partial}{\partial t}
\mathbb{P}\big(\theta_{n+k}\leq t,\,\varphi(\theta_{n+k})=j \ | \ X(0)=n,\varphi(0)=i\big)$$ is the corresponding probability density. Denote 
${\bf h}^{n,n+k}(t)=[h_{ij}^{n,n+k}(t)]_{i=1,\ldots,K_n;j=1,\ldots,K_{n+k}}$, $H_{ij}^{n,n+k}=\widetilde H_{ij}^{n,n+k}(0)$ and ${\bf H}^{n,n+k}=\widetilde{\bf H}^{n,n+k}(0)$. To evaluate $\widetilde{\bf H}^{n,n+k}(s)$, we follow the approach in Aksamit et al.~\cite{aksamit2024sensitivities}, summarised in Algorithm~\ref{Hs_algorithm} in Appendix~\ref{AppendixA}.

\section{Costs accumulated during times spent within specified levels}\label{sec:sojcosts}

We build on the ideas in~\cite{bean2013stochastic,aksamit2024sensitivities,Gusthesis,samuelson2020construction}, and develop expressions for the Laplace-Stieltjes transform (LST) of the distribution of total cost accumulated during the time the process spends in a specified set of levels $\mathcal{A}$, during its evolution. 

Assume that costs (or rewards) accrue at a rate $c(n,i) \geq 0$ per unit time in state $(n,i)$, and let ${\bf C}_n = \mathrm{diag}(c(n,i))_i$ for $n = 0, 1, \ldots, N$. For a set of levels of interest (desirable or undesirable) $\mathcal{A}\subset\{0,1,\ldots,N\}$, $C_{\mathcal{A}}(t)=\int_{u=0}^t c(X(u),\varphi(u))\times I(X(u)\in\mathcal{A})du$ gives the total cost accumulated at time $t$, where $I(\cdot)$ is an indicator function. For $1\leq k\leq n$,
\begin{eqnarray}
c^{n,n-k}_{\mathcal{A};i,j}(z)&=&\frac{\partial}{\partial z}
\mathbb{P}\big(
C_{\mathcal{A}}(\theta_{n-k})
\leq z,
\,\varphi(\theta_{n-k})=j \ | \ X(0)=n,\varphi(0)=i\big)
\end{eqnarray}
is the probability density of the total cost accumulated at the time $\theta_{n-k}$ at which the process hits level $(n-k)$ for the first time and does so in phase $j=0,1,\ldots,K_{n-k}$, given start from level $n$ in phase $i=0,1,\ldots,K_n$. Denote, ${\bf c}^{n,n-k}_{\mathcal A}(z)=[c^{n,n-k}_{\mathcal A; i,j}(z)]_{i=1,\ldots,K_n;j=1,\ldots,K_{n-k}}$.

Further, we define probability matrices ${\bf C}^{n,n-k}_{\mathcal A}(z)=\int_{u=0}^z {\bf c}^{n,n-k}_{\mathcal A}(u)du=[C^{n,n-k}_{\mathcal A; i,j}(z)]_{i=1,\ldots,n;j=1,\ldots,n-k}$, where $C^{n,n-k}_{\mathcal A; i,j}(z)$ is the probability that the total cost accumulated at the time $\theta_{n-k}$ at which the process hits level $(n-k)$ for the first time and does so in phase $j=0,1,\ldots,K_{n-k}$, given start from level $n$ in phase $i=0,1,\ldots,K_n$, is smaller or equal to $z$. We define the LST matrix of the distribution of the total cost accumulated during time spent in specified levels in the set $\mathcal{A}$ during a sample path corresponding to $G_{ij}^{n,n-k}$ (discussed earlier) as $\widetilde {\bf C}_{\mathcal{A}}^{n,n-k}(s)=[\widetilde C_{\mathcal{A};ij}^{n,n-k}(s)]_{i=1,\ldots,K_n;j=1,\ldots,K_{n-k}},$
such that $\widetilde C_{\mathcal{A};ij}^{n,n-k}(s)=\int_{z=0}^{\infty}
e^{-sz}c^{n,n-k}_{\mathcal{A};i,j}(z)dz$ is the LST of the distribution of the total cost accumulated during the times spent within the set $\mathcal{A}$ as recorded at the moment the process first visits the lower level $(n-k)$ and does so in phase $j=0,1,\ldots,K_{n-k}$, given start from level $n$ in phase $i=0,1,\ldots,K_n$. Note that $\widetilde {\bf C}_{\mathcal{A}}^{n,n-k}(0)={\bf C}_{\mathcal{A}}^{n,n-k}$. To evaluate $\widetilde {\bf C}_{\mathcal{A}}^{n,n-k}(s)$, we apply Lemma~\ref{lem:CsLST} below. A simple implementation of this result is presented in Algorithm~\ref{Cs_algorithm}, and its two memory efficient alternatives are Algorithms~\ref{alg:singlepair_C_n_nminusk}~\&~\ref{alg:multiplepair_C_n_nminusk} presented later in Section~\ref{sec:EfficientAlgorithms}.

\begin{Lemma}
\label{lem:CsLST}
We have,
\begin{eqnarray}
\widetilde{\bf C}_{\mathcal{A}}^{n,n-k}(s)&=&
\widetilde{\bf C}_{\mathcal{A}}^{n,n-1}(s)
\widetilde{\bf C}_{\mathcal{A}}^{n-1,n-2}(s)
\times
\cdots
\times
\widetilde{\bf C}_{\mathcal{A}}^{n-k+1,n-k}(s),
\\
\widetilde{\bf C}_{\mathcal{A}}^{N,N-1}(s)&=&
-({\bf Q}^{[N,N]}-s{\bf C}_N\times I(N\in\mathcal{A}))^{-1}
{\bf Q}^{[N,N-1]},
\end{eqnarray} 
and for $n=N-1,\ldots,n-k+1$,
\begin{eqnarray}
\widetilde{\bf C}_{\mathcal{A}}^{n,n-1}(s)&=&
-({\bf Q}^{[n,n]}-s{\bf C}_n\times I(n\in\mathcal{A})
+{\bf Q}^{[n,n+1]}\widetilde{\bf C}_{\mathcal{A}}^{n+1,n}(s)
)^{-1}
{\bf Q}^{[n,n-1]}. 
\end{eqnarray}
\end{Lemma}
\noindent{\bf Proof:} Suppose the process starts in the state $(n,i)$, that is, $(X(0),\varphi(0)) = (n,i)$, and we observe its evolution for $h$ units of time, where $h>0$ is small. First, we show that ${\bf Q}^{[n,n]}-s{\bf C}_n\times I(n\in\mathcal{A})$ is the generator of the LST of the distribution of the total cost accumulated while the process remains at level $n\in\mathcal{A}$. For this, we consider only the following two cases, as all other events occur with probability $o(h)$. 

\begin{itemize}
  \item The process remains in state $(n,i)$ until at least $h$ units of time. That is, during the time interval $[0,h]$, the process remains at level $n$ with no transition in phase. Probability for this event to occur is $e^{-q_{(n,i)(n,i)}h}$, and the cost accumulated during this time is $c(n,i)h$. The LST of the cost accumulated would be $e^{-sc(n,i)h}$. Multiply this by the probability and store the values in a diagonal matrix $e^{({\bf Q}^{[n,n]}-s{\bf C}_n\times I(n\in\mathcal{A}))h}$ as the $(i,i)^{th}$ entry such that $[e^{({\bf Q}^{[n,n]}-s{\bf C}_n\times I(n\in\mathcal{A}))h}]_{ii} = e^{-(q_{(n,i)(n,i)}+sc(n,i))h}$.\\
  It follows that 
  \begin{eqnarray}
   \label{eqn:Gen_ii_der}
      \left. \frac{d}{dh}\left[e^{({\bf Q}^{[n,n]}-s{\bf C}_n\times I(n\in\mathcal{A}))h}\right]_{ii}\right\vert_{h=0} = -(q_{(n,i)(n,i)}+sc(n,i)).
  \end{eqnarray}
  \item Alternatively, the process makes a single transition to state $(n,j)$, for some $j \neq i$, at some time $u$, $ 0 < u \leq h$, and then remains in state $(n,j)$ until time reaches $h$. 
    \begin{itemize}
      \item   Probability that the process remains in state $(n,i)$ at least until time $u \in (0,h]$ is $e^{-q_{(n,i)(n,i)}u}$, and the associated cost is $c(n,i)u$. The LST for this cost would be $e^{-sc(n,i)u}$.
      \item  Transition rate when the process jumps from state $(n,i)$ to $(n,j)$ is $q_{(n,i)(n,j)}$.
      \item  Probability that the process remains in state $(n,j)$ until the time reaches $h$, that is, for the remaining time $h-u$, is $e^{-q_{(n,j)(n,j)}(h-u)}$. The cost accumulated during this time is given by $c(n,j)(h-u)$ with the LST $e^{-sc(n,j)(h-u)}$. 
  \end{itemize}

Consequently, the probability density that the process jumps to the state $(n,j)$ after spending $u$ units of time in state $(n,i)$, and does so in phase $j$ until the total evolution time of the process reaches $h$ is given by 
\begin{eqnarray}
\label{eqn:totalprob}
  e^{-q_{(n,i)(n,i)}u} q_{(n,i)(n,j)} e^{-q_{(n,j)(n,j)}(h-u)},
\end{eqnarray}
and the total cost accumulated at time $h$ is $c(n,i)u+c(n,j)(h-u),$ with the LST 
\begin{eqnarray}
\label{eqn:LSTtotalcost}
  e^{-s\left(c(n,i)u + c(n,j)(h-u)\right)}.
\end{eqnarray}
We multiply~\ref{eqn:totalprob} and~\ref{eqn:LSTtotalcost}, integrate from $u=0$ to $h$ and store the resulting value in the matrix $e^{({\bf Q}^{[n,n]}-s{\bf C}_n\times I(n\in\mathcal{A}))h}$ as the $(i,j)^{th}$ entry, with
\begin{equation*}
\begin{aligned}
\left[e^{({\bf Q}^{[n,n]} - s{\bf C}_n I(n\in\mathcal{A}))h}\right]_{ij}
&= \int_{0}^{h}
e^{-q_{(n,i)(n,i)}u} \, q_{(n,i)(n,j)} \, e^{-q_{(n,j)(n,j)}(h-u)} \\
&\quad \times e^{-s\left(c(n,i)u + c(n,j)(h-u)\right)} \, du.
\end{aligned}
\end{equation*}

It then follows
 \begin{eqnarray}
 \label{eqn:Gen_ij_der}
      \left. \frac{d}{dh}\left[e^{({\bf Q}^{[n,n]}-s{\bf C}_n\times I(n\in\mathcal{A}))h}\right]_{ij}\right\vert_{h=0} 
    &=& q_{(n,i)(n,j)}.
  \end{eqnarray}
\end{itemize}
It is clear from~\eqref{eqn:Gen_ii_der} and~\eqref{eqn:Gen_ij_der} that ${\bf Q}^{[n,n]}-s{\bf C}_n\times I(n\in\mathcal{A})$ is the generator matrix of the LST of the distribution of cost accumulated while the process remains at level $n$. Therefore, the corresponding LST matrix of distribution of cost accumulated at time $t$ is given by
$$e^{({\bf Q}^{[n,n]}-s{\bf C}_n\times I(n\in\mathcal{A}))t},$$
and the LST matrix of the total cost is given by
\begin{eqnarray}
\int_{t=0}^{\infty}e^{
\left(
{\bf Q}^{[n,n]}-s{\bf C}_n\times I(n\in\mathcal{A})
\right)
t
}dt
&=&
-({\bf Q}^{[n,n]}-s{\bf C}_n\times I(n\in\mathcal{A}))^{-1}.
\end{eqnarray}
To find the LST matrix $\widetilde{\bf C}_{\mathcal{A}}^{N,N-1}(s)$ of the distribution of total cost accumulated at the first visit to level $N-1$, given the process starts at level $N$, we multiply the LST matrix of the total cost accumulated while the process remains at level $N$ and the rate matrix of transition from level $N$ to the level $N-1$, that is,
\begin{eqnarray}
    \widetilde{\bf C}_{\mathcal{A}}^{N,N-1}(s) = -({\bf Q}^{[N,N]}-s{\bf C}_N\times I(N\in\mathcal{A}))^{-1} {\bf Q}^{[N,N-1]}.
\end{eqnarray}
Further, suppose the process starts at level $N-1$ in some phase $i$ and after some time it hits the level $N-2$. During this evolution, the process may undergo multiple transitions to level $N$ prior to ultimately arriving at level $N-2$. To determine the LST matrix $\widetilde{\bf C}_{\mathcal{A}}^{N-1,N-2}(s)$ of the distribution of cost accumulated from the process's initial state at level $N-1$ to its first arrival at level $N-2$, we employ the following methodology. 

Initially, we assume that the process commences at level $N-1$, spends some time at that level, transitions to level $N$, resides there for a certain period, and subsequently returns to level $N-1$. By conditioning this sequence of transitions to occur `m' times, we can then proceed to calculate the LST matrix of the total cost as follows.
\begin{itemize}
    \item The LST matrix of the distribution of the total cost accumulated while the process remains at level $N-1$ is $-({\bf Q}^{[N-1,N-1]}-s{\bf C}_{N-1}\times I(N-1\in\mathcal{A}))^{-1}$.
    \item Rate matrix for the process to jump to level $N$ is ${\bf Q}^{[N-1,N]}$.
    \item The LST matrix of the distribution of the total cost accumulated when the process spends some time at level $N$ and then returns to level $N-1$ is $\widetilde{\bf C}_{\mathcal{A}}^{N,N-1}(s)$.
    \item Rate matrix for the process to jump from level $N-1$ to the level $N-2$ is ${\bf Q}^{[N-1,N-2]}$. 
\end{itemize}
By conditioning of having $m$ transitions to level $N$, given the process starts at level $N-1$, before ultimately hitting level $N-2$ for the first time, the LST matrix of the total cost accumulated is given by
\begin{equation*}
\begin{aligned}
\widetilde{\bf C}_{\mathcal{A}}^{N-1,N-2}(s) 
&= -({\bf Q}^{[N-1,N-1]} - s{\bf C}_{N-1} I(N-1\in\mathcal{A}))^{-1} \\
&\quad \times 
\sum_{m=0}^{\infty} 
\Big(
{\bf Q}^{[N-1,N]} 
\widetilde{\bf C}_{\mathcal{A}}^{N,N-1}(s) \\
&\quad \times 
\big(-({\bf Q}^{[N-1,N-1]} - s{\bf C}_{N-1} I(N-1\in\mathcal{A}))^{-1}\big)
\Big)^{m}
{\bf Q}^{[N-1,N-2]} \\
&= -({\bf Q}^{[N-1,N-1]} - s{\bf C}_{N-1} I(N-1\in\mathcal{A}))^{-1} \\
&\quad \times 
\Big(
I - {\bf Q}^{[N-1,N]} \widetilde{\bf C}_{\mathcal{A}}^{N,N-1}(s) \\
&\quad \times 
\big(-({\bf Q}^{[N-1,N-1]} - s{\bf C}_{N-1} I(N-1\in\mathcal{A}))^{-1}\big)
\Big)^{-1}  
{\bf Q}^{[N-1,N-2]} \\
&= -\Big[
\Big(
I - {\bf Q}^{[N-1,N]} \widetilde{\bf C}_{\mathcal{A}}^{N,N-1}(s) \\
&\quad \times 
\big(-({\bf Q}^{[N-1,N-1]} - s{\bf C}_{N-1} I(N-1\in\mathcal{A}))^{-1}\big)
\Big) \\
&\quad \times 
({\bf Q}^{[N-1,N-1]} - s{\bf C}_{N-1} I(N-1\in\mathcal{A}))
\Big]^{-1}
{\bf Q}^{[N-1,N-2]} \\
&= -\left(
{\bf Q}^{[N-1,N-1]}
- s{\bf C}_{N-1} I(N-1\in\mathcal{A})
+ {\bf Q}^{[N-1,N]} \widetilde{\bf C}_{\mathcal{A}}^{N,N-1}(s)
\right)^{-1} \\
&\quad \times {\bf Q}^{[N-1,N-2]}.
\end{aligned}
\end{equation*}

By similar arguments, we can show that for $n=N-1,\ldots,n-k+1$, the LST matrix of the distribution of total cost is given by 
\begin{eqnarray}
\widetilde{\bf C}_{\mathcal{A}}^{n,n-1}(s)
&=&
-({\bf Q}^{[n,n]}-s{\bf C}_n\times I(n\in\mathcal{A})
+{\bf Q}^{[n,n+1]}\widetilde{\bf C}_{\mathcal{A}}^{n+1,n}(s)
)^{-1}
{\bf Q}^{[n,n-1]}.
\end{eqnarray}
Finally, the overall cost for the process to reach level $n-k$ for the first time, given the process starts at level $n$, is determined by adding up the costs for each step along the way. This includes the cost to go from level $n$ to level $n-1$, then from level $n-1$ to level $n-2$, and so on, until reaching level $n-k$ from level $n-k+1$. Due to the Markov property of the process, it follows that the LST matrix $\widetilde{\bf C}_{\mathcal{A}}^{n,n-k}(s)$ of the distribution of total cost is given by 
\begin{eqnarray}
    \widetilde{\bf C}_{\mathcal{A}}^{n,n-k}(s)=
\widetilde{\bf C}_{\mathcal{A}}^{n,n-1}(s)
\widetilde{\bf C}_{\mathcal{A}}^{n-1,n-2}(s)
\times
\cdots
\times
\widetilde{\bf C}_{\mathcal{A}}^{n-k+1,n-k}(s).
\end{eqnarray}
This completes the proof. \rule{9pt}{9pt}

\begin{algorithm}
\caption{Evaluate $\widetilde{\bf C}_{\mathcal{A}}^{n,n-k}(s)$, $n=1,2,\ldots,N$}
\label{Cs_algorithm}
\begin{algorithmic}[1]

\Input $\mathbf{Q}, n, k$

\State Compute and store $\widetilde{\bf C}_{\mathcal{A}}^{N,N-1}(s)$ using
\begin{equation}
\widetilde{\bf C}_{\mathcal{A}}^{N,N-1}(s)
= -({\bf Q}^{[N,N]} - s{\bf C}_N I(N\in\mathcal{A}))^{-1}
{\bf Q}^{[N,N-1]}.
\label{eq_C_N_N-1}
\end{equation}

\State For $n = N-1,\ldots,n-k+1$, compute and store $\widetilde{\bf C}_{\mathcal{A}}^{n,n-1}(s)$ using
\begin{equation}
\begin{aligned}
\widetilde{\bf C}_{\mathcal{A}}^{n,n-1}(s)
&= -({\bf Q}^{[n,n]} - s{\bf C}_n I(n\in\mathcal{A})
+ {\bf Q}^{[n,n+1]}\widetilde{\bf C}_{\mathcal{A}}^{n+1,n}(s))^{-1}
{\bf Q}^{[n,n-1]}.
\end{aligned}
\label{eq:Cnnminus1}
\end{equation}

\State Compute $\widetilde{\bf C}_{\mathcal{A}}^{n,n-k}(s)$ using
\begin{equation}
\widetilde{\bf C}_{\mathcal{A}}^{n,n-k}(s)
= \widetilde{\bf C}_{\mathcal{A}}^{n,n-1}(s)
\widetilde{\bf C}_{\mathcal{A}}^{n-1,n-2}(s)
\cdots
\widetilde{\bf C}_{\mathcal{A}}^{n-k+1,n-k}(s).
\label{eq:Cnk}
\end{equation}

\end{algorithmic}
\end{algorithm}

Similarly, for $1\leq k\leq N-n$, we define
\begin{eqnarray}
c^{n,n+k}_{\mathcal{A};i,j}(z)&=&\frac{\partial}{\partial z}
\mathbb{P}\big(
C_{\mathcal{A}}(\theta_{n+k})
\leq z,
\,\varphi(\theta_{n+k})=j \ | \ X(0)=n,\varphi(0)=i\big)
,\\
\widetilde C_{\mathcal{A};ij}^{n,n+k}(s)&=&\int_{z=0}^{\infty}
e^{-sz}c^{n,n+k}_{\mathcal{A};i,j}(z)dz
,\\
\widetilde{\bf C}_{\mathcal{A}}^{n,n+k}(s)&=&[\widetilde C_{\mathcal{A};ij}^{n,n+k}(s)]_{i=1,\ldots,n;j=1,\ldots,n+k},
\end{eqnarray}
where $c^{n,n+k}_{\mathcal{A};i,j}(z)$ is the probability density of the total cost accumulated at the time $\theta_{n+k}$ at which the process hits level $(n+k)$ for the first time and does so in phase $j=0,1,\ldots,K_{n+k}$, given start from level $n$ in phase $i=0,1,\ldots,K_n$, and $\widetilde C_{\mathcal{A};ij}^{n,n+k}(s)$ is the LST of the distribution of the total cost accumulated during the times spent within the set $\mathcal{A}$, corresponding to $H_{ij}^{n,n+k}$. Denote ${\bf c}^{n,n+k}_{\mathcal A}(z)=[c^{n,n+k}_{\mathcal A; i,j}(z)]_{i=1,\ldots,K_n;j=1,\ldots,K_{n+k}}$. To find $\widetilde{\bf C}_{\mathcal{A}}^{n,n+k}(s)$, apply Lemma~\ref{lem:CsLST2} below, or its simple implementation in Algorithm~\ref{Cs_algorithm2}. Two memory efficient alternatives are Algorithms~\ref{alg:singlepair_C_n_nplusk}~\&~\ref{alg:multiplepair_C_n_nplusk} presented later in Section~\ref{sec:EfficientAlgorithms}.

\begin{Lemma}
\label{lem:CsLST2}
We have,
\begin{eqnarray}
\widetilde{\bf C}_{\mathcal{A}}^{n,n+k}(s)&=&
\widetilde{\bf C}_{\mathcal{A}}^{n,n+1}(s)
\widetilde{\bf C}_{\mathcal{A}}^{n+1,n+2}(s)
\times
\cdots
\times
\widetilde{\bf C}_{\mathcal{A}}^{n+k-1,n+k}(s),
\end{eqnarray}
where
\begin{eqnarray}
\widetilde{\bf C}_{\mathcal{A}}^{0,1}(s)&=&
-({\bf Q}^{[0,0]}-s{\bf C}_0\times I(0\in\mathcal{A}))^{-1}
{\bf Q}^{[0,1]},
\end{eqnarray} 
and for $n=1,\ldots,n+k-1$,
\begin{eqnarray}
\widetilde{\bf C}_{\mathcal{A}}^{n,n+1}(s)&=&
-({\bf Q}^{[n,n]}-s{\bf C}_n\times I(n\in\mathcal{A})
+{\bf Q}^{[n,n-1]}\widetilde{\bf C}_{\mathcal{A}}^{n-1,n}(s)
)^{-1}
{\bf Q}^{[n,n+1]}.
\end{eqnarray}
\end{Lemma}
\noindent{\bf Proof:} The result follows by arguments analogous to Lemma~\ref{lem:CsLST}. \rule{9pt}{9pt}

\begin{algorithm}
\caption{Evaluate $\widetilde{\bf C}_{\mathcal{A}}^{n,n+k}(s)$, $n=0,1,\ldots,N-1$}
\label{Cs_algorithm2}
\begin{algorithmic}[1]

\Input $\mathbf{Q}, n, k$

\State Compute $\widetilde{\bf C}_{\mathcal{A}}^{0,1}(s)$ using
\begin{equation}
\widetilde{\bf C}_{\mathcal{A}}^{0,1}(s)
= -({\bf Q}^{[0,0]} - s{\bf C}_0 I(0\in\mathcal{A}))^{-1}
{\bf Q}^{[0,1]}.
\label{eq_C_0_1}
\end{equation}

\State For $n=1,\ldots,n+k-1$, compute $\widetilde{\bf C}_{\mathcal{A}}^{n,n+1}(s)$ using
\begin{equation}
\begin{aligned}
\widetilde{\bf C}_{\mathcal{A}}^{n,n+1}(s)
&= -({\bf Q}^{[n,n]} - s{\bf C}_n I(n\in\mathcal{A})
+ {\bf Q}^{[n,n-1]}\widetilde{\bf C}_{\mathcal{A}}^{n-1,n}(s))^{-1}
{\bf Q}^{[n,n+1]}.
\end{aligned}
\label{eq:Cnnplus1}
\end{equation}

\State Compute $\widetilde{\bf C}_{\mathcal{A}}^{n,n+k}(s)$ using
\begin{equation}
\widetilde{\bf C}_{\mathcal{A}}^{n,n+k}(s)
= \widetilde{\bf C}_{\mathcal{A}}^{n,n+1}(s)
\widetilde{\bf C}_{\mathcal{A}}^{n+1,n+2}(s)
\cdots
\widetilde{\bf C}_{\mathcal{A}}^{n+k-1,n+k}(s).
\label{eq:Cnk_plus}
\end{equation}

\end{algorithmic}
\end{algorithm}

\section{Efficient Algorithms~\ref{alg:singlepair_C_n_nminusk}~\&~\ref{alg:multiplepair_C_n_nminusk} and~\ref{alg:singlepair_C_n_nplusk}~\&~\ref{alg:multiplepair_C_n_nplusk}}
\label{sec:EfficientAlgorithms}

The computations in Algorithms~\ref{Cs_algorithm}~\&~\ref{Cs_algorithm2} require storing the transition rate block matrices ${\bf Q}^{[n,n']}$. To store a single transition rate block matrix ${\bf Q}^{[n,n']}$, we require a memory of $16 K_n K_{n'}$ bytes. Altogether, storing all blocks of ${\bf Q}$ requires $\sum_{n=0}^{N} \sum_{n'=n, n \pm 1} 16 K_n K_{n'}$ bytes. However, the matrix ${\bf Q}$ is typically very sparse, usually containing only $d$, where $3 \leq d \leq 6$, non-zero elements per row (Baumann and Sandmann~\cite{Baumann2012413}). An efficient approach is to store the matrices ${\bf Q}^{[n,n']}$ in sparse format, where only the positions of non-zero entries and their corresponding values are recorded. Each non-zero value requires $16$ bytes to store ($4$ bytes for each of the row and column indices and $8$ bytes for each non-zero value). In this way, each matrix ${\bf Q}^{[n,n']}$ would require a maximum memory of $16 \cdot d \cdot K_n$ bytes.

To compute $\widetilde{\bf C}_{\mathcal{A}}^{n,n-k}(s)$ using Algorithm~\ref{Cs_algorithm}, we need to store at least the following matrices $\widetilde{\bf C}_{\mathcal{A}}^{n,n-1}(s)$, $\widetilde{\bf C}_{\mathcal{A}}^{n-1,n-2}(s)$, $\dots$, $\widetilde{\bf C}_{\mathcal{A}}^{n-k+1,n-k}(s)$. These intermediate matrices are usually not sparse and are stored in dense format. Therefore, to store them in dense format we require $8 \sum_{\nu=0}^{k-1} K_{\,n-\nu}\,K_{\,n-\nu-1}$ bytes. This memory requirement applies to a single evaluation of $\widetilde{\bf C}_{\mathcal{A}}^{n,n-k}(s)$. However, in practice, the analysis often requires computing $\widetilde{\bf C}_{\mathcal{A}}^{n,n-k}(s)$ for multiple pairs $(n,n-k)$. If all intermediate matrices are retained for reuse, the memory footprint grows proportionally with the number of evaluations, which can quickly become infeasible for large $N$ or when many repetitions are needed. This motivates the development of memory-efficient implementations that minimize storage by using incremental updates or by recomputing matrices on demand.

Therefore, we give two memory-efficient alternatives to Algorithm~\ref{Cs_algorithm}, referred to as Algorithm~\ref{alg:singlepair_C_n_nminusk} and Algorithm~\ref{alg:multiplepair_C_n_nminusk}. Algorithm~\ref{alg:singlepair_C_n_nminusk} is suitable when the computation of $\widetilde{\bf C}_{\mathcal{A}}^{n,n-k}(s)$ is needed for a single pair $(n,n-k)$ only. This algorithm requires less computing time than Algorithm~\ref{Cs_algorithm}, as it avoids storing all intermediate matrices by computing the product incrementally and discarding each matrix after use. 

Algorithm~\ref{alg:multiplepair_C_n_nminusk}, while applicable for computing $\widetilde{\bf C}_{\mathcal{A}}^{n,n-k}(s)$ for a single pair $(n,n-k)$, is particularly suitable for scenarios where computations for multiple pairs $(n_j,n_j-k_j)$ are required, for $j = 1, 2, \ldots, \xi$. Instead of storing all intermediate matrices, the algorithm partitions the levels from $N$ down to $r$, where $r$ is the lowest level we are certain the computation will not go below, into $u$ blocks. Each block has size $M = \lfloor (N - r + 0.5)/u \rfloor$, with the last block of size $\leq M$.

For each block, only one checkpoint matrix $\widetilde{\bf C}_{\mathcal{A}}^{N-\widetilde{k}M,\,N-\widetilde{k}M-1}(s)$ (the top matrix) is stored for $\widetilde{k} = 0, \ldots, u-1$. To compute $\widetilde{\bf C}_{\mathcal{A}}^{n_j,n_j-k_j}(s)$, given $N-\widetilde{k}M \leq n_j < N-(\widetilde{k}+1)M$, the algorithm retrieves the corresponding checkpoint matrix $\widetilde{\bf C}_{\mathcal{A}}^{N-\widetilde{k}M,\,N-\widetilde{k}M-1}(s)$ and recomputes the required intermediate matrices while maintaining the required recursive product. The total memory required for $K$ computations is given by 
$8 \sum_{\widetilde{k}=0}^{u-1} K_{N-\widetilde{k}M} K_{N-\widetilde{k}M-1} + 8 \sum_{j=1}^{\xi} K_{n_j} K_{n_j-k_j}$. For a single pair $(n,n-k)$, the memory required is given by $8\big(K_n K_{n-k} + \sum_{\widetilde{k}=0}^{u-1} K_{N-\widetilde{k}M} K_{N-\widetilde{k}M-1} \big)$,
which is of course greater than when using Algorithm~\ref{alg:singlepair_C_n_nminusk}.

We apply a similar strategy for computing $\widetilde{\mathbf{C}}_{\mathcal{A}}^{n,n+k}(s)$ and give two memory-efficient alternatives to Algorithm~\ref{Cs_algorithm2}, referred to as Algorithm~\ref{alg:singlepair_C_n_nplusk}~\&~\ref{alg:multiplepair_C_n_nplusk}. Algorithm~\ref{alg:singlepair_C_n_nplusk} is suitable for a single pair $(n,n+k)$ computation, whereas Algorithm~\ref{alg:multiplepair_C_n_nplusk} is designed for memory-efficient computation across multiple pairs.

\begin{algorithm}[H]
\caption{Evaluate $\widetilde{\bf C}_{\mathcal{A}}^{n,n-k}(s)$ for a single pair $(n,n-k)$}
\label{alg:singlepair_C_n_nminusk}
\begin{algorithmic}[1]

\Input ${\bf Q}, n, k$
\Output $\widetilde{\bf C}_{\mathcal{A}}^{n,n-k}(s)$

\State Compute $\widetilde{\bf C}_{\mathcal{A}}^{N,N-1}(s)=
-({\bf Q}^{[N,N]}-s{\bf C}_N\times I(N\in\mathcal{A}))^{-1}
{\bf Q}^{[N,N-1]}$

\State Set ${\bf B}=\widetilde{\bf C}_{\mathcal{A}}^{N,N-1}(s)$.

\For{$i = N-1, \ldots, n$}
\State Compute $\widetilde{\bf C}_{\mathcal{A}}^{i,i-1}(s)=
-({\bf Q}^{[i,i]}-s{\bf C}_i\times I(n\in\mathcal{A})
+{\bf Q}^{[i,i+1]}{\bf B})^{-1}
{\bf Q}^{[i,i-1]}$

\State Update ${\bf B}=\widetilde{\bf C}_{\mathcal{A}}^{i,i-1}(s)$

\EndFor

\For{$i = n-1, \ldots, n-k+1$}
\State Compute $\widetilde{\bf C}_{\mathcal{A}}^{i,i-1}(s)=
-({\bf Q}^{[i,i]}-s{\bf C}_i\times I(n\in\mathcal{A})
+{\bf Q}^{[i,i+1]}\widetilde{\bf C}_{\mathcal{A}}^{i+1,i}(s)
)^{-1}
{\bf Q}^{[i,i-1]}$

\State Update ${\bf B}={\bf B} \times \widetilde{\bf C}_{\mathcal{A}}^{i,i-1}(s)$

\EndFor

\State \Return $\widetilde{\bf C}_{\mathcal{A}}^{n,n-k}(s) = {\bf B}$

\end{algorithmic}
\end{algorithm}

\begin{algorithm}[H]
\caption{Evaluate $\widetilde{\bf C}_{\mathcal{A}}^{n,n-k}(s)$ for multiple pairs $(n_j,n_j-k_j)$}
\label{alg:multiplepair_C_n_nminusk}
\begin{algorithmic}[1]

\Input ${\bf Q}, (n_j,k_j)_{j=1}^{\xi}$
\Input Threshold $r$, number of blocks $u$, block size $M$
\Output $\widetilde{\bf C}_{\mathcal{A}}^{n_j,n_j-k_j}(s)$ for all $j=1,\ldots,\xi$

\State Compute $\widetilde{\bf C}_{\mathcal{A}}^{N,N-1}(s)=
-({\bf Q}^{[N,N]}-s{\bf C}_N I(N\in\mathcal{A}))^{-1}
{\bf Q}^{[N,N-1]}$

\State Store ${\bf B}_0=\widetilde{\bf C}_{\mathcal{A}}^{N,N-1}(s)$

\For{$\widetilde{k} = 1, \ldots, u-1$}
    \For{$i = N-(\widetilde{k}-1)M-1, \ldots, N-\widetilde{k}M$}
        \State Compute $\widetilde{\bf C}_{\mathcal{A}}^{i,i-1}(s)=
        -({\bf Q}^{[i,i]}-s{\bf C}_i I(i\in\mathcal{A})
        +{\bf Q}^{[i,i+1]}\widetilde{\bf C}_{\mathcal{A}}^{i+1,i}(s))^{-1}
        {\bf Q}^{[i,i-1]}$
    \EndFor

    \State Store ${\bf B}_{\widetilde{k}}=
    \widetilde{\bf C}_{\mathcal{A}}^{N-\widetilde{k}M,\,N-\widetilde{k}M-1}(s)$
\EndFor

\For{$j = 1, \ldots, \xi$}

    \State Determine $\widetilde{k}(n_j)=\widetilde{k}$ such that
    $N-\widetilde{k}M \leq n_j < N-(\widetilde{k}+1)M$

    \State Set ${\bf B} = {\bf B}_{\widetilde{k}}$

    \For{$i = N-\widetilde{k}M-1, \ldots, n_j$}
        \State Compute $\widetilde{\bf C}_{\mathcal{A}}^{i,i-1}(s)=
        -({\bf Q}^{[i,i]}-s{\bf C}_i I(i\in\mathcal{A})
        +{\bf Q}^{[i,i+1]} {\bf B})^{-1}
        {\bf Q}^{[i,i-1]}$

        \State Update ${\bf B}=\widetilde{\bf C}_{\mathcal{A}}^{i,i-1}(s)$
    \EndFor

    \For{$i = n_j-1, \ldots, n_j-k_j+1$}
        \State Compute $\widetilde{\bf C}_{\mathcal{A}}^{i,i-1}(s)=
        -({\bf Q}^{[i,i]}-s{\bf C}_i I(i\in\mathcal{A})
        +{\bf Q}^{[i,i+1]}\widetilde{\bf C}_{\mathcal{A}}^{i+1,i}(s))^{-1}
        {\bf Q}^{[i,i-1]}$

        \State Update ${\bf B}={\bf B}\,\widetilde{\bf C}_{\mathcal{A}}^{i,i-1}(s)$
    \EndFor

    \State \Return $\widetilde{\bf C}_{\mathcal{A}}^{n_j,n_j-k_j}(s)={\bf B}$

\EndFor

\end{algorithmic}
\end{algorithm}

\begin{algorithm}[H]
\caption{Evaluate $\widetilde{\bf C}_{\mathcal{A}}^{n,n+k}(s)$ for a single pair $(n,n+k)$}
\label{alg:singlepair_C_n_nplusk}
\begin{algorithmic}[1]

\Input ${\bf Q}, n, k$
\Output $\widetilde{\bf C}_{\mathcal{A}}^{n,n+k}(s)$

\State Compute $\widetilde{\bf C}_{\mathcal{A}}^{0,1}(s)=
-({\bf Q}^{[0,0]}-s{\bf C}_0 I(0\in\mathcal{A}))^{-1}
{\bf Q}^{[0,1]}$

\State Set ${\bf B}=\widetilde{\bf C}_{\mathcal{A}}^{0,1}(s)$

\For{$i = 1, \ldots, n$}
\State Compute $\widetilde{\bf C}_{\mathcal{A}}^{i,i+1}(s)=
-({\bf Q}^{[i,i]}-s{\bf C}_i I(i\in\mathcal{A})
+{\bf Q}^{[i,i-1]}{\bf B})^{-1}
{\bf Q}^{[i,i+1]}$

\State Update ${\bf B}=\widetilde{\bf C}_{\mathcal{A}}^{i,i+1}(s)$
\EndFor

\For{$i = n+1, \ldots, n+k-1$}
\State Compute $\widetilde{\bf C}_{\mathcal{A}}^{i,i+1}(s)=
-({\bf Q}^{[i,i]}-s{\bf C}_i I(i\in\mathcal{A})
+{\bf Q}^{[i,i-1]}\widetilde{\bf C}_{\mathcal{A}}^{i-1,i}(s))^{-1}
{\bf Q}^{[i,i+1]}$

\State Update ${\bf B}={\bf B}\,\widetilde{\bf C}_{\mathcal{A}}^{i,i+1}(s)$
\EndFor

\State \Return $\widetilde{\bf C}_{\mathcal{A}}^{n,n+k}(s) = {\bf B}$

\end{algorithmic}
\end{algorithm}

\begin{algorithm}[H]
\caption{Evaluate $\widetilde{\bf C}_{\mathcal{A}}^{n,n+k}(s)$ for multiple pairs $(n_j,n_j+k_j)$}
\label{alg:multiplepair_C_n_nplusk}
\begin{algorithmic}[1]

\Input ${\bf Q}, (n_j,k_j)_{j=1}^{\xi}$
\Input Threshold $r$, number of blocks $u$, block size $M$
\Output $\widetilde{\bf C}_{\mathcal{A}}^{n_j,n_j+k_j}(s)$ for all $j=1,\ldots,\xi$

\State Compute $\widetilde{\bf C}_{\mathcal{A}}^{0,1}(s)=
-({\bf Q}^{[0,0]}-s{\bf C}_0 I(0\in\mathcal{A}))^{-1}
{\bf Q}^{[0,1]}$

\State Store ${\bf B}_0=\widetilde{\bf C}_{\mathcal{A}}^{0,1}(s)$

\For{$\widetilde{k} = 1, \ldots, u-1$}
    \For{$i = (\widetilde{k}-1)M+1, \ldots, \widetilde{k}M$}
        \State Compute $\widetilde{\bf C}_{\mathcal{A}}^{i,i+1}(s)=
        -({\bf Q}^{[i,i]}-s{\bf C}_i I(i\in\mathcal{A})
        +{\bf Q}^{[i,i-1]}\widetilde{\bf C}_{\mathcal{A}}^{i-1,i}(s))^{-1}
        {\bf Q}^{[i,i+1]}$
    \EndFor

    \State Store ${\bf B}_{\widetilde{k}}=
    \widetilde{\bf C}_{\mathcal{A}}^{\widetilde{k}M,\widetilde{k}M+1}(s)$
\EndFor

\For{$j = 1, \ldots, \xi$}

    \State Determine $\widetilde{k}(n_j)=\widetilde{k}$ such that$    (\widetilde{k}-1)M \leq n_j < \widetilde{k}M$

    \State Set ${\bf B} = {\bf B}_{\widetilde{k}}$

    \For{$i = (\widetilde{k}-1)M+1, \ldots, n_j$}
        \State Compute $\widetilde{\bf C}_{\mathcal{A}}^{i,i+1}(s)=
        -({\bf Q}^{[i,i]}-s{\bf C}_i I(i\in\mathcal{A})
        +{\bf Q}^{[i,i-1]}{\bf B})^{-1}
        {\bf Q}^{[i,i+1]}$

        \State Update ${\bf B}=\widetilde{\bf C}_{\mathcal{A}}^{i,i+1}(s)$
    \EndFor

    \For{$i = n_j+1, \ldots, n_j+k_j-1$}
        \State Compute $\widetilde{\bf C}_{\mathcal{A}}^{i,i+1}(s)=
        -({\bf Q}^{[i,i]}-s{\bf C}_i I(i\in\mathcal{A})
        +{\bf Q}^{[i,i-1]}\widetilde{\bf C}_{\mathcal{A}}^{i-1,i}(s))^{-1}
        {\bf Q}^{[i,i+1]}$

        \State Update ${\bf B}={\bf B}\,\widetilde{\bf C}_{\mathcal{A}}^{i,i+1}(s)$
    \EndFor

    \State \Return $\widetilde{\bf C}_{\mathcal{A}}^{n_j,n_j+k_j}(s) = {\bf B}$

\EndFor

\end{algorithmic}
\end{algorithm}

\subsection{Algorithmic complexity: Algorithm~\ref{alg:singlepair_C_n_nminusk}~\&~\ref{alg:singlepair_C_n_nplusk}}

The first step (computing $\widetilde{\mathbf{C}}_{\mathcal{A}}^{N,N-1}(s)$) in Algorithm~\ref{alg:singlepair_C_n_nminusk} involves an inversion of size $K_N \times K_N$ and multiplication with a matrix of size $K_N \times K_{N-1}$, giving a complexity $O(K_N^3)$ for inversion and $O(K_N^2 K_{N-1})$ for multiplication. For the loop $N-1$ $\to$ $n$, each iteration involves an inversion of a $K_i \times K_i$ matrix and two multiplications with $K_{i+1}$ and $K_{i-1}$, costing $O(K_i^3 + K_i^2 K_{i+1} + K_i^2 K_{i-1})$. Similarly, the second loop $n-1$ $\to$ $n-k+1$ adds an extra multiplication for updating ${\bf B}$, so each step costs $O(K_i^3 + K_i^2 K_{i+1} + 2K_i^2 K_{i-1})$. Therefore, the overall complexity is
\begin{eqnarray*}
O\Bigg(\sum_{i=n-k+1}^{N} \big(K_i^3 + K_i^2 K_{i+1} + 2K_i^2 K_{i-1}\big)\Bigg),
\end{eqnarray*}
If $K_i = m$ for all $i$, this simplifies to $O((N-n+k)m^3)$. In the worst-case scenario when $n=N$ and $n-k=0$, this reduces to $O(N m^3)$. Similarly, the overall complexity for Algorithm~\ref{alg:singlepair_C_n_nplusk} is
\begin{eqnarray*}
    O\Bigg(\sum_{i=0}^{n+k-1} \big(K_i^3 + K_i^2 K_{i-1} + 2K_i^2 K_{i+1}\big)\Bigg),
\end{eqnarray*}
which simplifies to $O((n+k)m^3)$ when $K_i = m$ for all $i$, and reduces to $O(N m^3)$ in the worst-case scenario.

\subsection{Algorithmic complexity: Algorithm~\ref{alg:multiplepair_C_n_nminusk}~\&~\ref{alg:multiplepair_C_n_nplusk}}

The initial steps in Algorithm~\ref{alg:multiplepair_C_n_nminusk} compute checkpoint matrices for $u$ blocks, each requiring inversions and multiplications similar to Algorithm~\ref{alg:singlepair_C_n_nminusk}, giving complexity $O\big(\sum_{i=N-(u-1)M}^{N} (K_i^3 + K_i^2 K_{i+1} + K_i^2 K_{i-1})\big)$. For each pair $(n_j, k_j)$, the Algorithm~\ref{alg:multiplepair_C_n_nminusk} costs $O\big(\sum_{i=n_j}^{N-\widetilde k(n_j) M -1} (K_i^3 + K_i^2 K_{i+1} + K_i^2 K_{i-1})\big)$ for the loop $N-\widetilde k M -1$ $\to$ $n_j$, and it costs $O\big(\sum_{i=n_j-k_j+1}^{n_j-1} (K_i^3 + K_i^2 K_{i+1} + 2K_i^2 K_{i-1})\big)$ for the loop $n_j-1$ $\to$ $n_j-k_j+1$. Therefore, for $\xi$ pairs, the overall complexity is
\begin{equation*}
\begin{aligned}
O\Bigg(&\sum_{i=N-(u-1)M}^{N}\left(K_i^3 + K_i^2 K_{i+1} + K_i^2 K_{i-1}\right) \\
&+ \sum_{j=1}^{\xi}\sum_{i=n_j-k_j+1}^{N-\widetilde{k}(n_j)M -1}
\left(K_i^3 + K_i^2 K_{i+1} + 2K_i^2 K_{i-1}\right)\Bigg),
\end{aligned}
\end{equation*}
which simplifies to $O\big(u M m^3 + \sum_{j=1}^{\xi}(N-\widetilde{k}(n_j)M - n_j + k_j)m^3\big)$ when $K_i = m$ for all $i$, and reduces to $O(u M m^3 + \xi N m^3)$ in the worst case scenario. Similarly, the overall complexity of Algorithm~\ref{alg:multiplepair_C_n_nplusk} is
\begin{equation*}
\begin{aligned}
O\Bigg(&\sum_{i=0}^{(u-1)M}\left(K_i^3 + K_i^2 K_{i-1} + K_i^2 K_{i+1}\right) \\
&+ \sum_{j=1}^{\xi}\sum_{i=(\widetilde{k}(n_j)-1)M+1}^{n_j+k_j-1}
\left(K_i^3 + K_i^2 K_{i-1} + 2K_i^2 K_{i+1}\right)\Bigg),
\end{aligned}
\end{equation*}
which simplifies to $O\big(u M m^3 + \sum_{j=1}^{\xi}(n_j+k_j-\widetilde{k}(n_j)M)m^3\big)$ when $K_i = m$ for all $i$, and reduces to $O(u M m^3 + \xi N m^3)$ in the worst case scenario.

\section{Sensitivity analysis}\label{sec:SensAn}

Aksamit et al.~\cite{aksamit2024sensitivities} developed algorithms for the sensitivities of the stationary distribution $\bpi_{n}(\btheta)$, first hitting times to lower levels ${\bf G}^{n,n-k}(\btheta)$, and first hitting times to higher levels ${\bf H}^{n,n+k}(\btheta)$, where $\btheta$ is the vector of model parameters.  Below, we extend these to develop algorithms for the sensitivities of the LSTs $\widetilde{\bf C}_{\mathcal{A}}^{n,n-k}(s,\btheta)$ and $\widetilde{\bf C}_{\mathcal{A}}^{n,n+k}(s,\btheta)$ of distribution of total cost accumulated during the times spent within the set $\mathcal{A}$ as recorded at the moment the process first visits the lower level $(n-k)$ or a higher level $(n+k)$ and does so in phase $j=0,1,\ldots,K_{n-k}$ or $j=0,1,\ldots,K_{n+k}$ respectively, given start from level $n$ in phase $i=0,1,\ldots,K_n$. To evaluate the sensitivity of $\widetilde{\bf C}_{\mathcal{A}}^{n,n-k}(s,\btheta)$, we apply Algorithm~\ref{alg:C_n_nminusk_der_mem}, which follows from the Lemma~\ref{lem:Cs_theta_LST_nnminuk} below.

\begin{Lemma}
\label{lem:Cs_theta_LST_nnminuk}
We have,
\begin{eqnarray}
\label{eq:C_n_nminusk_theta}
\frac{\partial}{\partial\btheta}\widetilde{\bf C}_{\mathcal{A}}^{n,n-k}(s;\btheta)&=&
\frac{\partial \widetilde{\bf C}_{\mathcal{A}}^{n,n-k+1}(s;\btheta)}{ \partial \btheta}
\times\left(
{\bf I}_k\otimes \widetilde{\bf C}_{\mathcal{A}}^{n-k+1,n-k}(s;\btheta)
\right)
\nonumber\\
&&
+
\widetilde{\bf C}_{\mathcal{A}}^{n,n-k+1}(s;\btheta)
\times
\frac{\partial \widetilde{\bf C}_{\mathcal{A}}^{n-k+1,n-k}(s;\btheta)}{\partial \btheta},
\end{eqnarray}
where
\begin{eqnarray} 
\frac{\partial }{\partial \btheta}\widetilde{\bf C}_{\mathcal{A}}^{N,N-1}(s;\btheta)
  &=&
\left({\bf Q}^{[N,N]}(\btheta)-s{\bf C}_N\times I(N\in\mathcal{A})\right)^{-1}
\times
\frac{\partial {\bf Q}^{[N,N]}(\btheta)}{\partial \btheta}
\nonumber\\
&&
\times
\left(
{\bf I}_k
\otimes \left({\bf Q}^{[N,N]}(\btheta)-s{\bf C}_N\times I(N\in\mathcal{A})\right)^{-1}
\right)
\nonumber\\
&&
\times\left(
{\bf I}_k\otimes {\bf Q}^{[N,N-1]}(\btheta)
\right)
\nonumber
\\
&&
-
\left({\bf Q}^{[N,N]}(\btheta)-s{\bf C}_N\times I(N\in\mathcal{A})\right)^{-1}
\times
\frac{\partial {\bf Q}^{[N,N-1]}(\btheta)}{\partial \btheta},
        \label{C_N_N-1_theta}
\end{eqnarray}
and for $n=N-1,\ldots,n-k+1$,
\begin{equation}
\begin{aligned}
\frac{\partial }{\partial \btheta}&\widetilde{\bf C}_{\mathcal{A}}^{n,n-1}(s;\btheta)\\
&= \Big({\bf Q}^{[n,n]}(\btheta)-s{\bf C}_n I(n\in\mathcal{A})
+{\bf Q}^{[n,n+1]}(\btheta)\widetilde{\bf C}_{\mathcal{A}}^{n+1,n}(s,\btheta)
\Big)^{-1} \\
&\times
\Bigg(
\frac{\partial {\bf Q}^{[n,n]}(\btheta)}{\partial \btheta}
+
\frac{\partial{\bf Q}^{[n,n+1]}(\btheta)}{\partial \btheta}
\left({\bf I}_k \otimes \widetilde{\bf C}_{\mathcal{A}}^{n+1,n}(s,\btheta)\right)
\\
&\hspace{2.2cm}
+ {\bf Q}^{[n,n+1]}(\btheta)
\frac{\partial \widetilde{\bf C}_{\mathcal{A}}^{n+1,n}(s,\btheta)}{\partial \btheta}
\Bigg) \\
&\times
\Bigg(
{\bf I}_k \otimes
\Big({\bf Q}^{[n,n]}(\btheta)-s{\bf C}_n I(n\in\mathcal{A})
+{\bf Q}^{[n,n+1]}(\btheta)\widetilde{\bf C}_{\mathcal{A}}^{n+1,n}(s,\btheta)
\Big)^{-1}
\Bigg)
\\
&\times
\left({\bf I}_k\otimes {\bf Q}^{[n,n-1]}(\btheta)\right)
\\
&\quad-
\Big({\bf Q}^{[n,n]}(\btheta)-s{\bf C}_n I(n\in\mathcal{A})
+{\bf Q}^{[n,n+1]}(\btheta)\widetilde{\bf C}_{\mathcal{A}}^{n+1,n}(s,\btheta)
\Big)^{-1}
\frac{\partial {\bf Q}^{[n,n-1]}(\btheta)}{\partial \btheta}.
\end{aligned}
\label{eq:C_n_nminus1_theta}
\end{equation}
\end{Lemma}
\noindent{\bf Proof:} First, we calculate the sensitivity of 
$\widetilde{\bf C}_{\mathcal{A}}^{N,N-1}(s;\btheta)$ using the results
\begin{align*}
\frac{\partial}{\partial\btheta}
\left(
{\bf A}(\btheta)\times {\bf B}( \btheta)
\right)
&=
\frac{\partial {\bf A}(\btheta)}{ \partial \btheta}
\times\left(
{\bf I}_k\otimes {\bf B}(\btheta)
\right)
+
{\bf A}(\btheta)
\times
\frac{\partial {\bf B}(\btheta)}{\partial \btheta}\\
\frac{\partial ({\bf A}(\btheta))^{-1}}{\partial \btheta}
&=
-({\bf A}(\btheta))^{-1}
\times
\frac{\partial {\bf A}(\btheta)}{\partial \btheta}
\times
\left(
{\bf I}_k
\otimes ({\bf A}(\btheta))^{-1}
\right),
\end{align*}
as follows.

\begin{equation*}
\begin{aligned}
\frac{\partial}{\partial \btheta}&
\widetilde{\bf C}_{\mathcal{A}}^{N,N-1}(s;\btheta)
\\
&=
-\frac{\partial}{\partial \btheta}
\Big(
({\bf Q}^{[N,N]}(\btheta)-s{\bf C}_N I(N\in\mathcal{A}))^{-1}
{\bf Q}^{[N,N-1]}(\btheta)
\Big)
\\
&=
-\frac{\partial ({\bf Q}^{[N,N]}(\btheta)-s{\bf C}_N I(N\in\mathcal{A}))^{-1}}{\partial \btheta}
\left({\bf I}_k\otimes {\bf Q}^{[N,N-1]}(\btheta)\right)
\\
&\quad -
({\bf Q}^{[N,N]}(\btheta)-s{\bf C}_N I(N\in\mathcal{A}))^{-1}
\frac{\partial {\bf Q}^{[N,N-1]}(\btheta)}{\partial \btheta}
\\
&=
({\bf Q}^{[N,N]}(\btheta)-s{\bf C}_N I(N\in\mathcal{A}))^{-1}
\frac{\partial ({\bf Q}^{[N,N]}(\btheta)-s{\bf C}_N I(N\in\mathcal{A}))}{\partial \btheta}
\\
&\quad \times
\left({\bf I}_k \otimes
({\bf Q}^{[N,N]}(\btheta)-s{\bf C}_N I(N\in\mathcal{A}))^{-1}\right)
\left({\bf I}_k\otimes {\bf Q}^{[N,N-1]}(\btheta)\right)
\\
&\quad -
({\bf Q}^{[N,N]}(\btheta)-s{\bf C}_N I(N\in\mathcal{A}))^{-1}
\frac{\partial {\bf Q}^{[N,N-1]}(\btheta)}{\partial \btheta}
\\
&=
({\bf Q}^{[N,N]}(\btheta)-s{\bf C}_N I(N\in\mathcal{A}))^{-1}
\frac{\partial {\bf Q}^{[N,N]}(\btheta)}{\partial \btheta}
\\
&\quad \times
\left({\bf I}_k \otimes
({\bf Q}^{[N,N]}(\btheta)-s{\bf C}_N I(N\in\mathcal{A}))^{-1}\right)
\left({\bf I}_k\otimes {\bf Q}^{[N,N-1]}(\btheta)\right)
\\
&\quad -
({\bf Q}^{[N,N]}(\btheta)-s{\bf C}_N I(N\in\mathcal{A}))^{-1}
\frac{\partial {\bf Q}^{[N,N-1]}(\btheta)}{\partial \btheta}.
\end{aligned}
\end{equation*}

Now, for $n=N-1,\ldots,1$, we have
\begin{equation*}
\begin{aligned}
\frac{\partial }{\partial \btheta}&
\widetilde{\bf C}_{\mathcal{A}}^{n,n-1}(s;\btheta)
\\
&=
-\frac{\partial }{\partial \btheta}
\Big({\bf Q}^{[n,n]}(\btheta)-s{\bf C}_n I(n\in\mathcal{A})
+{\bf Q}^{[n,n+1]}(\btheta)\widetilde{\bf C}_{\mathcal{A}}^{n+1,n}(s,\btheta)
\Big)^{-1}
{\bf Q}^{[n,n-1]}(\btheta)
\\
&=
-\frac{\partial \Big({\bf Q}^{[n,n]}(\btheta)-s{\bf C}_n I(n\in\mathcal{A})
+{\bf Q}^{[n,n+1]}(\btheta)\widetilde{\bf C}_{\mathcal{A}}^{n+1,n}(s,\btheta)
\Big)^{-1}}{\partial \btheta}
\\
&\quad \times
\left({\bf I}_k\otimes {\bf Q}^{[n,n-1]}(\btheta)\right)
\\
&\quad -
\left({\bf Q}^{[n,n]}(\btheta)-s{\bf C}_n I(n\in\mathcal{A})
+{\bf Q}^{[n,n+1]}(\btheta)\widetilde{\bf C}_{\mathcal{A}}^{n+1,n}(s,\btheta)
\right)^{-1}
\frac{\partial {\bf Q}^{[n,n-1]}(\btheta)}{\partial \btheta},
\end{aligned}
\end{equation*}
where
\begin{equation*}
\begin{aligned}
\frac{\partial }{\partial \btheta}&
\left({\bf Q}^{[n,n]}(\btheta)-s{\bf C}_n I(n\in\mathcal{A})
+{\bf Q}^{[n,n+1]}(\btheta)\widetilde{\bf C}_{\mathcal{A}}^{n+1,n}(s,\btheta)
\right)^{-1}
\\
&=
- \left({\bf Q}^{[n,n]}(\btheta)-s{\bf C}_n I(n\in\mathcal{A})
+{\bf Q}^{[n,n+1]}(\btheta)\widetilde{\bf C}_{\mathcal{A}}^{n+1,n}(s,\btheta)
\right)^{-1}
\\
&\quad \times
\frac{\partial \left({\bf Q}^{[n,n]}(\btheta)-s{\bf C}_n I(n\in\mathcal{A})
+{\bf Q}^{[n,n+1]}(\btheta)\widetilde{\bf C}_{\mathcal{A}}^{n+1,n}(s,\btheta)
\right)}{\partial \btheta}
\\
&\quad \times
\left({\bf I}_k \otimes
\left({\bf Q}^{[n,n]}(\btheta)-s{\bf C}_n I(n\in\mathcal{A})
+{\bf Q}^{[n,n+1]}(\btheta)\widetilde{\bf C}_{\mathcal{A}}^{n+1,n}(s,\btheta)
\right)^{-1}\right)
\\
&=
- \left({\bf Q}^{[n,n]}(\btheta)-s{\bf C}_n I(n\in\mathcal{A})
+{\bf Q}^{[n,n+1]}(\btheta)\widetilde{\bf C}_{\mathcal{A}}^{n+1,n}(s,\btheta)
\right)^{-1}
\\
&\quad \times
\Bigg(
\frac{\partial {\bf Q}^{[n,n]}(\btheta)}{\partial \btheta}
+
\frac{\partial\left({\bf Q}^{[n,n+1]}(\btheta)\widetilde{\bf C}_{\mathcal{A}}^{n+1,n}(s,\btheta)\right)}{\partial \btheta}
\Bigg)
\\
&\quad \times
\left({\bf I}_k \otimes
\left({\bf Q}^{[n,n]}(\btheta)-s{\bf C}_n I(n\in\mathcal{A})
+{\bf Q}^{[n,n+1]}(\btheta)\widetilde{\bf C}_{\mathcal{A}}^{n+1,n}(s,\btheta)
\right)^{-1}\right)
\\
&=
- \left({\bf Q}^{[n,n]}(\btheta)-s{\bf C}_n I(n\in\mathcal{A})
+{\bf Q}^{[n,n+1]}(\btheta)\widetilde{\bf C}_{\mathcal{A}}^{n+1,n}(s,\btheta)
\right)^{-1}
\\
&\quad \times
\Bigg(
\frac{\partial {\bf Q}^{[n,n]}(\btheta)}{\partial \btheta}
+
\frac{\partial{\bf Q}^{[n,n+1]}(\btheta)}{\partial \btheta}
\left({\bf I}_k \otimes \widetilde{\bf C}_{\mathcal{A}}^{n+1,n}(s,\btheta)\right)
\\
&\qquad\qquad
+
{\bf Q}^{[n,n+1]}(\btheta)
\frac{\partial \widetilde{\bf C}_{\mathcal{A}}^{n+1,n}(s,\btheta)}{\partial \btheta}
\Bigg)
\\
&\quad \times
\left({\bf I}_k \otimes
\left({\bf Q}^{[n,n]}(\btheta)-s{\bf C}_n I(n\in\mathcal{A})
+{\bf Q}^{[n,n+1]}(\btheta)\widetilde{\bf C}_{\mathcal{A}}^{n+1,n}(s,\btheta)
\right)^{-1}\right).
\end{aligned}
\end{equation*}

Finally, for $k\geq 2$, we have
\begin{equation*}
\begin{aligned}
\frac{\partial}{\partial\btheta}
\widetilde{\bf C}_{\mathcal{A}}^{n,n-k}(s;\btheta)
\\
&=
\frac{\partial}{\partial\btheta}
\Big(
\widetilde{\bf C}_{\mathcal{A}}^{n,n-1}(s;\btheta)
\widetilde{\bf C}_{\mathcal{A}}^{n-1,n-2}(s;\btheta)
\cdots
\widetilde{\bf C}_{\mathcal{A}}^{n-k+1,n-k}(s;\btheta)
\Big)
\\
&=
\frac{\partial}{\partial\btheta}
\Big(
\widetilde{\bf C}_{\mathcal{A}}^{n,n-k+1}(s;\btheta)
\widetilde{\bf C}_{\mathcal{A}}^{n-k+1,n-k}(s;\btheta)
\Big)
\\
&=
\frac{\partial \widetilde{\bf C}_{\mathcal{A}}^{n,n-k+1}(s;\btheta)}{\partial \btheta}
\left({\bf I}_k\otimes \widetilde{\bf C}_{\mathcal{A}}^{n-k+1,n-k}(s;\btheta)\right)
\\
&\quad +
\widetilde{\bf C}_{\mathcal{A}}^{n,n-k+1}(s;\btheta)
\frac{\partial \widetilde{\bf C}_{\mathcal{A}}^{n-k+1,n-k}(s;\btheta)}{\partial \btheta}.
\end{aligned}
\end{equation*}

\begin{algorithm}[H]
\caption{Evaluate $\frac{\partial }{\partial \btheta}\widetilde{\bf C}_{\mathcal{A}}^{n,n-k}(s,\btheta)$, $n=1,2,\ldots,N$}
\label{alg:C_n_nminusk_der_mem}
\begin{algorithmic}[1]

\Input ${\bf Q}, n, k, \btheta$
\Output $\frac{\partial}{\partial \btheta}\widetilde{\bf C}_{\mathcal{A}}^{n,n-k}(s,\btheta)$

\State Compute 
\[\widetilde{\bf C}_{\mathcal{A}}^{N,N-1}(s,\btheta)=
-({\bf Q}^{[N,N]}(\btheta)-s{\bf C}_N I(N\in\mathcal{A}))^{-1}
{\bf Q}^{[N,N-1]}(\btheta)
\]

\State Set ${\bf B}=\widetilde{\bf C}_{\mathcal{A}}^{N,N-1}(s,\btheta)$

\State Compute $\frac{\partial }{\partial \btheta}\widetilde{\bf C}_{\mathcal{A}}^{N,N-1}(s;\btheta)$ using \eqref{C_N_N-1_theta} and store as ${\bf D}$.

\For{$i = N-1, \ldots, n-k+1$}

\State Compute $\frac{\partial }{\partial \btheta}\widetilde{\bf C}_{\mathcal{A}}^{i,i-1}(s;\btheta)$ using \eqref{eq:alg_Cnnminusk_mem} in Appendix~\ref{eq:equations}.

\State Update ${\bf D}=\frac{\partial }{\partial \btheta}\widetilde{\bf C}_{\mathcal{A}}^{i,i-1}(s;\btheta)$

\State Compute
\[
\widetilde{\bf C}_{\mathcal{A}}^{i,i-1}(s,\btheta)=
-({\bf Q}^{[i,i]}(\btheta)-s{\bf C}_i I(i\in\mathcal{A})
+{\bf Q}^{[i,i+1]}(\btheta){\bf B})^{-1}
{\bf Q}^{[i,i-1]}(\btheta)
\]

\State Update ${\bf B}=\widetilde{\bf C}_{\mathcal{A}}^{i,i-1}(s,\btheta)$

\EndFor

\State Compute $\frac{\partial }{\partial \btheta}\widetilde{\bf C}_{\mathcal{A}}^{n,n-k}(s;\btheta)$ using \eqref{eq:C_n_nminusk_theta}

\end{algorithmic}
\end{algorithm}

\noindent Similarly, we apply Algorithm~\ref{alg:C_n_nplusk_der_mem} to evaluate the sensitivity of $\widetilde{\bf C}_{\mathcal{A}}^{n,n+k}(s,\btheta)$, which follows from Lemma~\ref{lem:Cs_theta_LST_nnplusk} below.

\begin{Lemma}
\label{lem:Cs_theta_LST_nnplusk}
We have,
\begin{eqnarray}
\label{eq:C_n_nplusk_theta}
\frac{\partial}{\partial\btheta}\widetilde{\bf C}_{\mathcal{A}}^{n,n+k}(s;\btheta)&=&
\frac{\partial \widetilde{\bf C}_{\mathcal{A}}^{n,n+k-1}(s;\btheta)}{ \partial \btheta}
\times\left(
{\bf I}_k\otimes \widetilde{\bf C}_{\mathcal{A}}^{n+k-1,n+k}(s;\btheta)
\right)
\nonumber\\
&&
+
\widetilde{\bf C}_{\mathcal{A}}^{n,n+k-1}(s;\btheta)
\times
\frac{\partial \widetilde{\bf C}_{\mathcal{A}}^{n+k-1,n+k}(s;\btheta)}{\partial \btheta},
\end{eqnarray}
where
\begin{eqnarray} 
\frac{\partial }{\partial \btheta}\widetilde{\bf C}_{\mathcal{A}}^{0,1}(s;\btheta)
  &=&
\left({\bf Q}^{[0,0]}(\btheta)-s{\bf C}_0\times I(0\in\mathcal{A})\right)^{-1}
\times
\frac{\partial {\bf Q}^{[0,0]}(\btheta)}{\partial \btheta}
\nonumber\\
&&
\times
\left(
{\bf I}_k
\otimes \left({\bf Q}^{[0,0]}(\btheta)-s{\bf C}_0\times I(0\in\mathcal{A})\right)^{-1}
\right)
\times\left(
{\bf I}_k\otimes {\bf Q}^{[0,1]}(\btheta)
\right)
\nonumber
\\
&&
-
\left({\bf Q}^{[0,0]}(\btheta)-s{\bf C}_0\times I(0\in\mathcal{A})\right)^{-1}
\times
\frac{\partial {\bf Q}^{[0,1]}(\btheta)}{\partial \btheta},
        \label{C_0_1_theta}
\end{eqnarray}
and for $n=1,\ldots,n+k-1$,
\begin{equation*}
\begin{aligned}
\frac{\partial }{\partial \btheta}&
\widetilde{\bf C}_{\mathcal{A}}^{n,n+1}(s;\btheta)
\\
&=
\left({\bf Q}^{[n,n]}(\btheta)-s{\bf C}_n I(n\in\mathcal{A})
+{\bf Q}^{[n,n-1]}(\btheta)\widetilde{\bf C}_{\mathcal{A}}^{n-1,n}(s,\btheta)
\right)^{-1}
\\
&\quad \times
\Bigg(
\frac{\partial {\bf Q}^{[n,n]}(\btheta)}{\partial \btheta}
+
\frac{\partial{\bf Q}^{[n,n-1]}(\btheta)}{\partial \btheta}
\left({\bf I}_k \otimes \widetilde{\bf C}_{\mathcal{A}}^{n-1,n}(s,\btheta)\right)
\\
&\qquad\qquad
+
{\bf Q}^{[n,n-1]}(\btheta)
\frac{\partial \widetilde{\bf C}_{\mathcal{A}}^{n-1,n}(s,\btheta)}{\partial \btheta}
\Bigg)
\\
&\quad \times
\left(
{\bf I}_k \otimes
\left({\bf Q}^{[n,n]}(\btheta)-s{\bf C}_n I(n\in\mathcal{A})
+{\bf Q}^{[n,n-1]}(\btheta)\widetilde{\bf C}_{\mathcal{A}}^{n-1,n}(s,\btheta)
\right)^{-1}
\right)
\\
&\quad \times
\left({\bf I}_k\otimes {\bf Q}^{[n,n+1]}(\btheta)\right)
\\
&\quad -
\left({\bf Q}^{[n,n]}(\btheta)-s{\bf C}_n I(n\in\mathcal{A})
+{\bf Q}^{[n,n-1]}(\btheta)\widetilde{\bf C}_{\mathcal{A}}^{n-1,n}(s,\btheta)
\right)^{-1}
\frac{\partial {\bf Q}^{[n,n+1]}(\btheta)}{\partial \btheta}.
\end{aligned}
\end{equation*}
\end{Lemma}
\noindent{\bf Proof:} The proof follows by arguments analogous to Lemma~\ref{lem:Cs_theta_LST_nnminuk}.

\begin{algorithm}[H]
\caption{Evaluate $\frac{\partial }{\partial \btheta}\widetilde{\bf C}_{\mathcal{A}}^{n,n+k}(s,\btheta)$, $n=0,1,\ldots,N-1$}
\label{alg:C_n_nplusk_der_mem}
\begin{algorithmic}[1]

\Input ${\bf Q}, n, k, \btheta$
\Output $\frac{\partial}{\partial \btheta}\widetilde{\bf C}_{\mathcal{A}}^{n,n+k}(s,\btheta)$

\State Compute
\[
\widetilde{\bf C}_{\mathcal{A}}^{0,1}(s,\btheta)=
-({\bf Q}^{[0,0]}(\btheta)-s{\bf C}_0 I(0\in\mathcal{A}))^{-1}
{\bf Q}^{[0,1]}(\btheta).
\]

\State Set ${\bf B}=\widetilde{\bf C}_{\mathcal{A}}^{0,1}(s,\btheta)$

\State Compute $\frac{\partial }{\partial \btheta}\widetilde{\bf C}_{\mathcal{A}}^{0,1}(s;\btheta)$ using \eqref{C_0_1_theta} and store as ${\bf D}$

\For{$i = 1, \ldots, n+k-1$}

\State Compute $\frac{\partial }{\partial \btheta}\widetilde{\bf C}_{\mathcal{A}}^{i,i+1}(s;\btheta)$ using \eqref{eq:alg_Cnplusk_mem} in Appendix~\ref{eq:equations}.

\State Update ${\bf D}=\frac{\partial }{\partial \btheta}\widetilde{\bf C}_{\mathcal{A}}^{i,i+1}(s;\btheta)$

\State Compute
\[
\widetilde{\bf C}_{\mathcal{A}}^{i,i+1}(s,\btheta)=
-({\bf Q}^{[i,i]}(\btheta)-s{\bf C}_i I(i\in\mathcal{A})
+{\bf Q}^{[i,i-1]}(\btheta){\bf B})^{-1}
{\bf Q}^{[i,i+1]}(\btheta)
\]

\State Update ${\bf B}=\widetilde{\bf C}_{\mathcal{A}}^{i,i+1}(s,\btheta)$

\EndFor

\State Compute $\frac{\partial }{\partial \btheta}\widetilde{\bf C}_{\mathcal{A}}^{n,n+k}(s;\btheta)$ using \eqref{eq:C_n_nplusk_theta}

\end{algorithmic}
\end{algorithm}

\section{Application Examples: Redirected and Transferred Patients }\label{sec:application}

In this section, we consider the application of the QBD models in Section~\ref{sec:QBD-models} to a hospital system with finite capacity $N$ that serves two types of patients: Type-A (complex) and Type-B (others), classified based on their Diagnosis Related Group (DRG) category described in Rahmawati et al.~\cite{rahmawati2021cost}. Specifically, patients whose DRG category indicates \emph{major complexity} were classified as Type-A, while all remaining patients were classified as Type-B. Analyses of hospital data suggest that admission and discharge processes can be well approximated by Poisson arrivals with exponentially distributed service times, see Whitt and Zhang~\cite{whitt2017data}. This assumption is often used in queueing‑based models of hospital systems and is explicitly adopted by researchers, see Wang et al.~\cite{wang2025markov}.

\subsection{Input parameters}
\label{input_parameters}

Parameters $\lambda_A$,$\lambda_B$, $\mu_A$, and $\mu_B$, summarised in Table~\ref{tab:parametersQBDs} were estimated from a five-year hospital data obtained from an Australian tertiary referral hospital.  Using the recorded length-of-stay data for each patient type, we estimated the service rates $\mu_A$ and $\mu_B$ by fitting exponential distributions using the \texttt{fitdist} command in \textsc{Matlab}. We obtained the arrival rates $\lambda_A$ and $\lambda_B$ by fitting Poisson distribution using the \texttt{poissfit} command in \textsc{Matlab} to the recorded daily number of arrivals of each type of patients. Further, in line with the reported complexity-classification performance, which remains robust under misclassification rates up to $25\%$ (see Saghafian et al.~\cite{saghafian2014complexity}), we take $p_{AA}=0.85$ and $p_{BA}=0.15$ as representative values for the correct and incorrect identification of complex (Type-A) patients upon arrival.

The remaining parameters $p$, $M_B$, and $N$, correspond to the admission policy choices. Although hospitals typically designate a fixed number of beds $(N-M_B)$ to be reserved for complex (Type-A) patients, informal discussions with hospital managers suggest that this policy is not enforced rigidly in practice. Instead, when the system is highly occupied, managers may still admit other (Type‑B) patients into the reserved beds $(N-M_B)$ with a small probability $p$, reflecting operational flexibility rather than strict reservation. This behaviour captures a moderate level of selectivity rather than absolute prioritisation. Motivated by these observations, we set $N-M_B=10$ and choose $p=0.25$.

\begin{table}[htbp]
\centering
\begin{tabular}{l l l}
\hline
Parameter & Type-A & Type-B \\
\hline

Arrival rate (per type, per day) 
& $\lambda_{A}=16.1298$ 
& $\lambda_{B}=46.7864$ \\

Mean LoS (in days) 
& $\mathbb{E}(T_A)=6.73$ 
& $\mathbb{E}(T_B)=2.50$ \\

$95\%$ CI of Mean LoS 
& $[6.65, 6.81]$ 
& $[2.48, 2.52]$ \\

Departure rate (per patient, per day) 
& $\mu_{A}=0.1486$ 
& $\mu_{B}=0.4002$ \\

\hline
\end{tabular}

\caption{Model parameters for QBD-I, QBD-II, and QBD-III.}
\label{tab:parametersQBDs}

\end{table}

\subsection{Effect of N on long-run performance measures}

For the completeness of the analysis, first we study the effect of the size of the system $N$ on the following long-run key performance measures:
\begin{itemize}
\item  $\pi_{N,\bullet} = \sum_i \pi_{N,i}$ -- proportion of time the system is busy;

\item  $\pi^{B}_{\text{restrict}} = 100 \times \sum_{n=M_B}^{N} \pi_{n,\bullet}$ -- percentage of time the occupancy is at least $M_B$.

\item $L=\sum_n n\pi_{n,\bullet} $ -- mean number of customers in the system.

\item $L_A= \sum_i i\pi_{\bullet,i}$ -- mean number of Type-A customers in the system.

\item $L_B = L - L_A $ -- mean number of Type-B customers in the system.

\item $N_A = 100(L_A/L)$ -- proportion of Type-A customers in the system.

\item $N_B = 100(L_B/L)$ -- proportion of Type-B customers in the system.

\item $O_N=100(L/N)$ -- mean occupancy of the system.
\end{itemize}

Further, under QBD-III, the rates at which Type-A and Type-B customers are redirected, are
\begin{align}
    \lambda_{A:\text{redirect}} &=  \sum_{n=M_B}^{N-1} \sum_{i=0}^{n} \pi_{(n,i)} (1-p)\lambda_A (1-p_{AA}) + \pi_{N,\bullet}
   \lambda_A(1-p_{AA}) + \pi_{N,N}p_{AA}\lambda_A,
   \label{QBDIIIredirectA}
   \\
    \lambda_{B:\text{redirect}} &= \sum_{n=M_B}^{N-1} \sum_{i=0}^{n} \pi_{(n,i)} (1-p)\lambda_B (1-p_{BA})) + \pi_{N,\bullet} \lambda_B(1-p_{BA}) + \pi_{N,N} p_{BA}\lambda_B.
    \label{QBDIIIredirectB}
\end{align}
For QBD-II, the corresponding redirection rates can be obtained by putting $p=1$ in the above. For QBD-I, the redirection rates can be obtained by putting $p=1$ as well as $p_{AA}=0$.

Finally, under QBD-II and QBD-III,  $$\lambda_{\text{transfer}}=\big(\lambda_A p_{AA} + \lambda_B p_{BA}\big) (\pi_{N,\bullet} - \pi_{N,N})$$ is the mean transfer rate.

As expected, we observe that $L$, $L_A$, and $L_B$ increase with $N$ (Figures~\ref{metrics_as_N_II}-\ref{QBDIII_LA_N_with_N_MB_10}). For all values of $N$, $L_A$ is largest under QBD-III and smallest under QBD-I, whereas $L_B$ exhibits the opposite pattern. This is because QBD-II prioritises Type-A patients, as compared to QBD-I, through the transfer policy and QBD-III further strengthens this priority by controlling the Type-B admissions through guard-channel threshold policy. We note that $L_A$ is the most stable under QBD-III, indicating that reserving a small number of beds for complex patients may stabilise the number of Type-A patients in the system which may help in capacity planning and resource utilisation.

In Figures~\ref{metrics_as_N_I}a and~\ref{metrics_as_N_I}b, we observe that $\pi^{*}_{N,\bullet}$ and $\pi^{B}_{\text{restrict}}$ decrease as $N$ increases, reflecting the increased availability of capacity as the number of beds increases, which is expected. For a given $N$, $\pi^{*}_{N,\bullet}$ is higher under QBD-II than QBD-I due to the transfer policy, which favours Type-A admissions who have longer lengths of stay, thereby increasing congestion. In contrast, $\pi^{*}_{N,\bullet}$ remains small under QBD-III because of the guard-channel threshold $M_B$ and probabilistic admission control $p$. Therefore, reserving a small number of beds for complex patients can significantly reduce congestion, though it leads to increased redirection of Type-B patients, as discussed below.

Figures~\ref{metrics_as_N_I}c,~\ref{metrics_as_N_I}d,~\ref{metrics_as_N_I}e, and~\ref{metrics_as_N_I}f illustrate the monotonic decrease in the transfer rates $(\lambda_{\text{transfer}})$ and redirection rates $(\lambda_{\text{redirect}}, \lambda_{\text{A:redirect}}, \lambda_{\text{B:redirect}})$ as the system capacity $N$ increases. This decreasing behaviour occurs because of the decrease in the percentage of time the system is congested ($\pi^{*}_{N,\bullet}$), as shown in Figure~\ref{metrics_as_N_I}a. Under QBD-III, the transfer rate $\lambda_{\text{transfer}}$ remains small because the system is rarely fully occupied ($\pi^{*}_{N,\bullet}$ is small) and so Type-A arrivals seldom encounter a full system, resulting in a very few Type-B transfers.

Figure~\ref{metrics_as_N_I}e shows that $\lambda_{\text{A:redirect}}$ is the highest under QBD-I for any value of $N$. This follows from the fact that Type-A patients are prioritised under QBD-II and QBD-III, reducing the likelihood that they are redirected, compared to QBD-I. An interesting observation is that, although the system is rarely full under QBD-III, the Type-A redirection rate $\lambda_{\text{A:redirect}}$ is higher under QBD-III than under QBD-II. This behaviour is primarily driven by misclassification of Type-A arrivals once the guard-channel threshold $M_B$ has been reached, which is captured using the term $\sum_{n=M_B}^{N-1} \sum_{i=0}^{n} \pi_{(n,i)} (1-p)\lambda_A (1-p_{AA})$ in~\eqref{QBDIIIredirectA}.

Similarly, Figure~\ref{metrics_as_N_I}f shows that $\lambda_{\text{B:redirect}}$ is higher under QBD-III than under QBD-II. This increase is a direct consequence of the guard-channel threshold policy $(M_B,p)$, under which Type-B patients are increasingly blocked once the threshold $M_B$ is reached. The contribution of this mechanism to the Type-B redirection rate is quantified by the term $\sum_{n=M_B}^{N-1} \sum_{i=0}^{n} \pi_{(n,i)} (1-p)\lambda_B (1-p_{BA})$ in~\eqref{QBDIIIredirectB}. Further, $O_N$ decreases with $N$ and remains consistently lowest under QBD-III for all values of $N$, which is expected because $\pi^{*}_{N,\bullet}$ decreases with $N$ and remains lowest under QBD-III. Finally, in Table~\ref{tab:metrics_N220_QBDs} we present the numerical values of the performance measures for $N=220$, under the models QBD-I, QBD-II, and QBD-III. 

These observations suggest that reserving a small number of beds for Type-A patients may be a suitable policy choice, provided the system can bear a modest increase in the redirection of Type-B patients (which will affect some external system, and could be used to evaluate its suitable capacity).

\begin{table}[htbp]
\centering
\renewcommand{\arraystretch}{1.3}

\resizebox{\textwidth}{!}{
\begin{tabular}{lcccccccccc}

\hline
Models 
& $\pi^{*}_{N,\bullet}$ 
& $\pi^{B}_{\text{restrict}}$ 
& $\lambda_{\text{transfer}}$ 
& $\lambda_{\text{redirect}}$ 
& $\lambda^{A}_{\text{redirect}}$ 
& $\lambda^{B}_{\text{redirect}}$ 
& $L$ 
& $L_A$ 
& $L_B$ 
& $O_N$ \\
\hline

QBD-I 
& $6.727$ 
& $61.70$ 
& $-$ 
& $4.2325$ 
& $1.0851$ 
& $3.1474$ 
& $210.32$ 
& $101.27$ 
& $109.05$ 
& $95.60\%$ \\

QBD-II 
& $8.368$ 
& $67.78$ 
& $1.7346$ 
& $3.5303$ 
& $0.2025$ 
& $3.3279$ 
& $211.47$ 
& $107.21$ 
& $104.26$ 
& $96.12\%$ \\

QBD-III 
& $0.011$ 
& $23.25$ 
& $0.0022$ 
& $7.3574$ 
& $0.4219$ 
& $6.9355$ 
& $205.09$ 
& $108.57$ 
& $96.52$ 
& $93.22\%$ \\

\hline
\end{tabular}
}

\caption{Some performance measures for $N=220$, under the models QBD-I, QBD-II, and QBD-III. We note that $\lambda_{\text{redirect}}=\lambda_{A:\text{redirect}}+\lambda_{B:\text{redirect}}$, $\pi^{*}_{N,\bullet}=100 \times \pi_{N,\bullet}$, and $\pi^{B}_{\text{restrict}}=100 \times \sum_{n=M_B}^{N} \pi_{n,\bullet}$. We assumed $p_{AA}=0.85$, $p_{BA}=0.15$, $N=220$, $M_B=210$, $p=0.25$, and $\lambda_A$, $\lambda_B$, $\mu_A$, $\mu_B$ as given in Table~\ref{tab:parametersQBDs}.}
\label{tab:metrics_N220_QBDs}
\end{table}

\begin{figure}[htbp]
\centering
\includegraphics[width=0.46\textwidth]{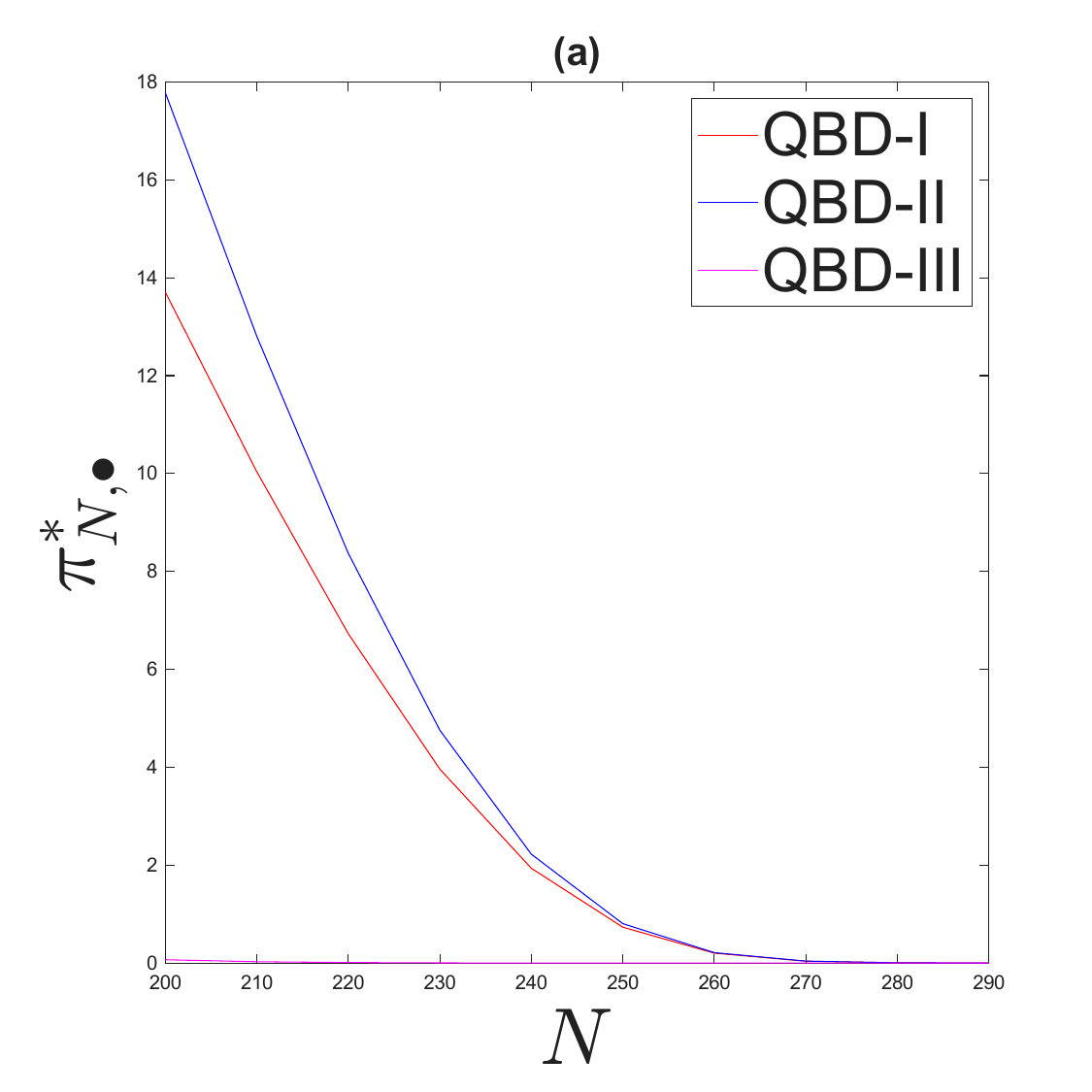}
\hfill
\includegraphics[width=0.46\textwidth]{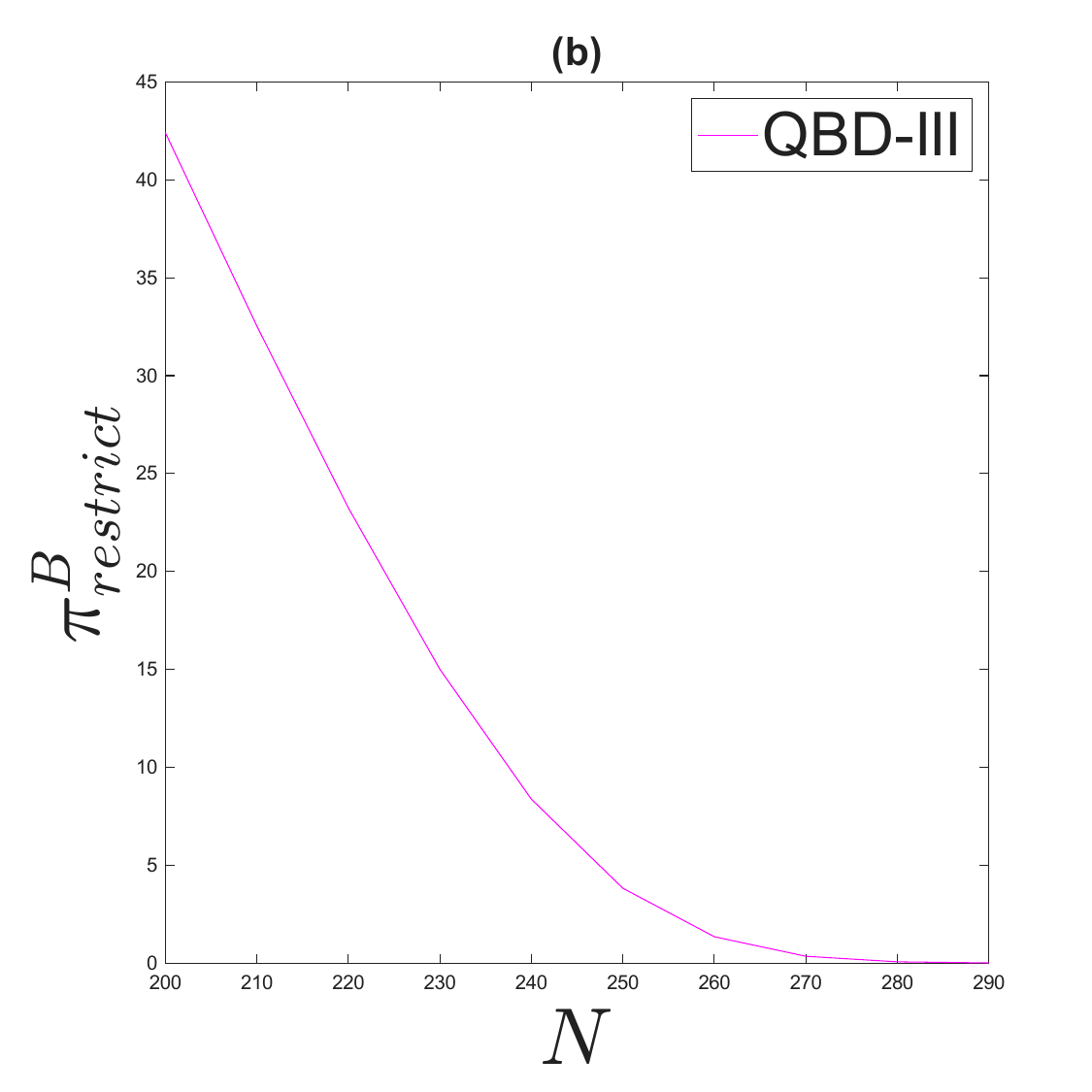}

\includegraphics[width=0.46\textwidth]{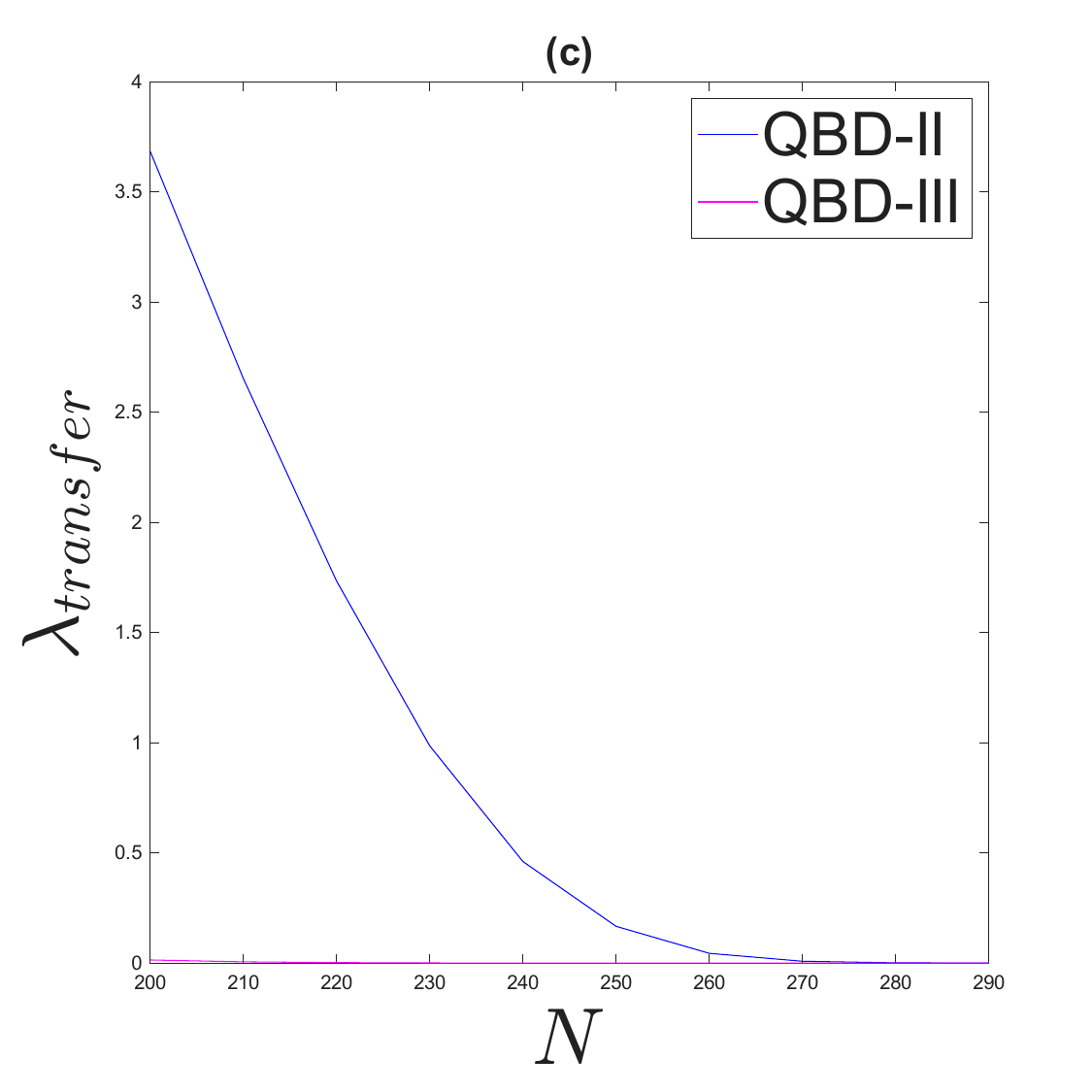}
\hfill
\includegraphics[width=0.46\textwidth]{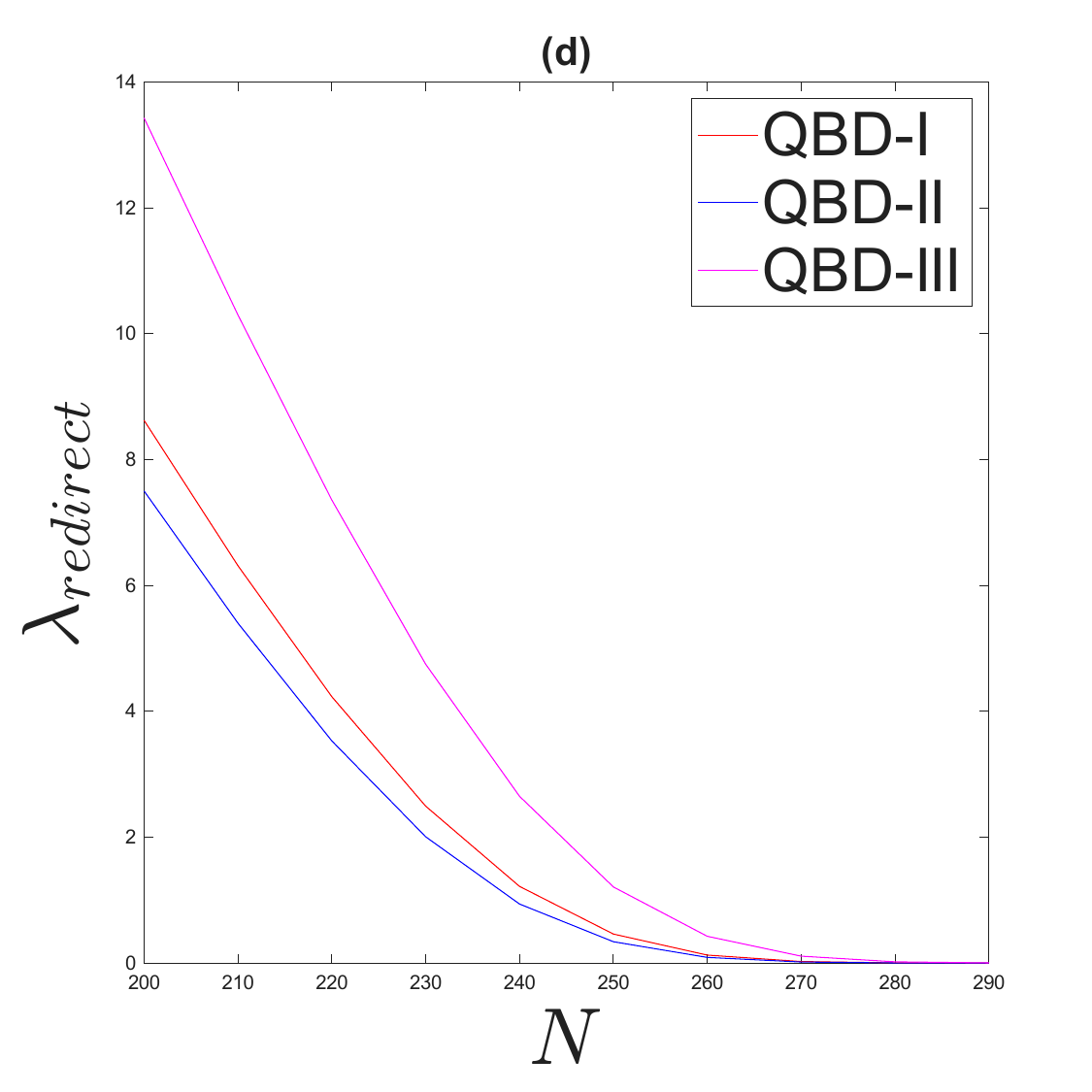}

\includegraphics[width=0.46\textwidth]{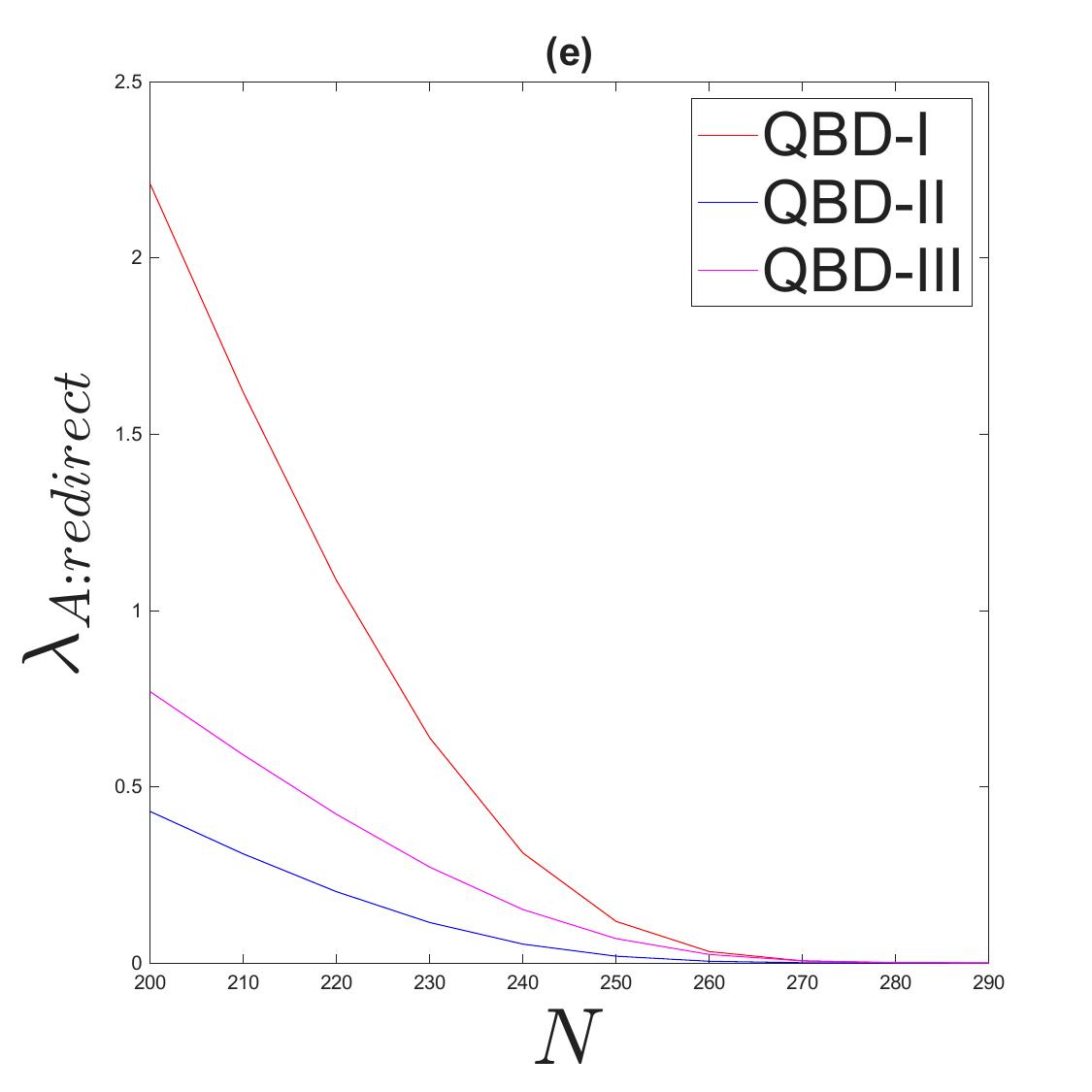}
\hfill
\includegraphics[width=0.46\textwidth]{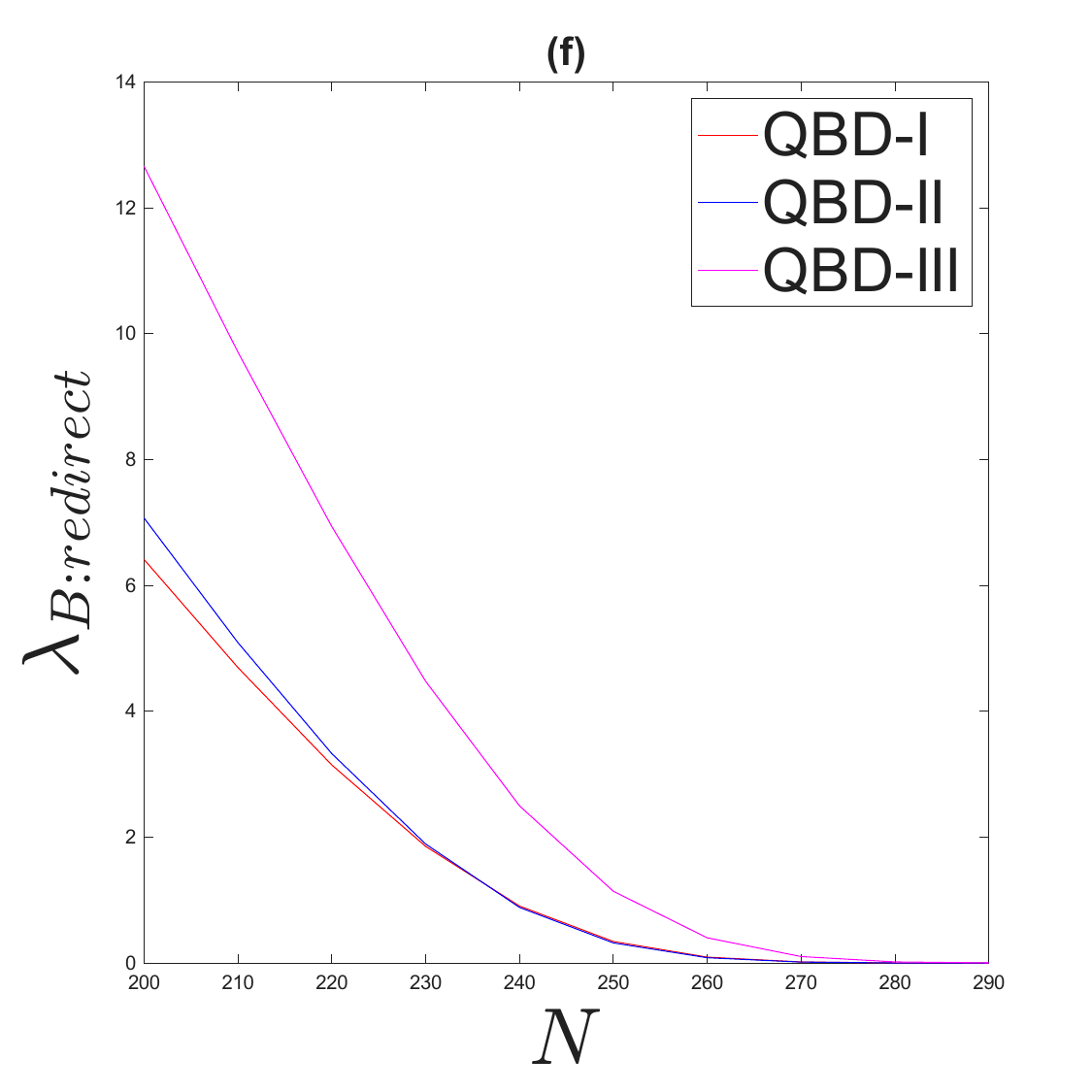}

\caption{$\pi^{*}_{N,\bullet}$, $\pi^{B}_{\text{restrict}}$, $\lambda_{\text{transfer}}$, $\lambda_{\text{redirect}}$, $\lambda_{A:\text{redirect}}$, and $\lambda_{B:\text{redirect}}$ as a function of $N$. We assumed $p_{AA}=0.85$, $p_{BA}=0.15$, $N=220$, $M_B=210$, $p=0.25$, and $\lambda_A$, $\lambda_B$, $\mu_A$, $\mu_B$ as given in Table~\ref{tab:parametersQBDs}.}
\label{metrics_as_N_I}
\end{figure}

\begin{figure}[htbp]
\centering
\includegraphics[width=0.48\textwidth]{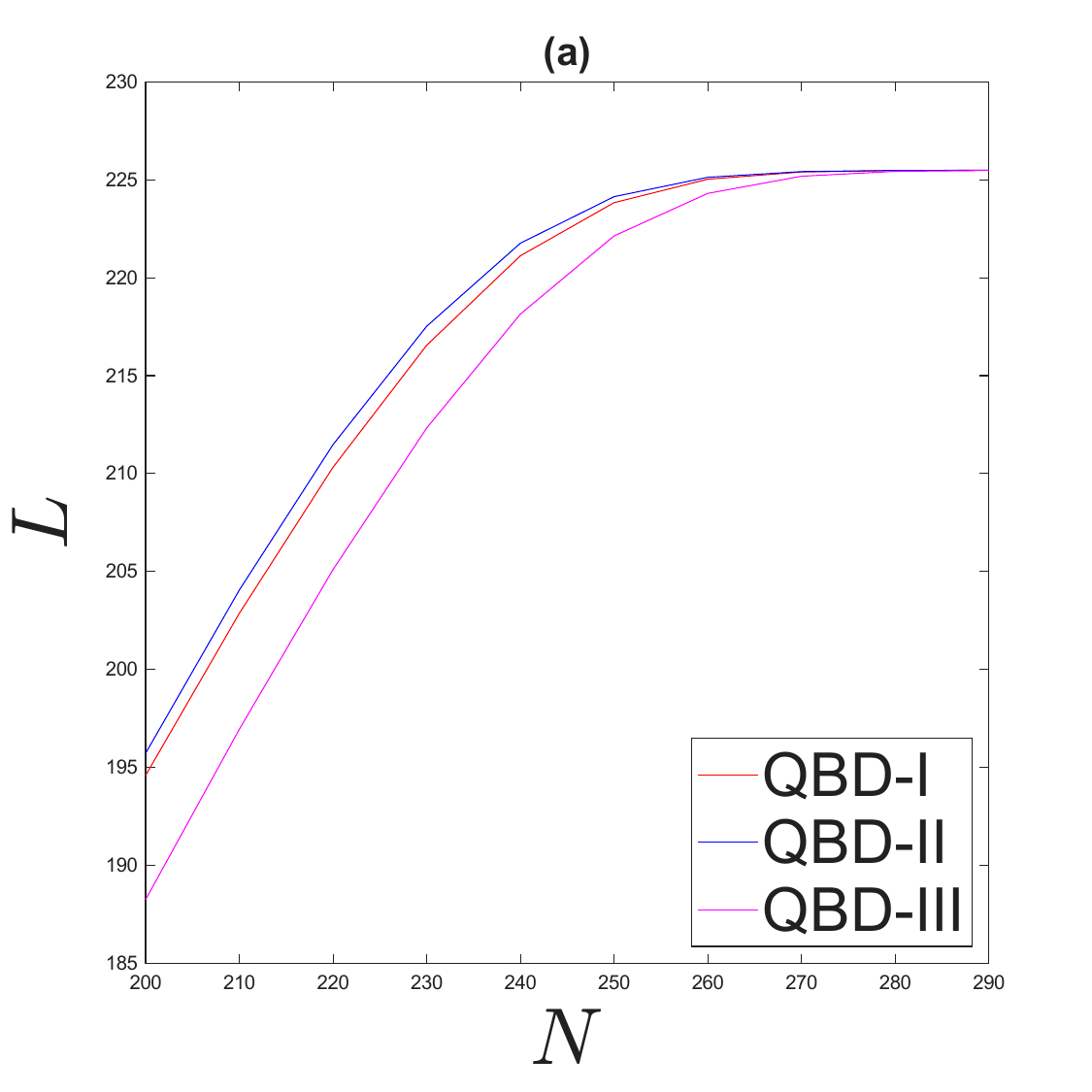}
\hfill
\includegraphics[width=0.48\textwidth]{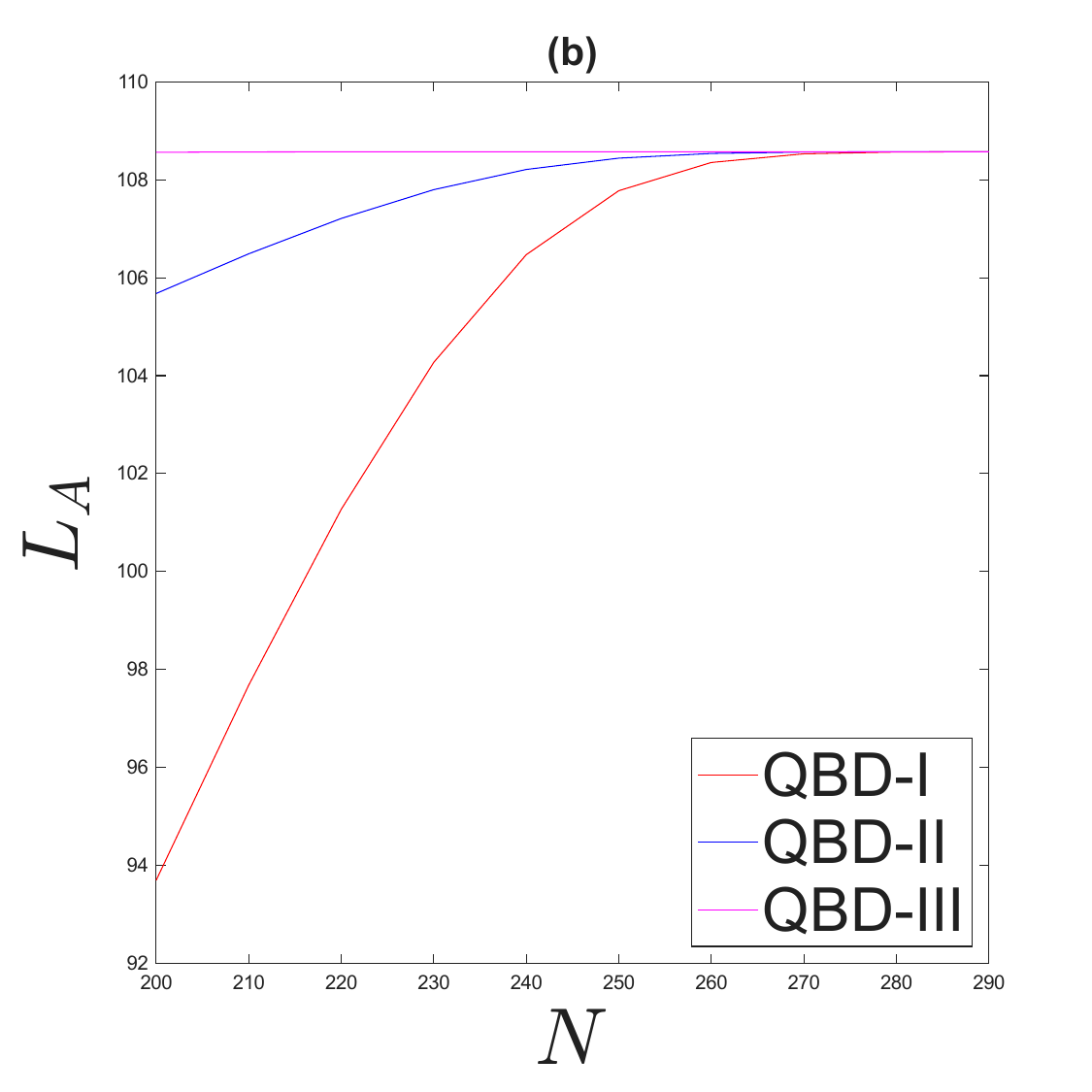}

\includegraphics[width=0.48\textwidth]{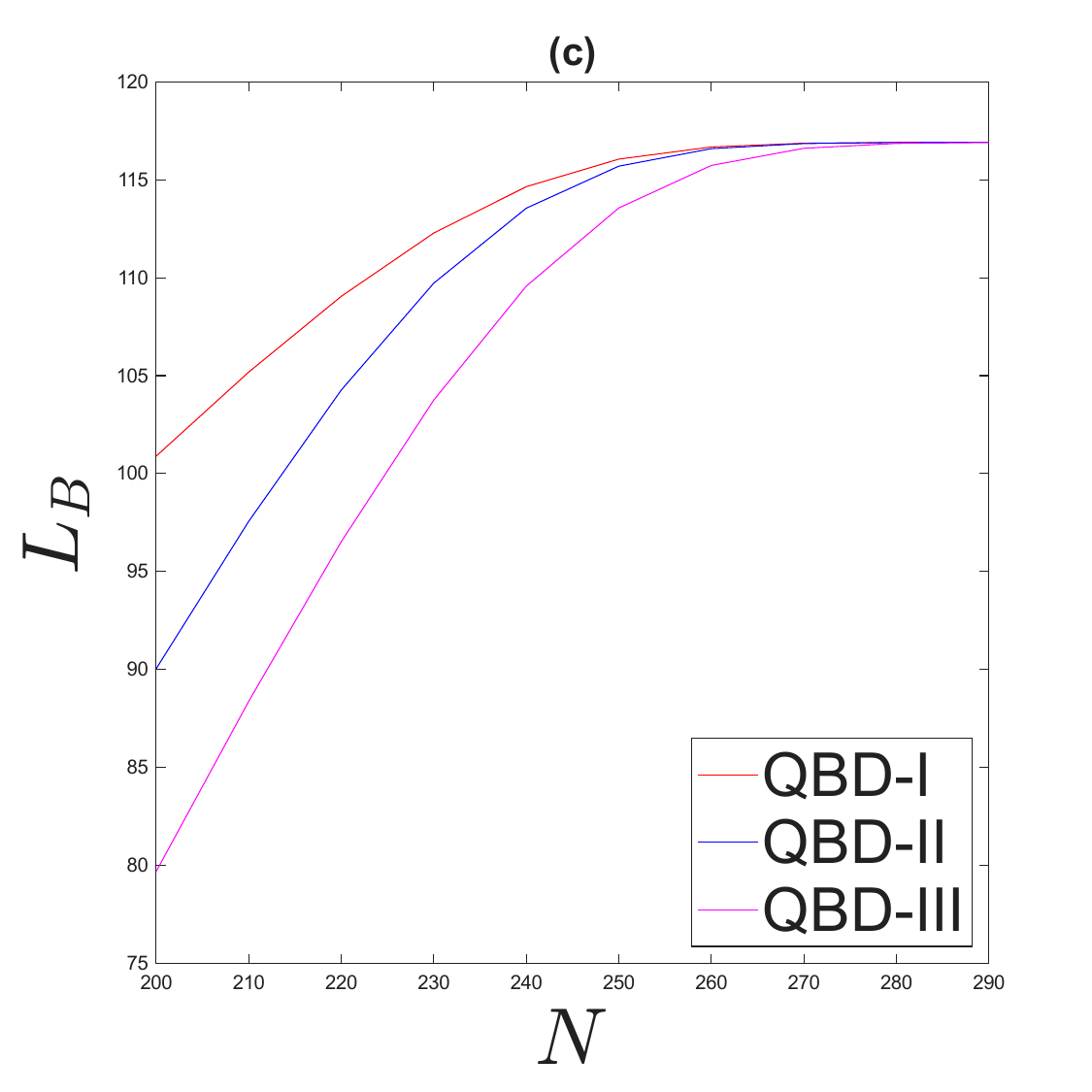}
\hfill
\includegraphics[width=0.48\textwidth]{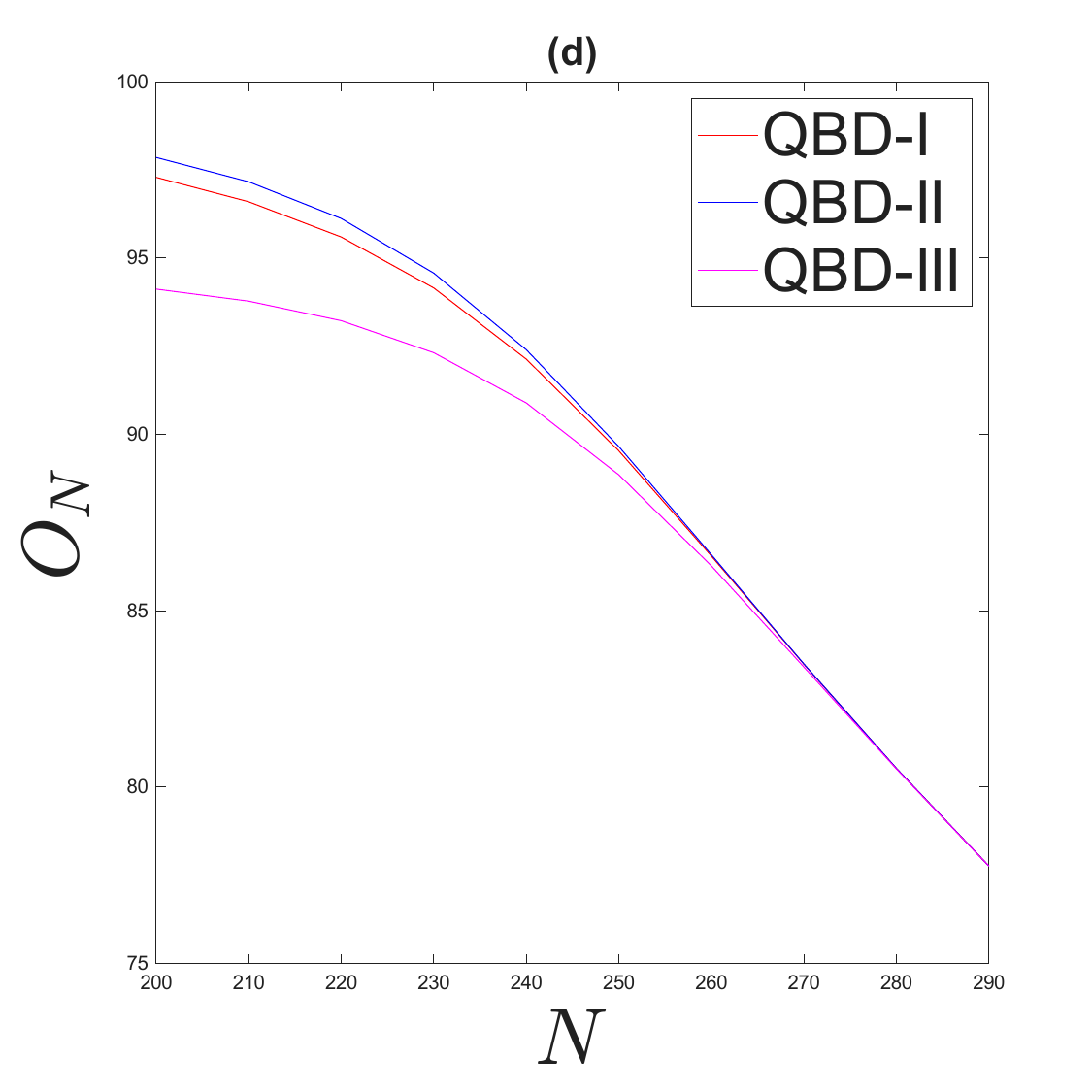}

\caption{$L$, $L_A$, $L_B$, and $O_N$ as a function of $N$. We assumed $p_{AA}=0.85$, $p_{BA}=0.15$, $N=220$, $M_B=210$, $p=0.25$, and $\lambda_A$, $\lambda_B$, $\mu_A$, $\mu_B$ as given in Table~\ref{tab:parametersQBDs}.}
\label{metrics_as_N_II}
\end{figure}

\begin{figure}[htbp]
\centering
\includegraphics[width=0.48\textwidth]{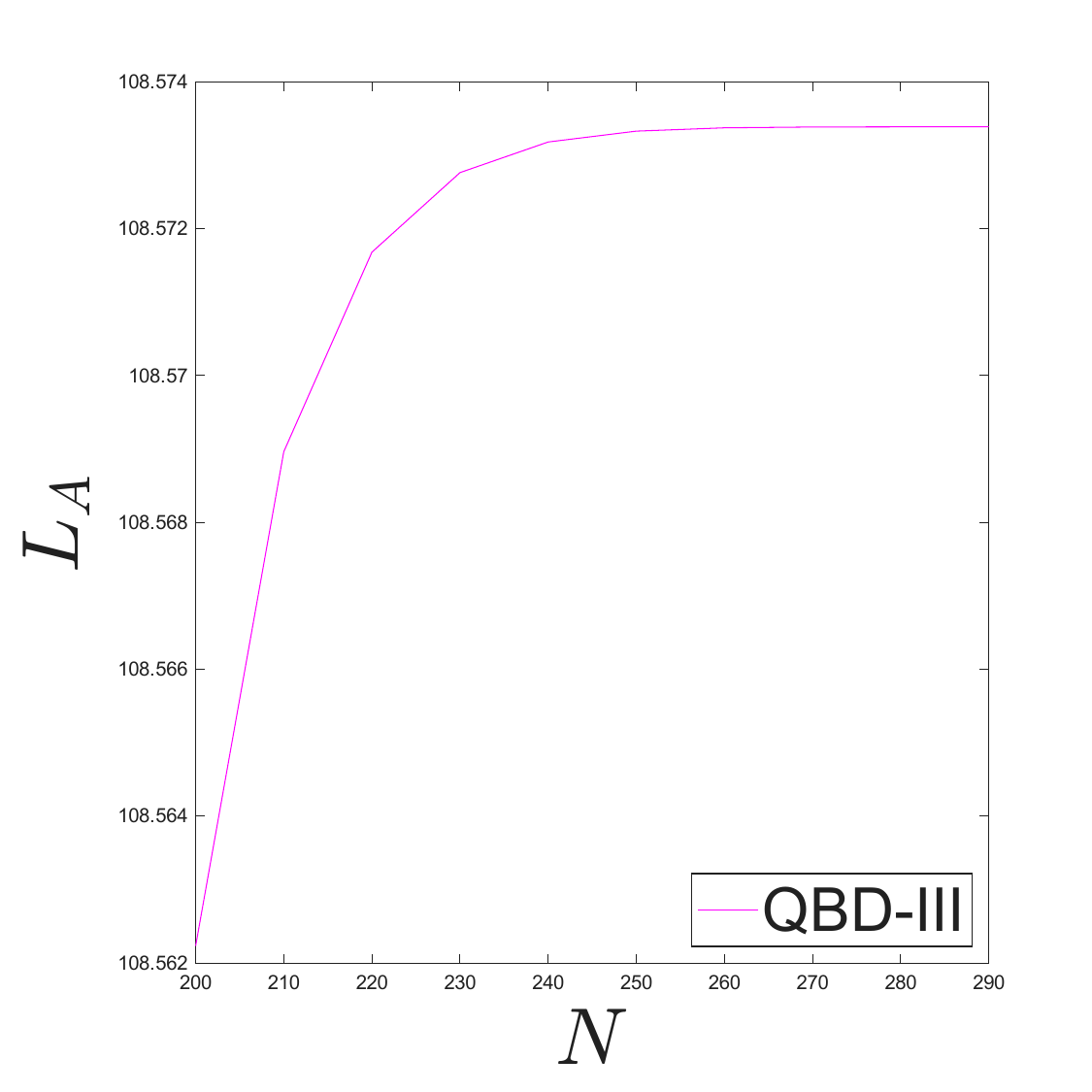}

\caption{$L_A$ as a function of $N$, under the model QBD-III. We assumed $p_{AA}=0.85$, $p_{BA}=0.15$, $N=220$, $M_B=210$, $p=0.25$, and $\lambda_A$, $\lambda_B$, $\mu_A$, $\mu_B$ as given in Table~\ref{tab:parametersQBDs}.}
\label{QBDIII_LA_N_with_N_MB_10}
\end{figure}

\subsection{Hitting time analysis}
\label{sec:hitting_tim_analysis}

Later, in Section~\ref{sec:CostAnalysis}, we consider the distribution of {\em costs} accumulated at the times until $k$ beds become available. So it is useful to also analyse the distribution of the corresponding first hitting times. Therefore, we compute the distribution of the time until $k$ beds become available, for $k \in \{5,10,15\}$, conditional on the system initially being full $(n=N=220)$ with $i=120$ complex patients. Computations are performed using Algorithm~\ref{Gs_algorithm} in Appendix~\ref{AppendixA} and the numerical inversion methods in Den Iseger~\cite{DenIseger_2006}. The results are shown in Figure~\ref{fig:first_hitting_times_below} and Table~\ref{tab:mean_time_below}.

The output indicates that the time until $k$ beds become available is much shorter under QBD-III compared to QBD-I and QBD-II, because the guard-channel threshold policy $(M_B, p)$ in QBD-III substantially supresses the arrivals once the occupancy reaches the threshold $M_B$. This reduces the likelihood that the system remains highly congested, leading to a significantly smaller percentage of time $\pi^{B}_{\text{restrict}}$ there are at least $M_B$ patients in the system, under QBD‑III, as evidenced in Table~\ref{tab:metrics_N220_QBDs}.

Further, the time until $k$ beds become available may be longer under QBD-II than QBD-I. This is because QBD-II prioritises the complex patients (Type‑A) through transfer policy and they remain in the system for longer periods due to longer lengths of stay, which prolongs the duration of high occupancy. 

These findings highlight the potential advantages of QBD‑III for hospital decision‑making, as the guard‑channel policy leads not only to reduced likelihood of congestion but also ensures faster recovery of beds once congestion occurs. In contrast, policies that rely solely on transfers (QBD‑II) may correspond to longer periods of high occupancy.

\begin{figure}[htbp]
\centering
\includegraphics[width=0.45\textwidth]{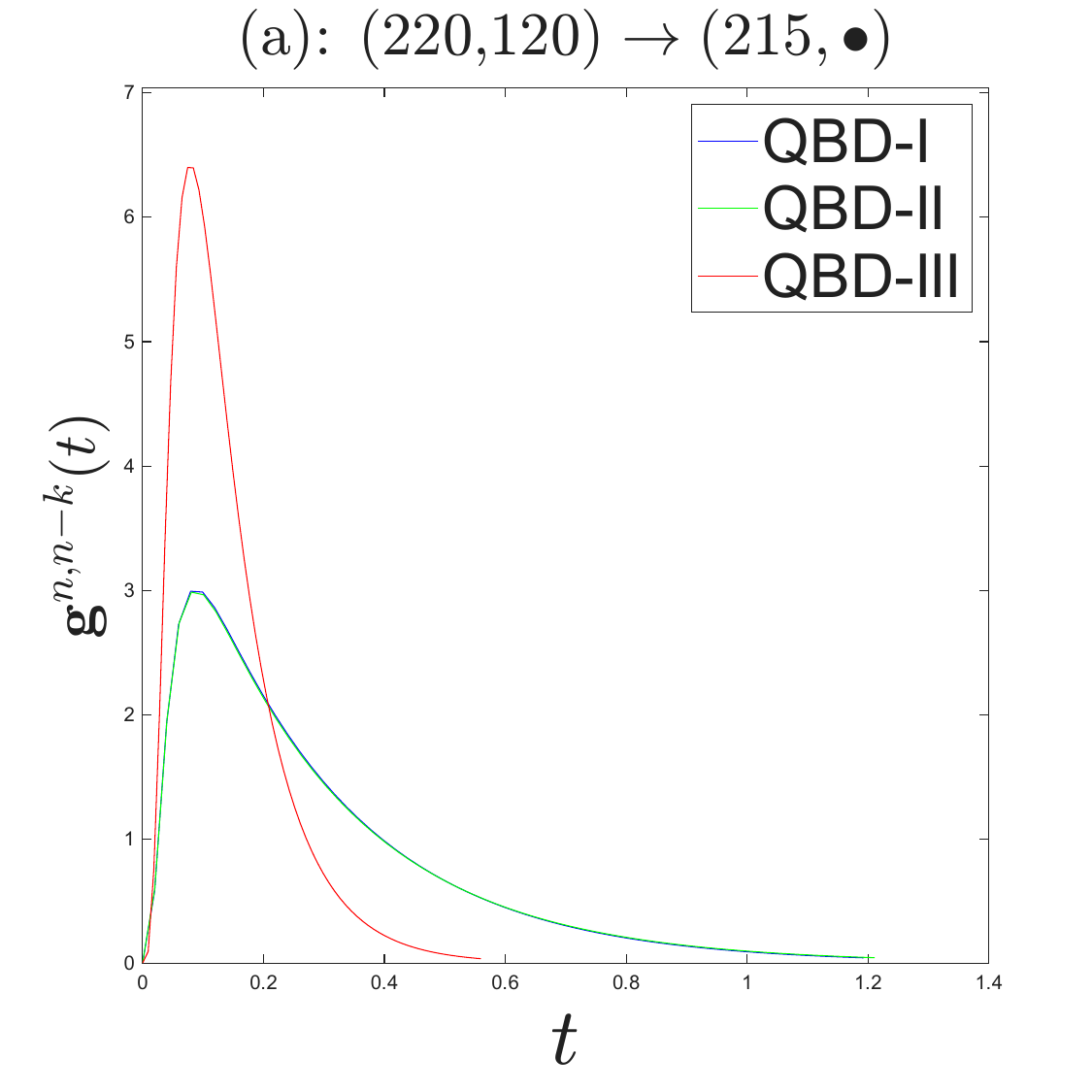}
\hfill
\includegraphics[width=0.45\textwidth]{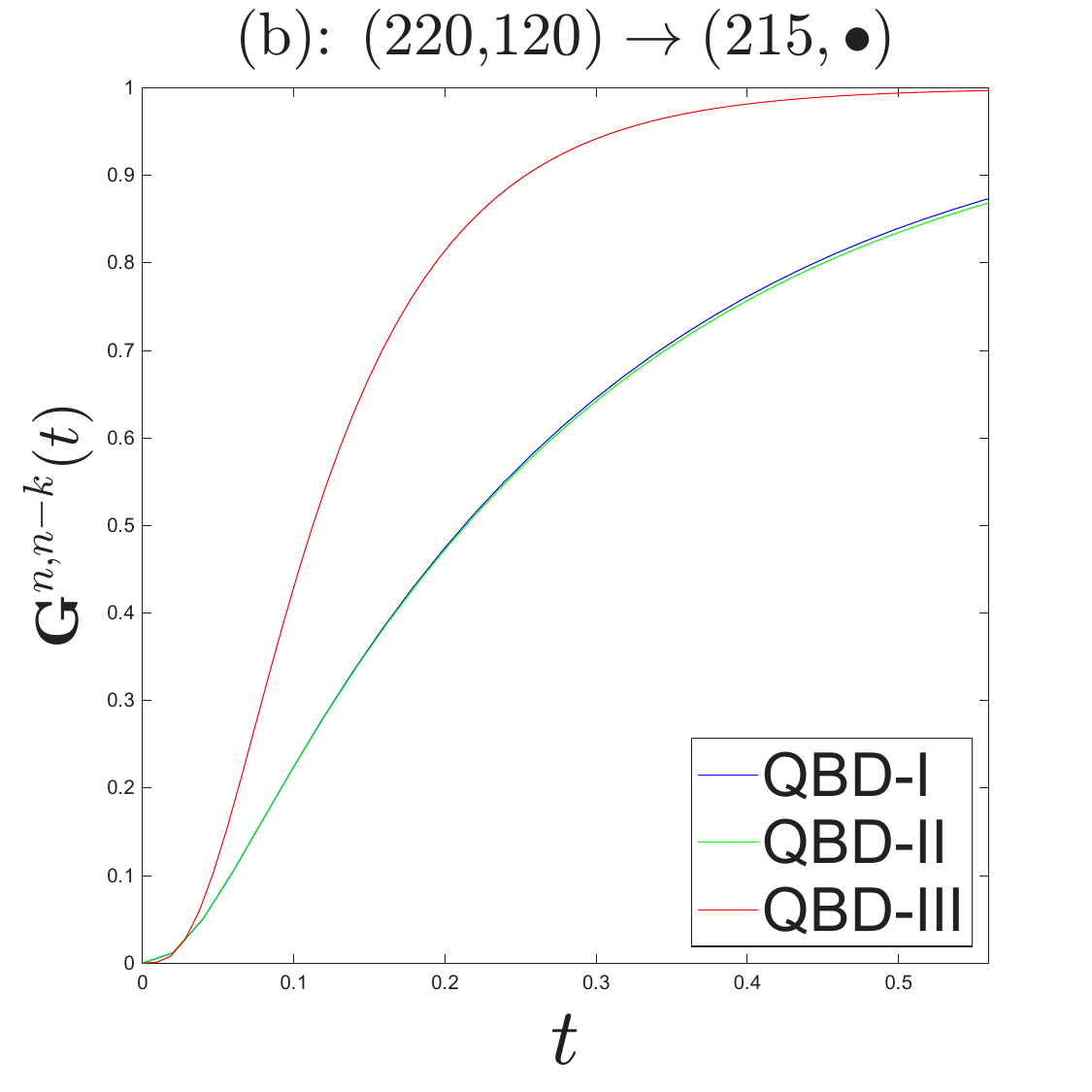}

\includegraphics[width=0.45\textwidth]{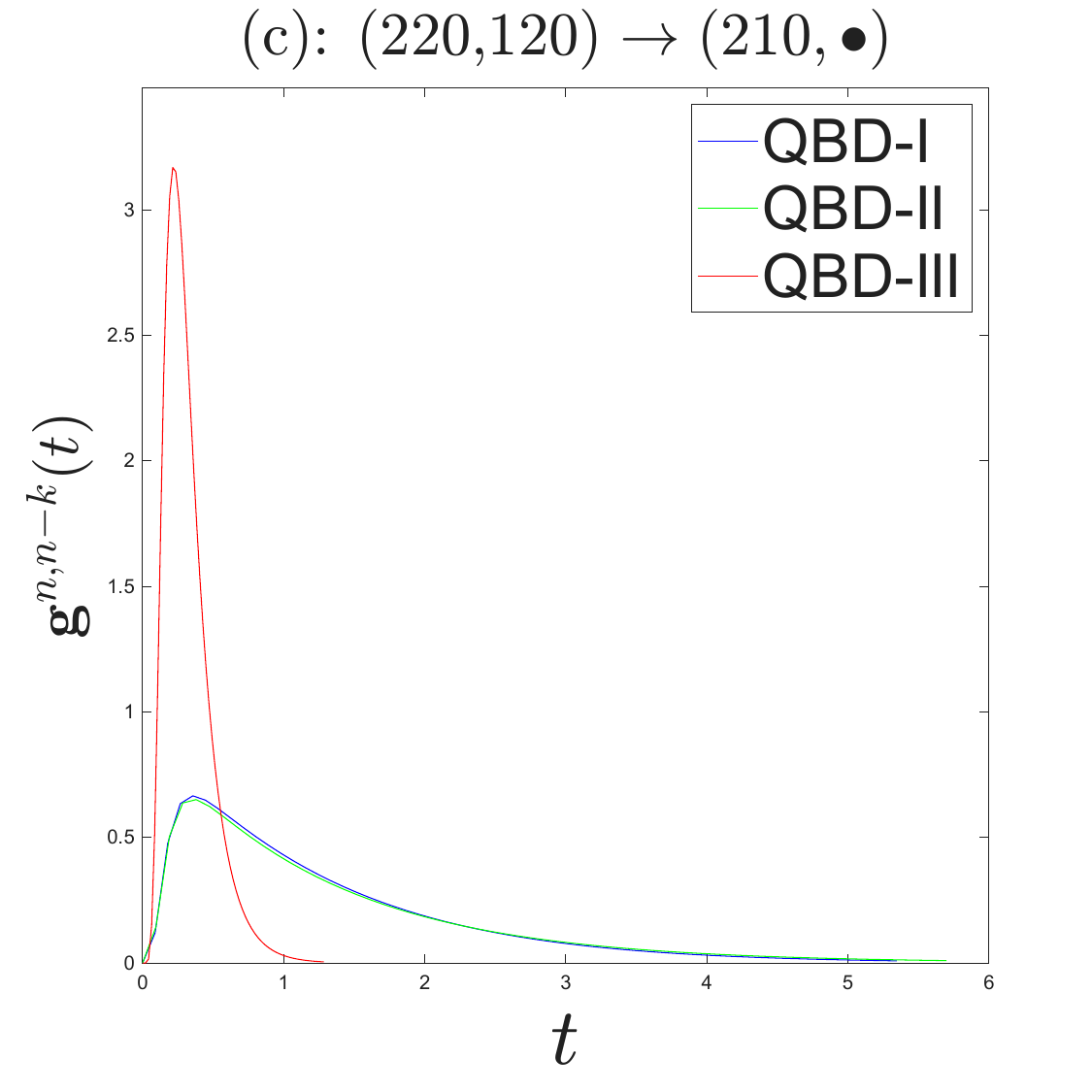}
\hfill
\includegraphics[width=0.45\textwidth]{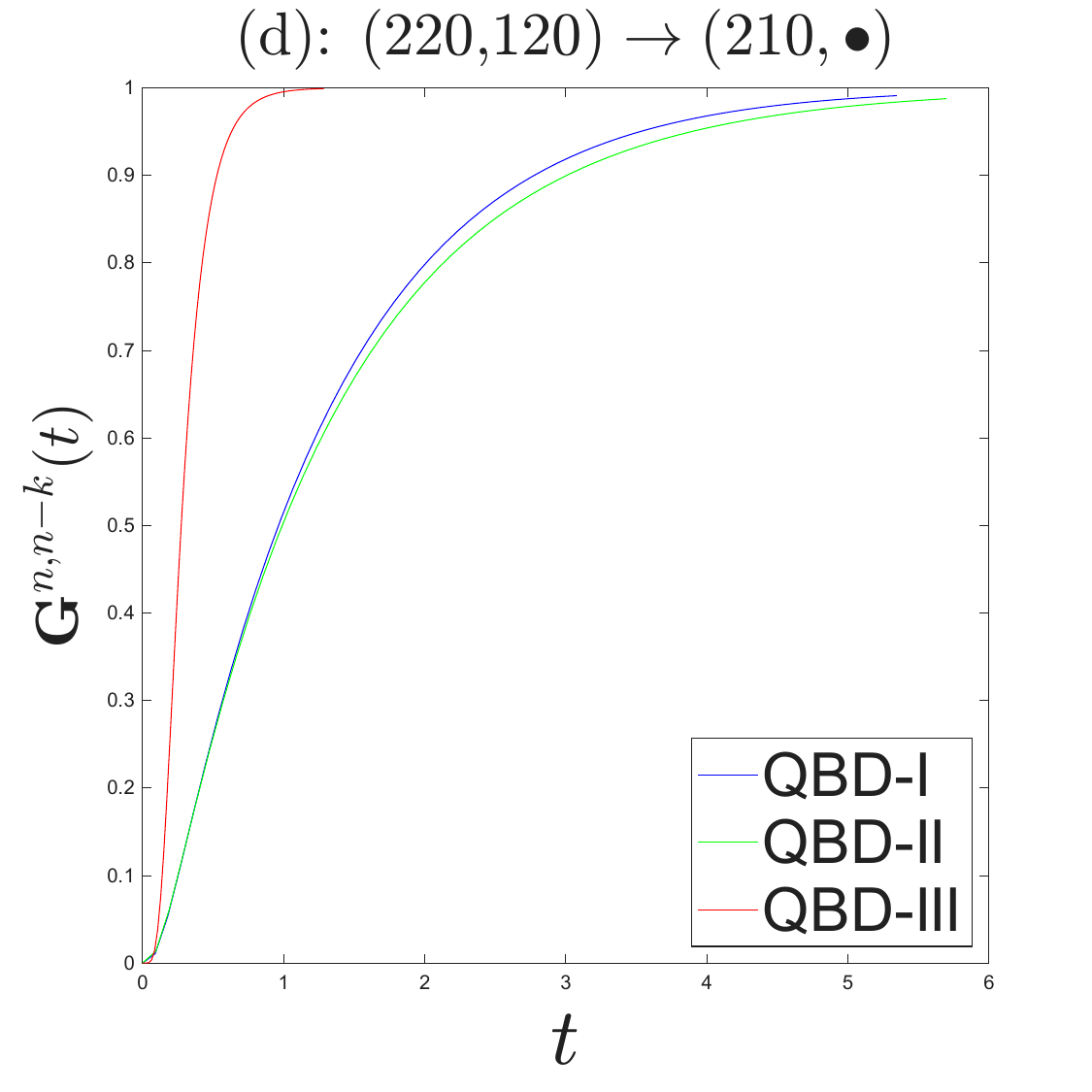}

\includegraphics[width=0.45\textwidth]{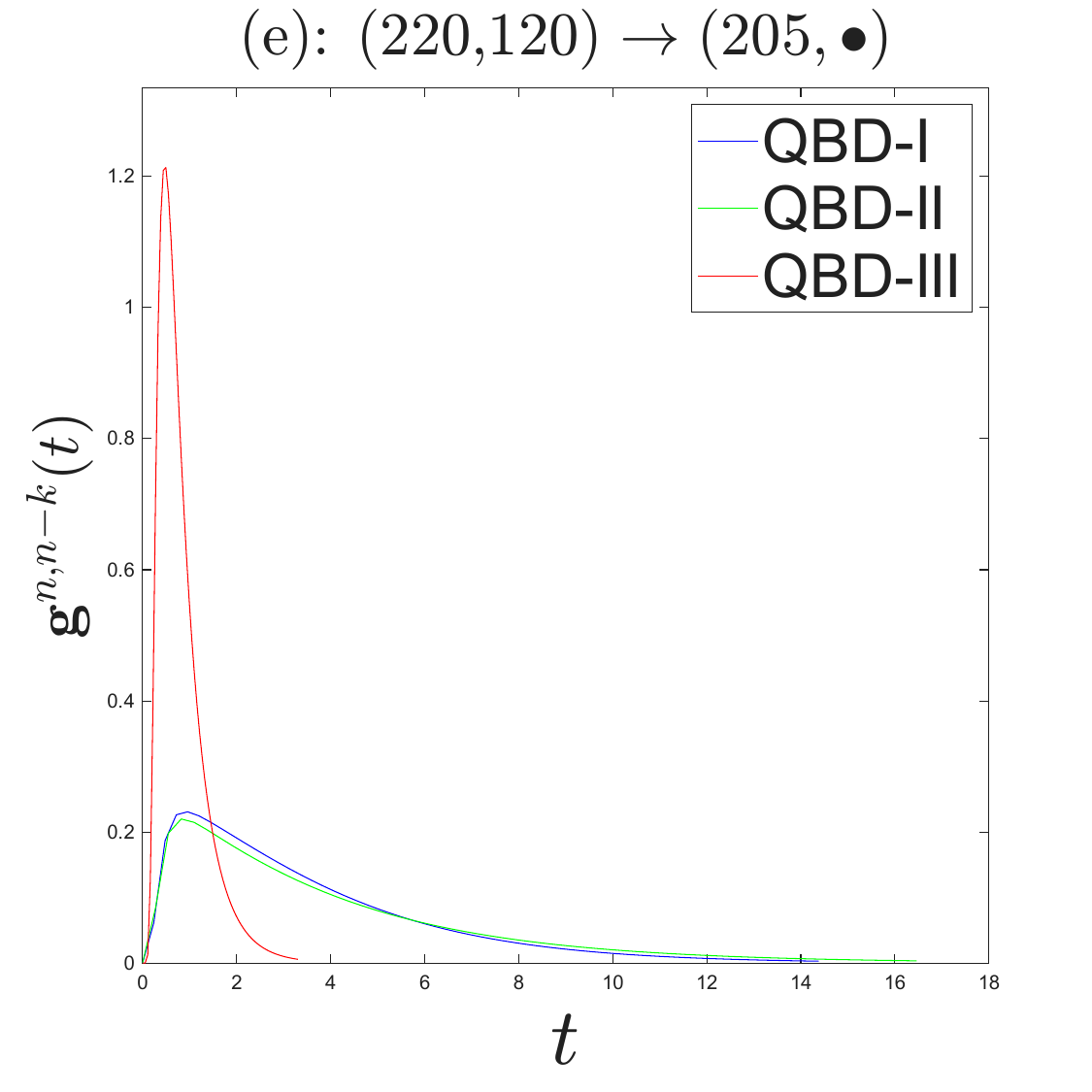}
\hfill
\includegraphics[width=0.45\textwidth]{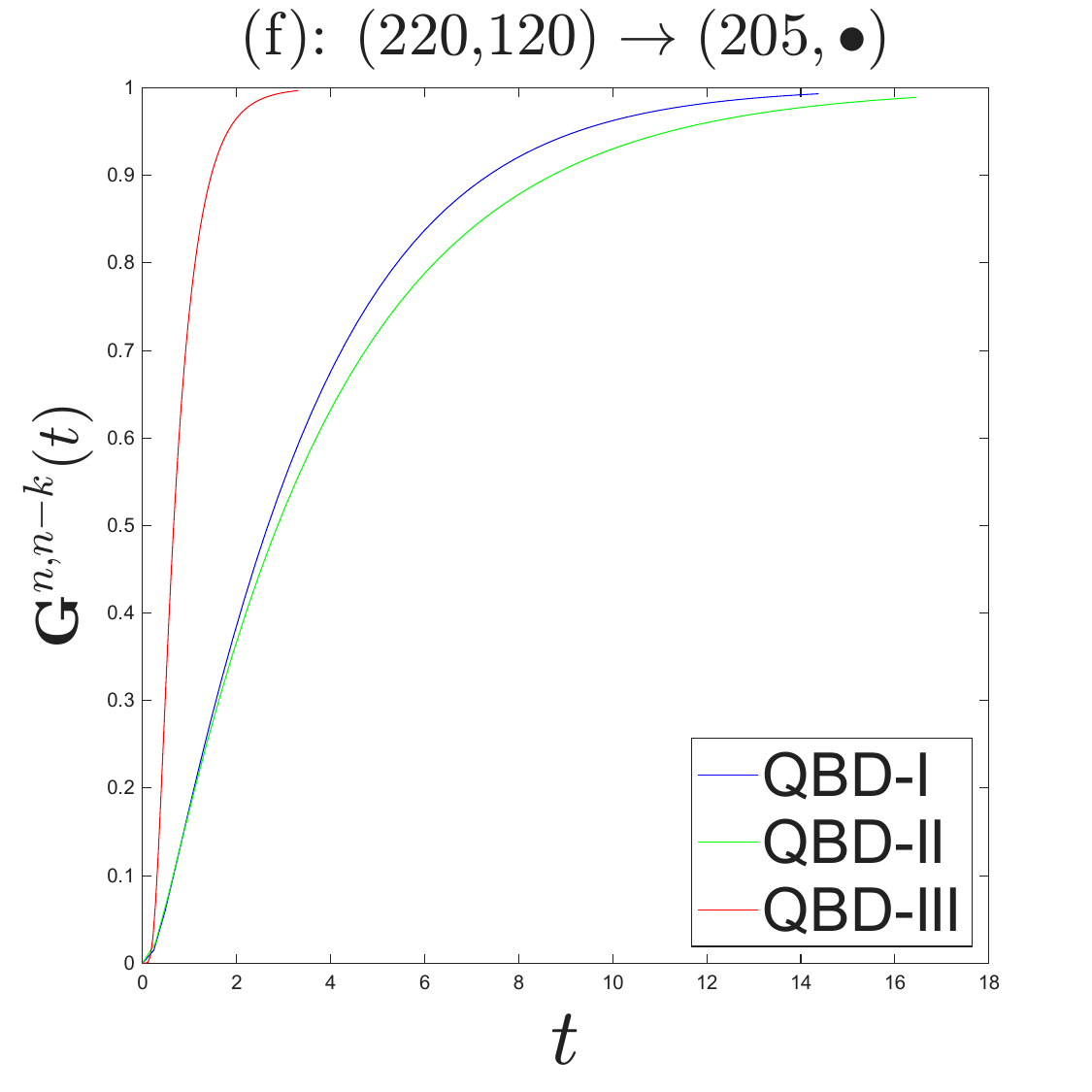}

\caption{Probability density ${\bf g}^{n,n-k}(t)$ and cumulative probability ${\bf G}^{n,n-k}(t)$ of the time until the $k=5,10,15$ beds become empty, given that the system is full $(n=N=220)$ with $i=120$ \emph{complex} patients. Here, $t$ denotes the time in days. We assumed $p_{AA}=0.85$, $p_{BA}=0.15$, $N=220$, $M_B=210$, $p=0.25$, and $\lambda_A$, $\lambda_B$, $\mu_A$, $\mu_B$ as given in Table~\ref{tab:parametersQBDs}.}
\label{fig:first_hitting_times_below}
\end{figure}

\begin{table}[htbp]
\centering
\renewcommand{\arraystretch}{1.2}

\begin{tabular}{lccc}
\hline
Models & $k=5$ & $k=10$ & $k=15$ \\
\hline

QBD-I   & $0.2982$ & $1.3368$ & $3.5944$ \\
QBD-II  & $0.3028$ & $1.4250$ & $4.1165$ \\
QBD-III & $0.1399$ & $0.3212$ & $0.8260$ \\

\hline
\end{tabular}
\caption{Mean time (in days) until $k$ beds become empty, given the system is full $(n=N=220)$ with $i=120$ \emph{complex} patients. We assumed $p_{AA}=0.85$, $p_{BA}=0.15$, $N=220$, $M_B=210$, $p=0.25$, and $\lambda_A$, $\lambda_B$, $\mu_A$, $\mu_B$ as given in Table~\ref{tab:parametersQBDs}.}

\label{tab:mean_time_below}
\end{table}

\subsection{Cost analysis
}\label{sec:CostAnalysis}

We now compute the distribution of the cost accumulated at the times until $k\in\{5,10,15\}$ beds become available, given the initial system is full $(n=N=220)$ with $i=120$ complex patients, using Algorithm~\ref{alg:multiplepair_C_n_nminusk} and the numerical inversion methods in Den Iseger~\cite{DenIseger_2006}. 

First, suppose a complexity-weighted congestion cost per unit time spent in state $(n,i)$ is given by
\begin{eqnarray}
    c(n,i) = c_A i + c_B (n-i),
\end{eqnarray}
where $n$ is the total number of patients, $i$ denotes the number of Type‑A (complex) patients in the system, $c_A$ is the cost accumulated per unit time for each Type‑A (complex) patient in the system, and $c_B$ is the cost accumulated per unit time for each Type‑B (other) patient. Also, we assume that the specified set of levels is $\mathcal{A}=\{N-k+1,\ldots,N-1,N\}$ for $k\in\{5,10,15\}$.

Next, we assume that complex patients on average place higher demands on clinical resources than other patients, and to reflect this, let $c_A=1$ and $c_B=0.4$. The resulting cost distributions for the three models QBD‑I, QBD‑II, and QBD‑III, are presented in Figure~\ref{fig:first_hitting_costs_below}. The first, second, and third rows in Figure~\ref{fig:first_hitting_costs_below} correspond to $k=5$, $k=10$, and $k=15$, respectively. Table~\ref{tab:mean_cost_below} shows the mean costs.

For all $k$, the cost densities ${\bf c}^{n,n-k}(z)$ are concentrated near smaller values of $z$, for QBD-III, indicating that the cost accumulated until $k$ beds become available tends to be lower under QBD-III, compared to QBD-I and QBD-II. In contrast, the densities ${\bf c}^{n,n-k}(z)$ for QBD-I and QBD-II have much wider shape and long right-hand tails, which means that higher cost accumulated until $k$ beds become available is more likely to occur under these models. Further, QBD‑II consistently exhibits marginally heavier tails across all values of $k$, indicating that the cost accumulated under QBD-II can be slightly higher than QBD-I. The difference in the densities ${\bf c}^{n,n-k}(z)$ of QBD-I and QBD-II becomes more prominent as we increase the value of $k$.

A closer inspection of Figure~\ref{fig:first_hitting_costs_below} highlights the magnitude of these differences. For $k=5$, Figure~\ref{fig:first_hitting_costs_below}(a) shows that most of the probability mass for QBD‑III is concentrated below $z=40$, while the corresponding cumulative distribution in Figure~\ref{fig:first_hitting_costs_below}(b) indicates that there is almost $90\%$ probability that no more than $z=40$ units of cost are required until $k=5$ beds become available, under QBD-III. For the QBD-I and QBD-II, this probability is almost $60\%$. For $k=10$, most of the density under QBD‑III lies below $z=100$ (Figure~\ref{fig:first_hitting_costs_below}(c)), and Figure~\ref{fig:first_hitting_costs_below}(d) shows that the probability of incurring at most $z=100$ units of cost until $k=10$ beds become available is close to $95\%$, whereas it remains around $35\%$ under QBD‑I and QBD‑II. Similarly, for $k=15$, Figure~\ref{fig:first_hitting_costs_below}(e) indicates that the density for QBD‑III is largely concentrated below $z=500$, and Figure~\ref{fig:first_hitting_costs_below}(f) shows that the probability of requiring no more than $500$ units of cost until $k=15$ beds become available is close to $1$ under QBD‑III, compared to approximately $60\%$ under QBD‑I and QBD‑II.

These results demonstrate that the guard‑channel threshold policy $(M_B, p)$ not only accelerates the recovery of the capacity following congestion but also substantially reduces the associated cost. On the other hand, policies relying on redirection (QBD-I) or transfers (QBD-II) may leave the system vulnerable to prolonged periods of high cost accumulation. Therefore, reserving a fraction of beds for complex patients might be a suitable decision to mitigate congestion and enable faster and cost-efficient recovery of beds following congestion, while maintaining resilience to unforeseen demand.

\begin{table}[htbp]
\centering
\renewcommand{\arraystretch}{1.2}

\begin{tabular}{lccc}
\hline
Models & $k=5$ & $k=10$ & $k=15$ \\
\hline

QBD-I   & $47.40$ & $209.91$ & $555.65$ \\
QBD-II  & $48.38$ & $226.50$ & $647.79$ \\
QBD-III & $22.34$ & $50.98$  & $129.88$ \\

\hline
\end{tabular}

\caption{Mean units of cost until $k$ beds become empty, given the system is full $(n=N=220)$ with $i=120$ \emph{complex} patients. We assumed $p_{AA}=0.85$, $p_{BA}=0.15$, $N=220$, $M_B=210$, $p=0.25$, and $\lambda_A$, $\lambda_B$, $\mu_A$, $\mu_B$ as given in Table~\ref{tab:parametersQBDs}.}
\label{tab:mean_cost_below}
\end{table}

\begin{figure}[htbp]
\centering

\includegraphics[width=0.45\textwidth]{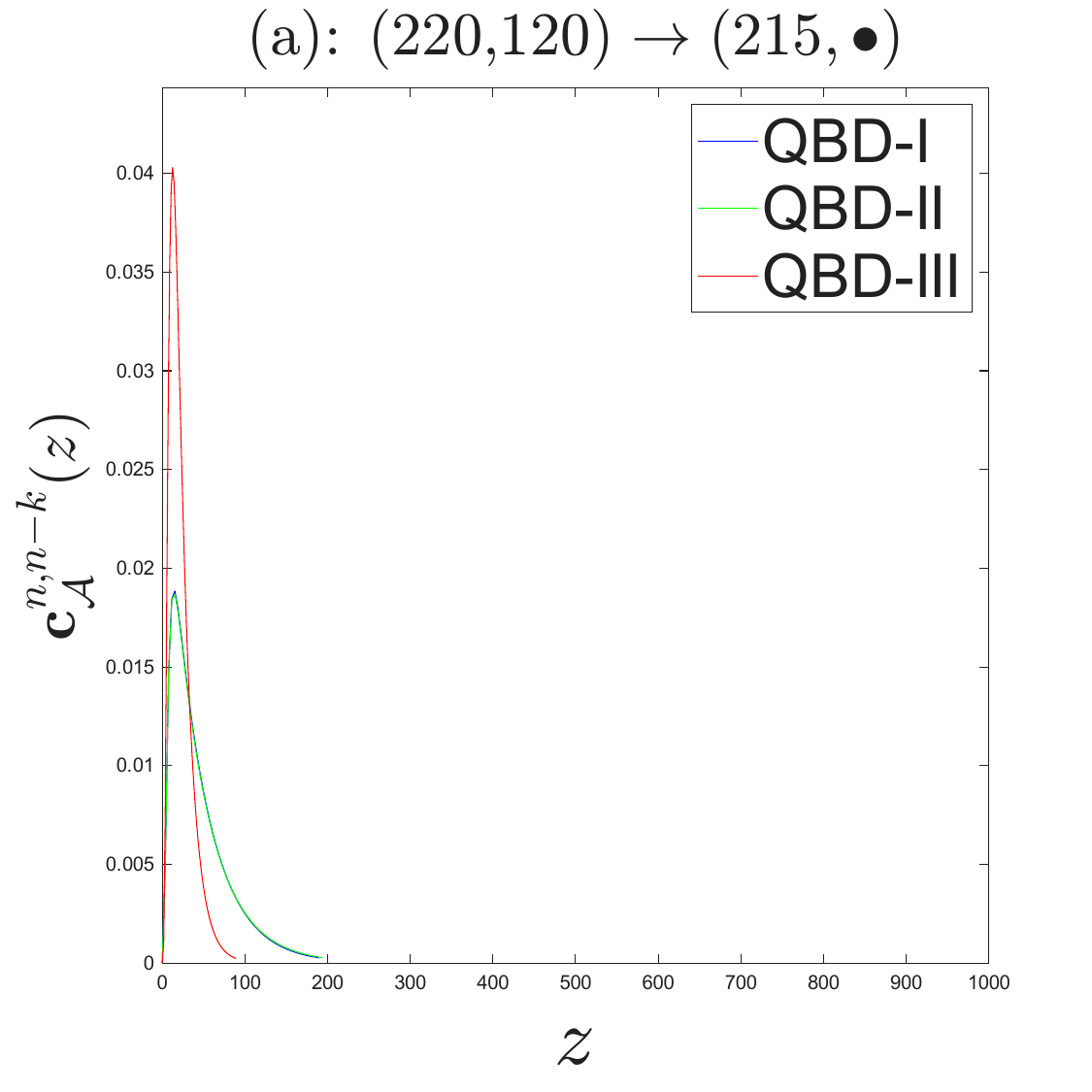}
\hfill
\includegraphics[width=0.45\textwidth]{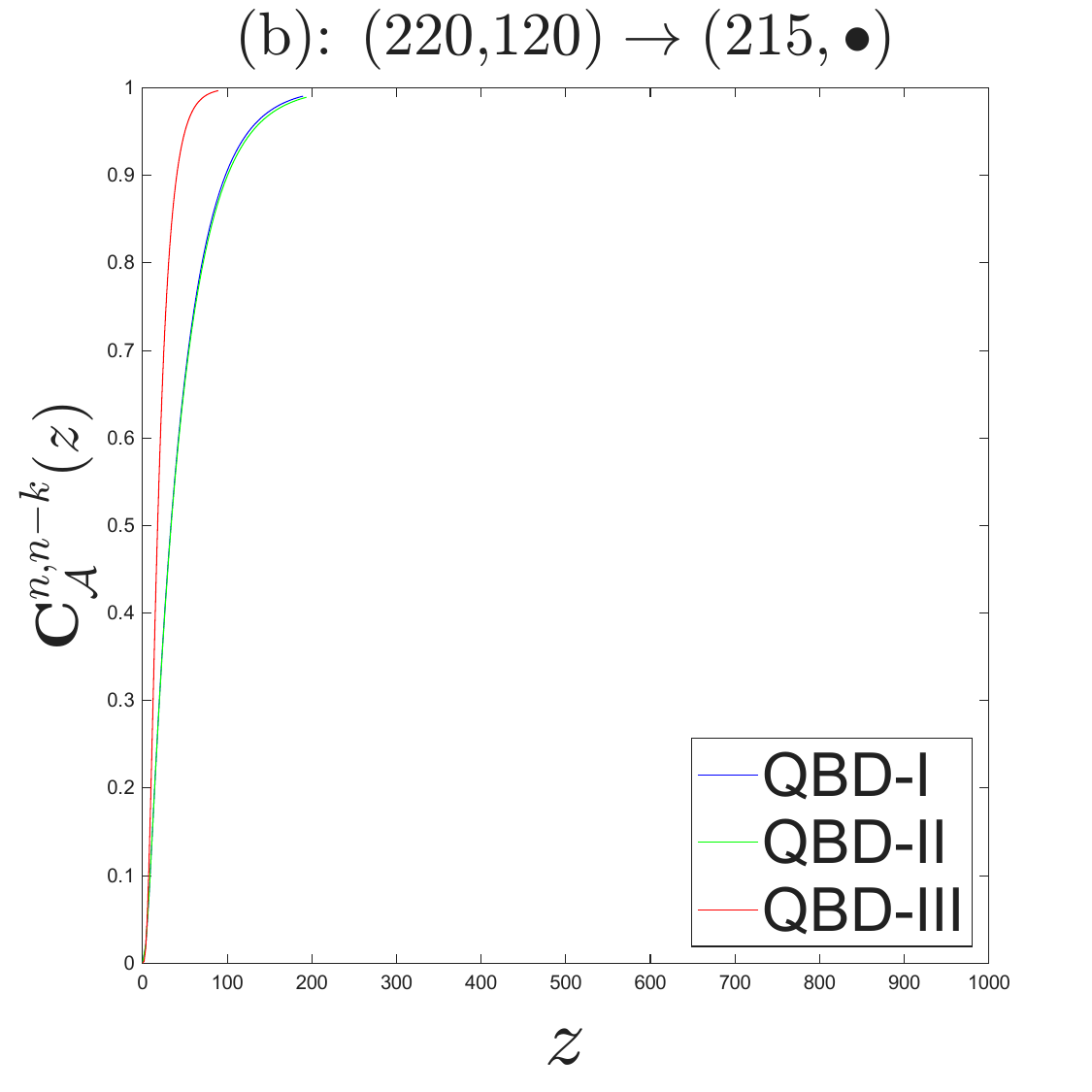}

\includegraphics[width=0.45\textwidth]{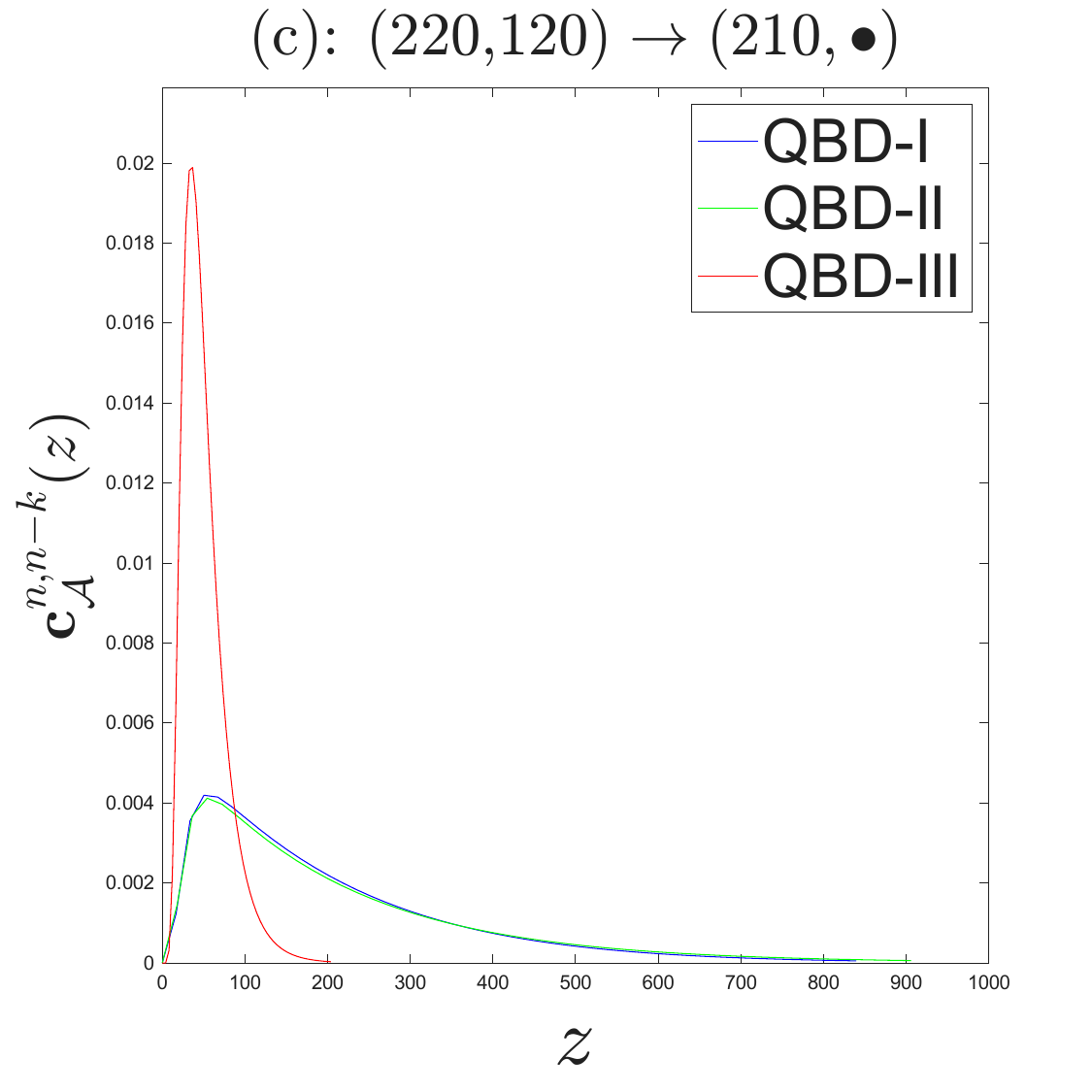}
\hfill
\includegraphics[width=0.45\textwidth]{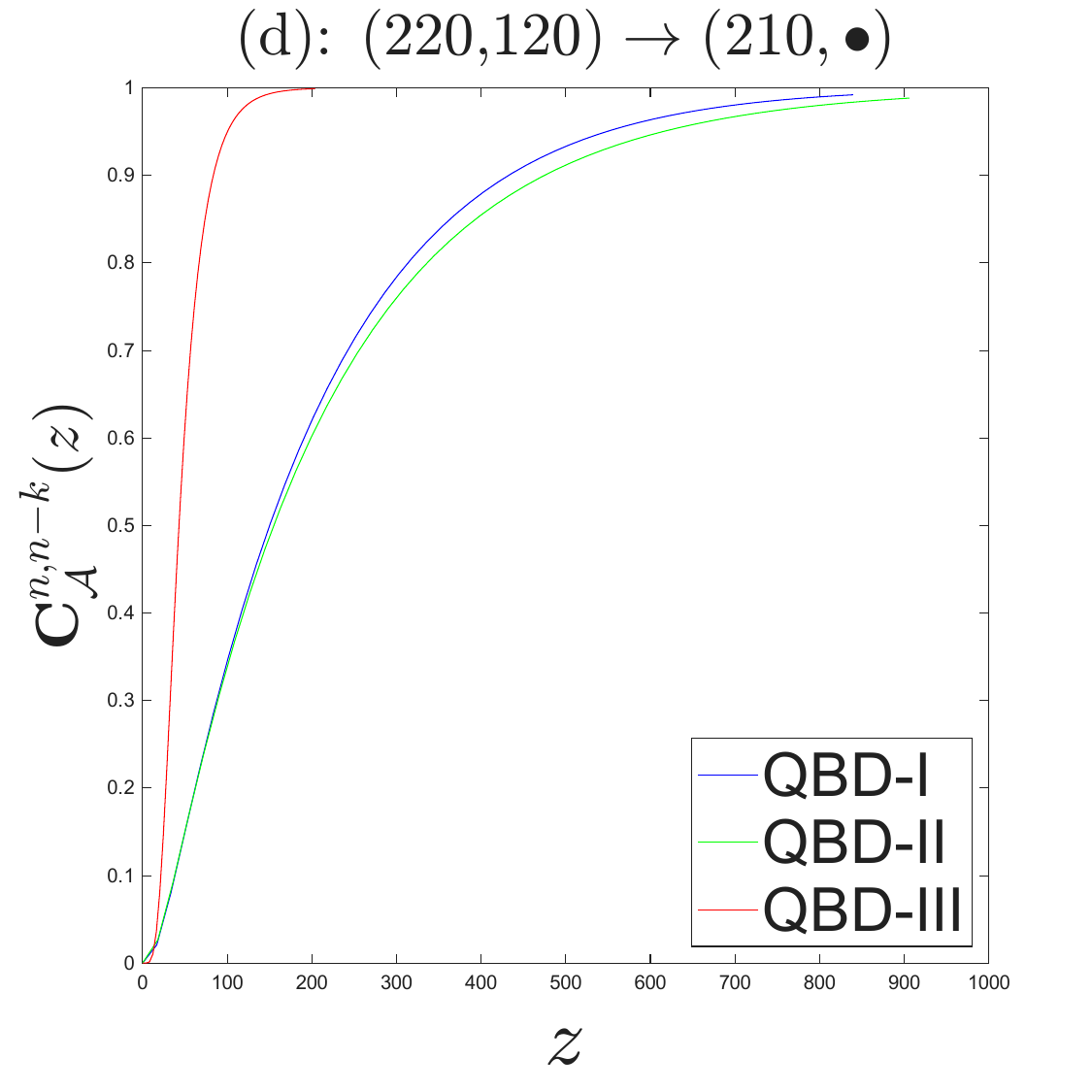}

\includegraphics[width=0.45\textwidth]{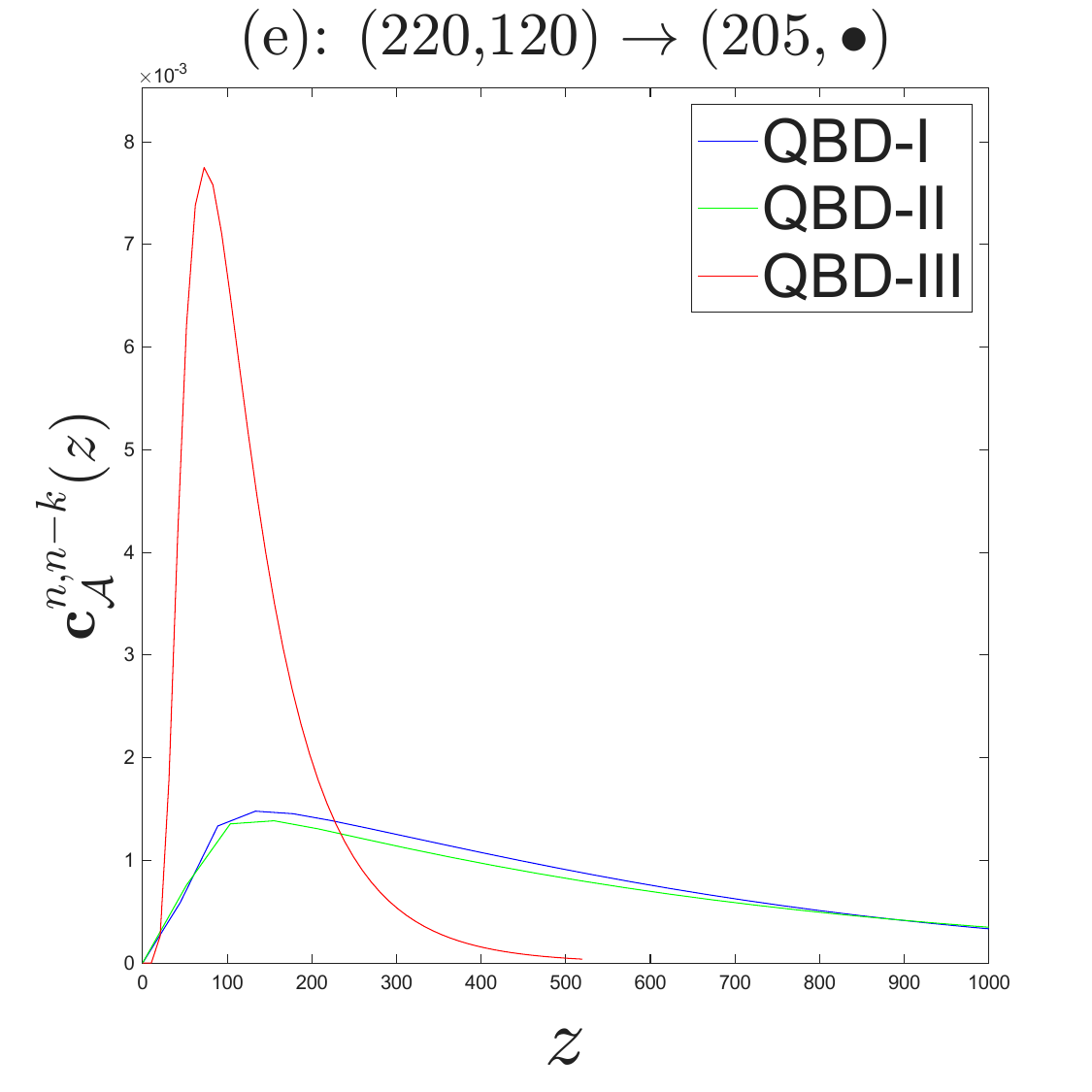}
\hfill
\includegraphics[width=0.45\textwidth]{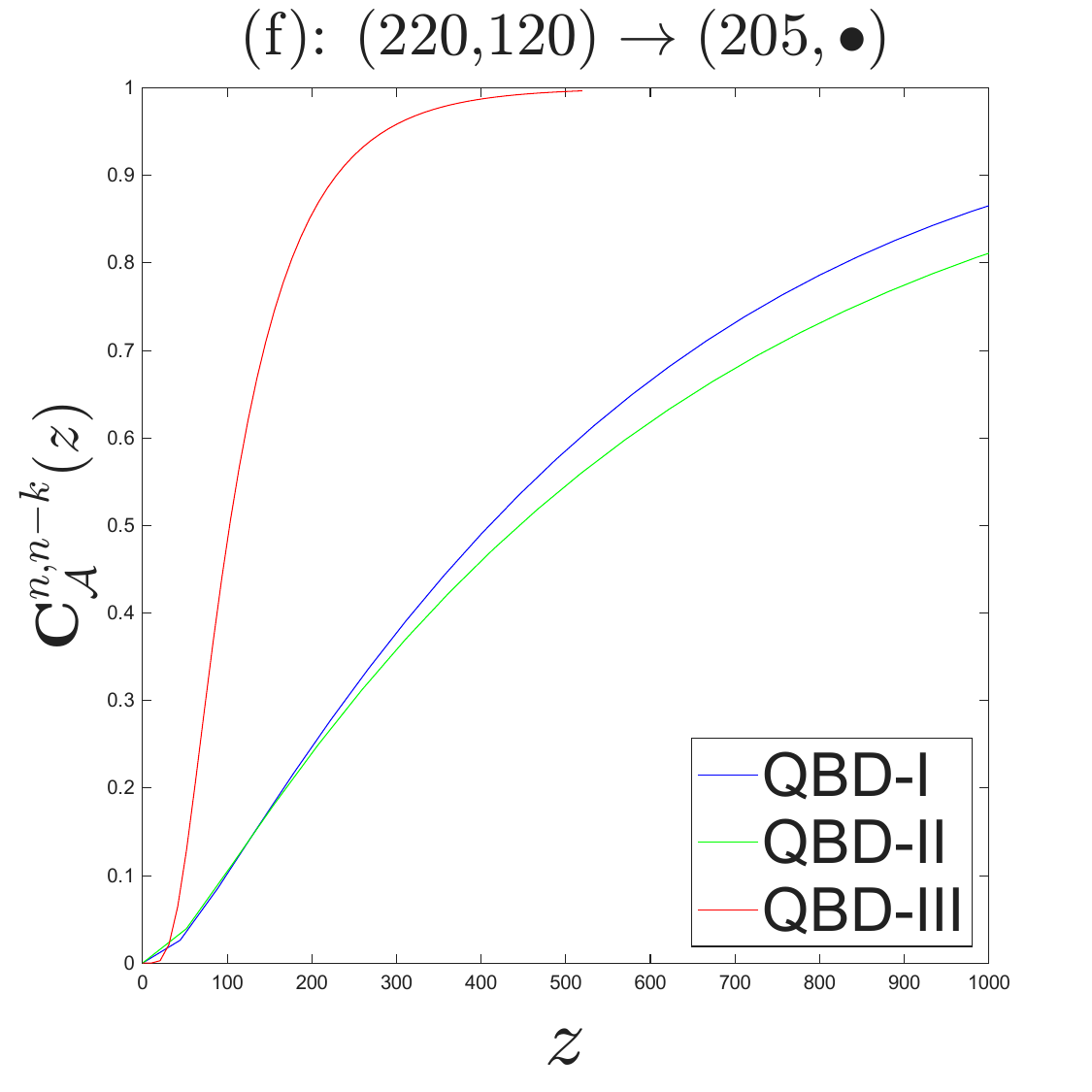}

\caption{Probability density ${\bf c}^{n,n-k}(z)$ and cumulative probability ${\bf C}^{n,n-k}(z)$ of the cost accumulated until the $k=5,10,15$ beds become empty, given that the system is full $(n=N=220)$ with $i=120$ \emph{complex} patients. Here, $z$ denotes the units of the cost accumulated per day. We assumed $p_{AA}=0.85$, $p_{BA}=0.15$, $N=220$, $M_B=210$, $p=0.25$, and $\lambda_A$, $\lambda_B$, $\mu_A$, $\mu_B$ as given in Table~\ref{tab:parametersQBDs}.}
\label{fig:first_hitting_costs_below}
\end{figure}

\subsection{Sensitivity}

We evaluate the sensitivity of the densities ${\bf c}^{n,n-k}_{\mathcal A}(z)$ with respect to the model parameters $\lambda_A$, $\lambda_B$, $\mu_A$, and $\mu_B$, using Algorithm~\ref{alg:C_n_nminusk_der_mem} and the numerical inversion methods in Den Iseger~\cite{DenIseger_2006}. The outputs are presented in Figure~\ref{fig:sensitivity_QBDI_QBDII_QBDIII}. 

We observe that as $\lambda_A$ or $\lambda_B$ increases, the corresponding derivative of ${\bf c}^{n,n-k}_{\mathcal A}(z)$ changes its sign from negative to positive at some point $z^{*}$, and then later, approaches zero as $z$ increases. This indicates that the values of ${\bf c}^{n,n-k}_{\mathcal A}(z)$ get smaller for all $z<z^{*}$ and larger for all $z>z^{*}$ than before, when we increase $\lambda_A$ or $\lambda_B$. That is, there is a higher probability of observing larger values of costs $z$ (since probability of observing some values of $z$ is an area under the density in the corresponding range). The opposite is true for the increasing values of $\mu_A$ or $\mu_B$, as expected.

Indeed, when $\lambda_A$ or $\lambda_B$ increases, more patients are arriving to the system each day. Consequently the times and so the costs until $k$ beds become available following the congestion increases. The corresponding density curves in Figure~\ref{fig:first_hitting_costs_below} would in such case shift to the right, implying a decrease in density for lower cost values $z$, an increase for moderate values, and little or no change for large values. In contrast, when $\mu_A$ or $\mu_B$ increases, more patients are departing from the system each day, which results in shorter times and lower costs to free $k$ beds following congestion. The corresponding density curves in Figure~\ref{fig:first_hitting_costs_below} would in such case shift to the left, indicating an increase in density for lower cost values $z$, a decrease for moderate values, and little or no change for larger values.

These findings highlight that an increase in demand reduces the likelihood of low-cost recovery of beds following congestion. We note that the hospitals operating near-full capacity are highly susceptible to demand fluctuations and may experience significant cost escalation driven by prolonged high-occupancy periods. Therefore, the admission policy regulation and guard-channel thresholds as discussed in our examples, can assist with the effective management of patient flow.

\begin{figure}[htbp]
\centering

\includegraphics[width=0.47\textwidth]{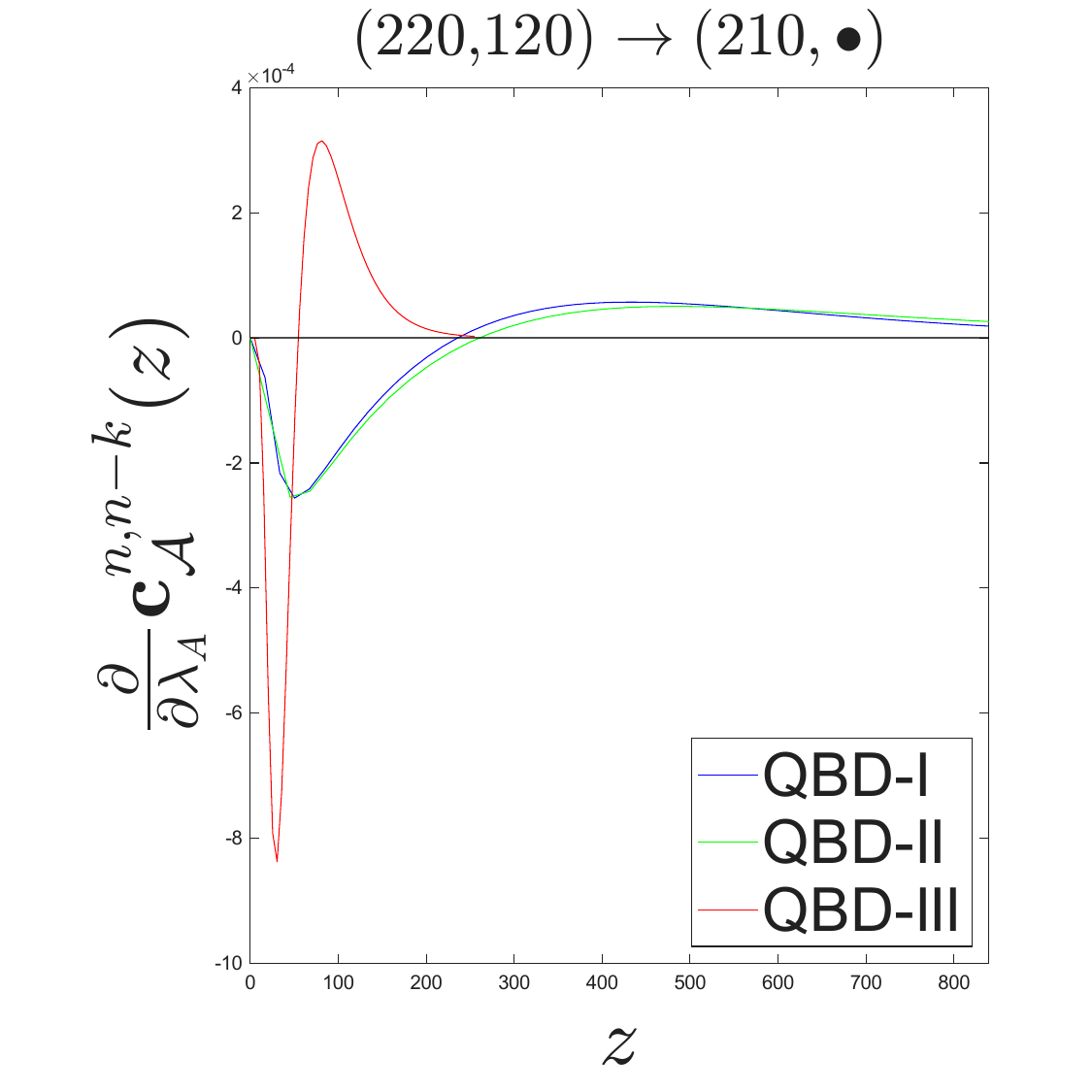}
\hfill
\includegraphics[width=0.47\textwidth]{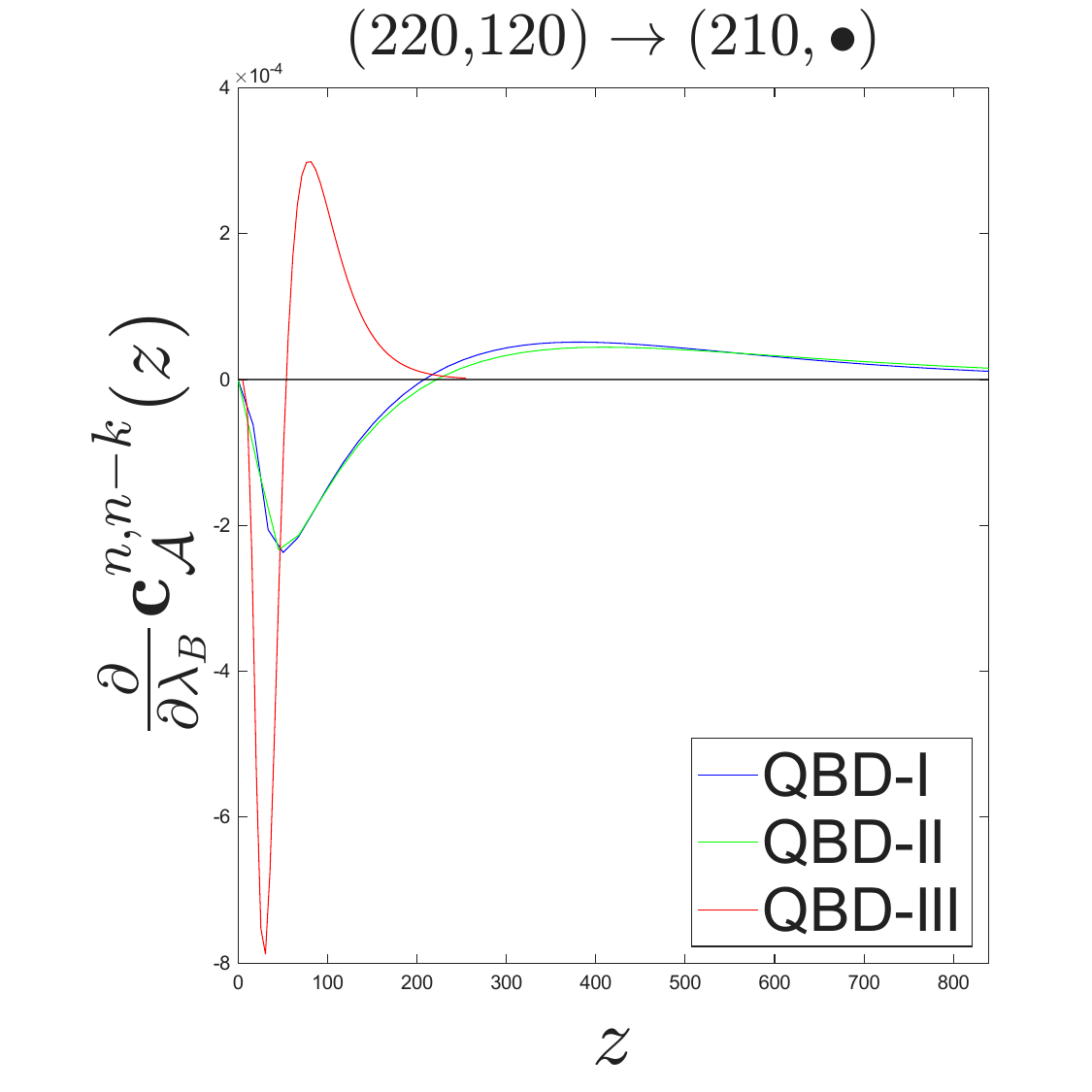}

\includegraphics[width=0.47\textwidth]{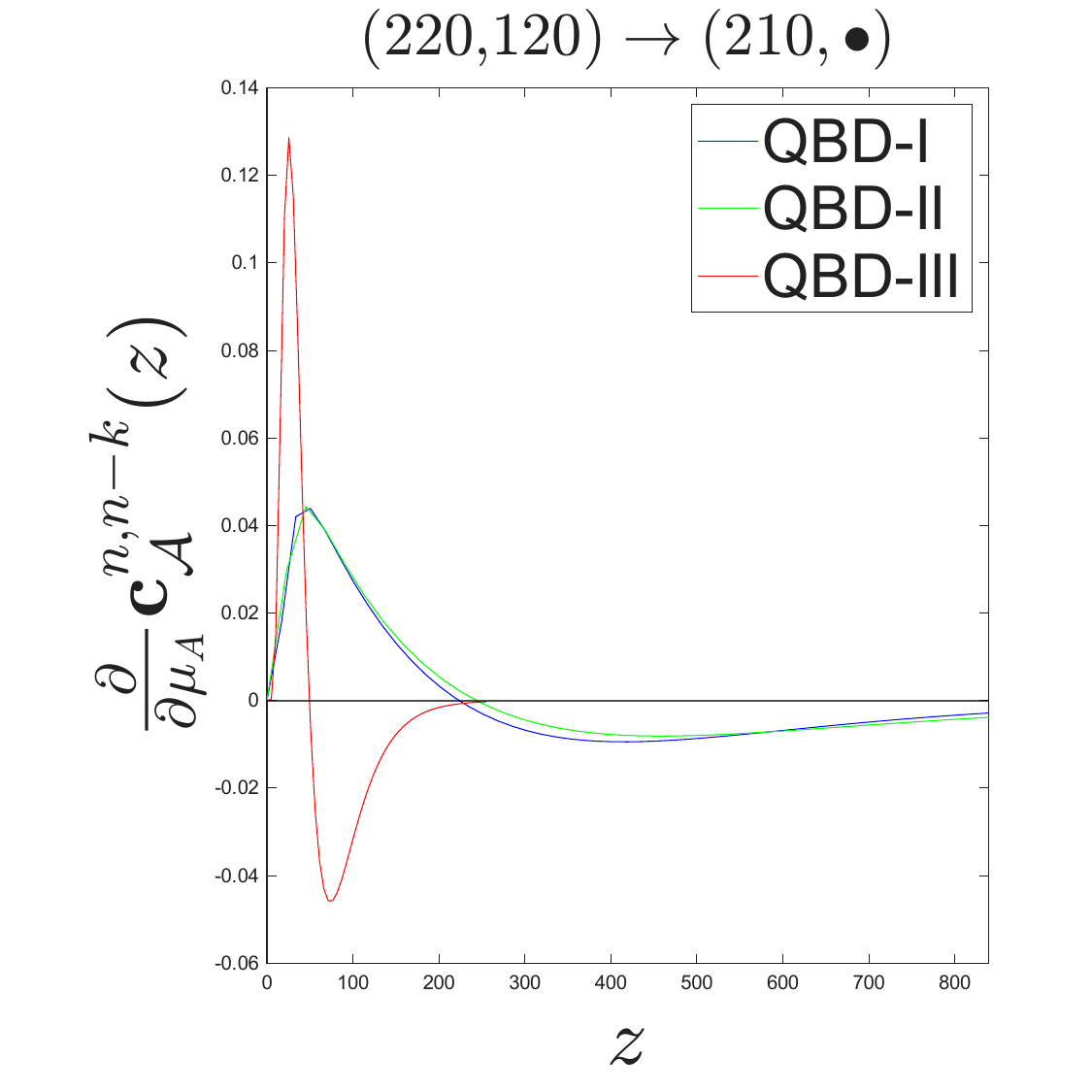}
\hfill
\includegraphics[width=0.47\textwidth]{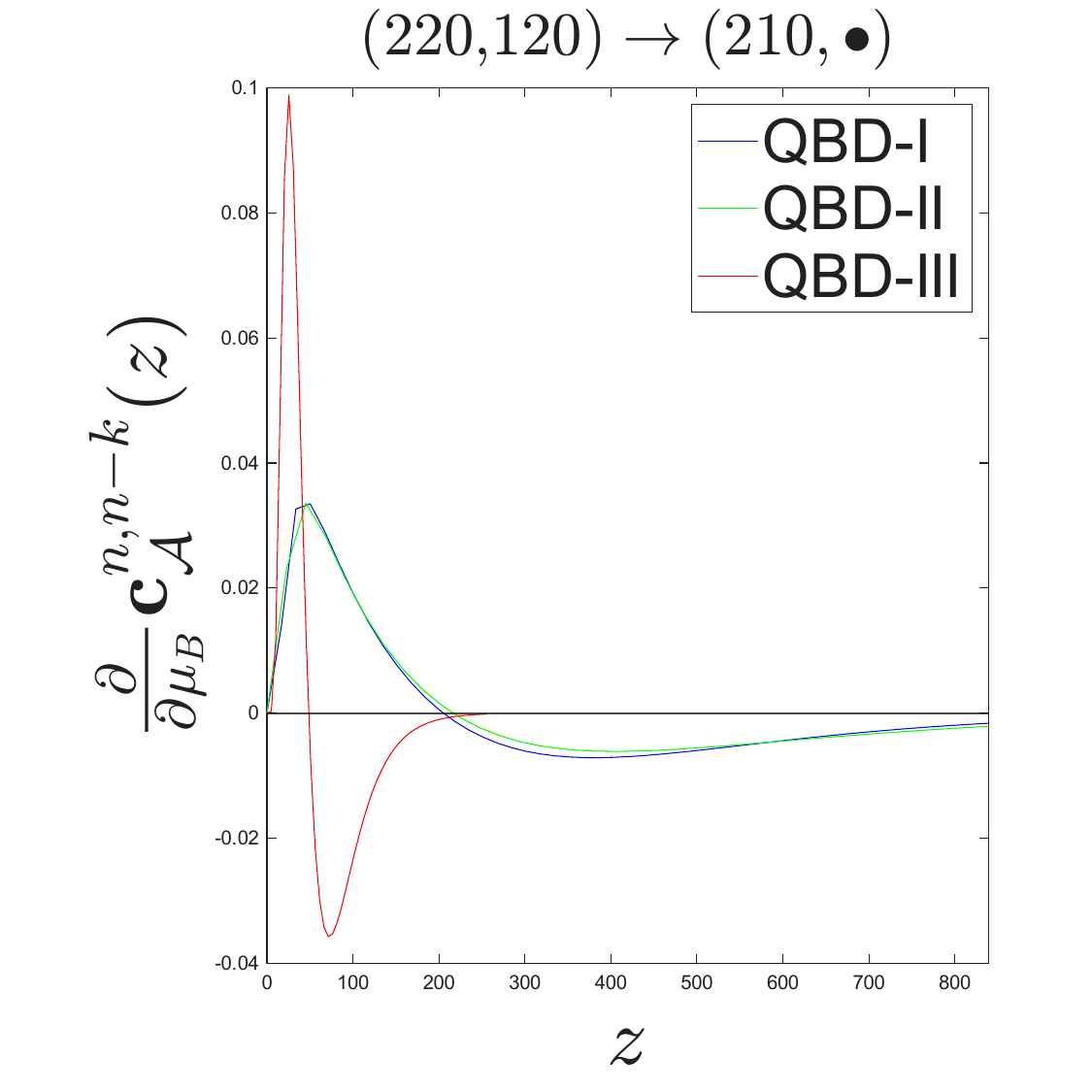}

\caption{Sensitivity of the distribution of cost ${\bf c}^{n,n-k}_{\mathcal A}(z)$ accumulated until $k=10$ beds become available, given the system is full $(n=N=220)$ with $i=120$ complex patients, with respect to model parameters $\lambda_A$, $\lambda_B$, $\mu_A$, and $\mu_B$ under the model QBD-I. Here, $z$ denotes the units of the cost accumulated per day. We assumed $p_{AA}=0.85$, $p_{BA}=0.15$, $N=220$, $M_B=210$, $p=0.25$, and $\lambda_A$, $\lambda_B$, $\mu_A$, $\mu_B$ as given in Table~\ref{tab:parametersQBDs}.}
\label{fig:sensitivity_QBDI_QBDII_QBDIII}
\end{figure}

\section{Conclusions}\label{sec:Conclusion}

We considered level-dependent quasi-birth-and-death (LD-QBD) process and its application to the cost analysis of multi-server systems. We derived analytical expressions and recursive algorithms to compute the Laplace–Stieltjes transforms (LSTs) of the distribution of costs accumulated during the time the system spends within specified occupancy thresholds. We also developed results and algorithms for the sensitivity analysis of the LSTs of the distribution of costs. We discussed algorithmic complexity and gave memory efficient versions of the proposed algorithms. 

The results presented here have potential application in many service systems, such as healthcare, as discussed throughout the paper. We illustrated the application potential of our results in multi-server systems through numerical examples with parameters based on data obtained from a tertiary referral hospital in Australia. In our examples, we gave insights useful in hospital management.

In our future work we aim to extend our results to time-inhomogeneous QBD processes, where transition rates vary with time to model demand fluctuation due to seasonal effects. Next, an interesting direction would be to apply phase-type distributions for customers' service time together with multiple customer types, enabling a more versatile representation of real-world systems. The current work also has the potential to be extended to QBD processes with jumps to model catastrophes.

\section{Acknowledgements}
\label{Acknowledgements}

\noindent{\bf Data}\\
\noindent Data used in the paper was obtained following ethical approval from the Tasmanian Health and Medical Human Research Ethics Committee (HREC No 23633) and site-specific approval from the Research Governance Office of the Tasmanian Health Service.

\bigskip\noindent{\bf Authorship contribution statement}\\
\noindent The following are the contributions of the authors, M. Abdullah Khokhar (MAK), Ma\l gorzata M. O'Reilly (MMO), and Richard Turner (RT): 
\begin{itemize}
	\item	Conceptualisation - mathematical background: MAK and MMO;
	\item  
	Conceptualisation - clinical background: MAK and RT;
 \item Methodology - model development: MAK, MMO, and RT;
  \item Formal analysis - proof of theorems: MAK and MMO;
 \item Methodology - algorithms and code: MAK and MMO;
\item Investigation - numerical analysis: MAK;
	\item Supervision: MMO and RT;
\item Writing - original draft: MAK and MMO;
\item Writing - review \& editing: MAK, MMO, and RT.
\end{itemize}

\bigskip\noindent{\bf Declaration of competing interests}\\
\noindent The authors have no competing interests to declare that are relevant to the content of this article.

\bigskip\noindent{\bf Funding}\\
\noindent This research did not receive any specific grant from funding agencies in the public, commercial, or not-for-profit sectors.

\begin{appendices}

\section{}\label{AppendixA}
\renewcommand{\theequation}{\ref{AppendixA}.\arabic{equation}}
\setcounter{equation}{0}

\begin{algorithm}[H]
\caption{Evaluate $\{\bpi_n\}_{n=0,1,\ldots,N}$ \hfill(adapted from Aksamit et al.~\cite{aksamit2024sensitivities})}
\label{stationary_algorithm}
\begin{algorithmic}[1]

\Input ${\bf Q}$
\Output $\{\bpi_n\}_{n=0,1,\ldots,N}$

\State Compute $\widehat{\bf R}^{(0)} = -{\bf Q}^{[1,0]}({\bf Q}^{[0,0]})^{-1}$

\For{$n = 1, \ldots, N-1$}
\State Compute $\widehat{\bf R}^{(n)} =
-{\bf Q}^{[n+1,n]}
(\widehat{\bf R}^{(n-1)}{\bf Q}^{[n-1,n]}+{\bf Q}^{[n,n]})^{-1}$
\EndFor

\State Compute $\bpi_N$ by solving
\begin{equation}
\begin{cases}
\bpi_N(\widehat{\bf R}^{(N-1)}{\bf Q}^{[N-1,N]}+{\bf Q}^{[N,N]}) = {\bf 0}, \\
\bpi_N\Big({\bf 1}+\sum_{n=0}^{N-1}\prod_{k=N-1}^{n}\widehat{\bf R}^{(k)}{\bf 1}\Big) = 1.
\end{cases}
\label{eq:piN}
\end{equation}

\For{$n = N-1, \ldots, 0$}
\State Compute
\[
\bpi_n = \bpi_{n+1}\widehat{\bf R}^{(n)}
\]
\EndFor

\end{algorithmic}
\end{algorithm}

\begin{algorithm}[H]
\caption{Evaluate $\widetilde{\bf G}^{n,n-k}(s)$ \hfill(adapted from Aksamit et al.~\cite{aksamit2024sensitivities})}
\label{Gs_algorithm}
\begin{algorithmic}[1]

\Input ${\bf Q}, n, k$
\Output $\widetilde{\bf G}^{n,n-k}(s)$

\State Compute
\[
\widetilde{\bf G}^{N,N-1}(s) =
-({\bf Q}^{[N,N]} - s{\bf I})^{-1}{\bf Q}^{[N,N-1]}
\]

\For{$i = N-1, \ldots, n-k+1$}
\State Compute $\widetilde{\bf G}^{i,i-1}(s) =
-\Big({\bf Q}^{[i,i]} - s{\bf I}
+ {\bf Q}^{[i,i+1]}\widetilde{\bf G}^{i+1,i}(s)\Big)^{-1}
{\bf Q}^{[i,i-1]}$
\EndFor

\State Compute
\[
\widetilde{\bf G}^{n,n-k}(s) =
\widetilde{\bf G}^{n,n-1}(s)
\widetilde{\bf G}^{n-1,n-2}(s)
\cdots
\widetilde{\bf G}^{n-k+1,n-k}(s)
\]

\end{algorithmic}
\end{algorithm}

\begin{algorithm}[H]
\caption{Evaluate $\widetilde{\bf H}^{n,n+k}(s)$ \hfill (adapted from Aksamit et al.~\cite{aksamit2024sensitivities})}
\label{Hs_algorithm}
\begin{algorithmic}[1]

\Input ${\bf Q}, n, k$
\Output $\widetilde{\bf H}^{n,n+k}(s)$

\State Compute
\[
\widetilde{\bf H}^{0,1}(s) =
-({\bf Q}^{[0,0]} - s{\bf I})^{-1}{\bf Q}^{[0,1]}
\]

\For{$i = 1, \ldots, n+k-1$}
\State Compute $\widetilde{\bf H}^{i,i+1}(s) =
-\Big({\bf Q}^{[i,i]} - s{\bf I}
+ {\bf Q}^{[i,i-1]}\widetilde{\bf H}^{i-1,i}(s)\Big)^{-1}
{\bf Q}^{[i,i+1]}$
\EndFor

\State Compute
\[
\widetilde{\bf H}^{n,n+k}(s) =
\widetilde{\bf H}^{n,n+1}(s)
\widetilde{\bf H}^{n+1,n+2}(s)
\cdots
\widetilde{\bf H}^{n+k-1,n+k}(s)
\]

\end{algorithmic}
\end{algorithm}

\section{} \label{eq:equations}
\renewcommand{\theequation}{\ref{eq:equations}.\arabic{equation}}
\setcounter{equation}{0}

\subsection{Equation for Algorithm~\ref{alg:C_n_nminusk_der_mem}}

\begin{eqnarray}
\lefteqn{\frac{\partial }{\partial \btheta}\widetilde{\bf C}_{\mathcal{A}}^{n,n-1}(s;\btheta)}
\nonumber
\\
&=&
\left({\bf Q}^{[n,n]}(\btheta)-s{\bf C}_n\times I(n\in\mathcal{A})
+{\bf Q}^{[n,n+1]}(\btheta)\times{\bf B}
\right)^{-1}
\nonumber\\
&&
\times
\left(\frac{\partial {\bf Q}^{[n,n]}(\btheta)}{\partial \btheta}
+
\frac{\partial{\bf Q}^{[n,n+1]}(\btheta)}{\partial \btheta} \times \left({\bf I}_k
\otimes \widetilde{\bf C}_{\mathcal{A}}^{n+1,n}(s,\btheta)\right) + {\bf Q}^{[n,n+1]}(\btheta) \times {\bf D} \right)
\nonumber\\
&&
\times
\left(
{\bf I}_k
\otimes \left({\bf Q}^{[n,n]}(\btheta)-s{\bf C}_n\times I(n\in\mathcal{A})
+{\bf Q}^{[n,n+1]}(\btheta)\times{\bf B}
\right)^{-1}
\right)
\left(
{\bf I}_k\otimes {\bf Q}^{[n,n-1]}(\btheta)
\right)
\nonumber
 \\
&&
	-
	\left({\bf Q}^{[n,n]}(\btheta)-s{\bf C}_n
 \times I(n\in\mathcal{A})
+{\bf Q}^{[n,n+1]}(\btheta)\times{\bf B}
\right)^{-1}
	\times
	\frac{\partial {\bf Q}^{[n,n-1]}(\btheta)}{\partial \btheta}. 
\label{eq:alg_Cnnminusk_mem}
\end{eqnarray}

\subsection{Equation for Algorithm~\ref{alg:C_n_nplusk_der_mem}}

\begin{eqnarray}
\lefteqn{\frac{\partial }{\partial \btheta}\widetilde{\bf C}_{\mathcal{A}}^{n,n+1}(s;\btheta)}
\nonumber
\\
&=&
\left({\bf Q}^{[n,n]}(\btheta)-s{\bf C}_n\times I(n\in\mathcal{A})
+{\bf Q}^{[n,n-1]}(\btheta) \times {\bf B}
\right)^{-1}
\nonumber\\
&&
\times
\left(\frac{\partial {\bf Q}^{[n,n]}(\btheta)}{\partial \btheta}
+
\frac{\partial{\bf Q}^{[n,n-1]}(\btheta)}{\partial \btheta} \times \left({\bf I}_k
\otimes {\bf B}\right) + {\bf Q}^{[n,n-1]}(\btheta) \times {\bf D} \right)
\nonumber\\
&&
\times
\left(
{\bf I}_k
\otimes \left({\bf Q}^{[n,n]}(\btheta)-s{\bf C}_n\times I(n\in\mathcal{A})
+{\bf Q}^{[n,n-1]}(\btheta)\times {\bf B}
\right)^{-1}
\right)
\left(
{\bf I}_k\otimes {\bf Q}^{[n,n+1]}(\btheta)
\right)
\nonumber
 \\
&&
	-
	\left({\bf Q}^{[n,n]}(\btheta)-s{\bf C}_n
 \times I(n\in\mathcal{A})
+{\bf Q}^{[n,n-1]}(\btheta)\times {\bf B}
\right)^{-1}
	\frac{\partial {\bf Q}^{[n,n+1]}(\btheta)}{\partial \btheta}. 
\label{eq:alg_Cnplusk_mem}
\end{eqnarray}

\end{appendices}

\bibliography{references}

\end{document}